\documentclass[12pt]{amsart}

\usepackage{amsfonts, amstext, amsmath, amsthm, amscd, amssymb, enumitem, url}
\usepackage{tikz-cd}
\usepackage{xcolor}
\usepackage{svg}
\usepackage{pdfpages}
\usepackage{pdflscape}

\usepackage[parfill]{parskip}
\usepackage{microtype}

\usepackage[colorlinks,citecolor=cyan,linkcolor=magenta]{hyperref}
\usepackage[nameinlink]{cleveref}

\newtheorem{thm}{Theorem}[section]

\newtheorem{lemma}[thm]{Lemma}
\newtheorem{prop}[thm]{Proposition}

\newtheorem{conj}[thm]{Conjecture}

\newtheorem*{prop:noperfectfitexists}{\Cref{prop:noperfectfitexists}}
\newtheorem*{thm:cuspedBirkhoffsection}{\Cref{thm:cuspedBirkhoffsection}}
\newtheorem*{thm:closedBirkhoffsection}{\Cref{thm:closedBirkhoffsection}}

\theoremstyle{definition}
\newtheorem{defn}[thm]{Definition}
\newtheorem{rmk}[thm]{Remark}
\newtheorem{constr}[thm]{Construction}
\newtheorem{quest}[thm]{Question}

\numberwithin{equation}{section}

\renewcommand{\epsilon}{\varepsilon}

\usepackage{stmaryrd}
\newcommand{\cut}{\!\bbslash\!}
\newcommand{\ind}{\mathrm{ind}}
\renewcommand{\top}{\mathrm{top}}
\newcommand{\red}{\mathrm{red}}

\begin{document}

\title[Constructing Birkhoff sections]{Constructing Birkhoff sections for pseudo-Anosov flows with controlled complexity}

\author{Chi Cheuk Tsang}
\address{University of California, Berkeley \\
    970 Evans Hall \#3840 \\
    Berkeley, CA 94720-3840}
\email{chicheuk@math.berkeley.edu}
\thanks{Chi Cheuk Tsang was partially supported by a grant from the Simons Foundation \#376200.}

\maketitle

\begin{abstract}
We introduce a new method of constructing Birkhoff sections for pseudo-Anosov flows, which uses the connection between pseudo-Anosov flows and veering triangulations. This method allows for explicit constructions, as well as control over the Birkhoff section in terms of its Euler characteristic and the complexity of the boundary orbits. In particular, we show that any transitive pseudo-Anosov flow has a Birkhoff section with two boundary components.
\end{abstract}

\section{Introduction} \label{sec:intro}

Birkhoff sections are a classical tool for studying flows on 3-manifolds, appearing back in the early $20^{\text{th}}$ century in work of Poincare and Birkhoff. They can be used to reduce questions about the dynamics of 3-dimensional flows to dynamics of surface homeomorphisms. See \cite{Fra92} and \cite{HWZ98} for two classical applications. The question of when a flow has a Birkhoff section, especially when the flow is a Reeb flow, is still a popular research topic. See \cite{CM21} and \cite{CDHR22} for some recent progress.

One class of flows which is long known to admit Birkhoff sections are Anosov flows, which were introduced by D. V. Anosov back in the 1960s. This result is due to work of Fried in \cite{Fri83}. The class of Anosov flows was later expanded to the class of pseudo-Anosov flows by Thurston for wider applicability in the study of 3-manifold topology, and Brunella generalized Fried's result to this larger class of flows in \cite{Bru95}. 

\begin{thm}[{\cite{Fri83}, \cite{Bru95}}] \label{thm:classicexistence} 
Let $\phi$ be a transitive pseudo-Anosov flow on a closed 3-manifold. Then $\phi$ has a Birkhoff section.
\end{thm}


One unsatisfying feature of Fried's and Brunella's proofs of \Cref{thm:classicexistence}, however, is that they are compactness arguments: Roughly, the proofs go by constructing small local transverse surfaces, using compactness to argue that finitely many of these cover up the flow, then piecing them together. As a result, one has no control over how complicated the Birkhoff section is.

In this paper, we present a new method of constructing Birkhoff sections for pseudo-Anosov flows which does allow for control over the complexity. This method uses the recent technology of veering triangulations and their connection with pseudo-Anosov flows: Roughly, given a pseudo-Anosov flow on an oriented closed 3-manifold, there exists a veering triangulation whose faces are positively transverse to the flow.

The idea is to construct transverse surfaces to the flow using the combinatorics of the triangulation. By piecing these surfaces together and resolving the self-intersections, we get a Birkhoff section to the flow. The control over the complexity comes from the fact that the combinatorics of the triangulation can be explicitly described, which allows for control over the constructed transverse surfaces. Our main result is that it is always possible to arrange for the final Birkhoff section to have two boundary components.

\begin{thm} \label{thm:twoboundarycomponents}
Let $\phi$ be a transitive pseudo-Anosov flow on a closed 3-manifold. Then $\phi$ has a Birkhoff section with two boundary components, where each boundary component is embedded along a closed orbit of $\phi$.
\end{thm}

\Cref{thm:twoboundarycomponents} follows from the more technical \Cref{thm:closedBirkhoffsection} which in addition gives explicit bounds on the Euler characteristic of the Birkhoff section as well as the complexity of its boundary components. 

For the rest of this introduction, we discuss the context of \Cref{thm:twoboundarycomponents} within the literature of pseudo-Anosov flows, describe some ideas in the proof of \Cref{thm:twoboundarycomponents}, and provide an outline of the paper.

\subsection{Context for `two boundary components'}

In \cite[P.57]{Thu97}, Thurston asks for a description of `the minimal collections of orbits that need to be removed for the [pseudo-Anosov] flow to admit a section'. \Cref{thm:twoboundarycomponents} can be seen as some progress towards this: For a general pseudo-Anosov flow, a minimal collection with the least number of elements will have two orbits. In fact, \Cref{thm:twoboundarycomponents} gives slightly more since it asserts that one can find a Birkhoff section with one boundary component embedded along each of two closed orbits.

In the case of Anosov flows, Marty recently showed that an Anosov flow admits a Birkhoff section with only one boundary component if and only if it is skew $\mathbb{R}$-covered (\cite[Theorem G]{Mar21}, \cite[Theorem E]{Mar23}). Together with \Cref{thm:twoboundarycomponents}, this provides the following neat trichotomy for Anosov flows. 

\begin{table}[h]
    \centering
    \renewcommand{\arraystretch}{1.2}
    \begin{tabular}{|c|c|}
        \hline
        Orbit space & Birkhoff section \\
        \hline
        Trivial $\mathbb{R}$-covered & Closed \\
        Skew $\mathbb{R}$-covered & One boundary component \\
        Non-$\mathbb{R}$-covered & Two boundary components \\
        \hline
    \end{tabular}
\end{table}

See also \cite[Table 1]{Mar23}.

In particular, since non-$\mathbb{R}$-covered Anosov flows are known to exist in plenty (see, for example \cite{BI20}), this shows that \Cref{thm:twoboundarycomponents} is sharp. This also gives an exact classification of Anosov flows for which \Cref{thm:twoboundarycomponents} is sharp.

\subsection{Veering triangulations}

We review some key points from the theory of veering triangulations in order to explain some ideas in the proof of \Cref{thm:twoboundarycomponents}. See \Cref{subsec:vt}, \Cref{sec:LMT}, and \Cref{sec:closedorbits} for more details.

Let $\phi$ be a pseudo-Anosov flow on an oriented closed 3-manifold $M$. Let $\mathcal{C}$ be a collection of closed orbits satisfying the technical condition that $\phi$ has no perfect fits relative to $\mathcal{C}$; we show that such a collection always exists when $\phi$ is transitive.

\begin{prop:noperfectfitexists}
Let $\phi$ be a transitive pseudo-Anosov flow on a closed 3-manifold. There exists a collection of orbits $\mathcal{C}$ such that $\phi$ is without perfect fits relative to $\mathcal{C}$. In fact, $\mathcal{C}$ can be chosen to be the set of singular orbits and one other orbit.
\end{prop:noperfectfitexists}

The main result is then that there exists a veering triangulation on the cusped manifold $M \backslash \bigcup \mathcal{C}$. Moreover, the $2$-skeleton of the triangulation is positively transverse to $\phi$ and the combinatorics of the triangulation encodes the dynamics of the flow. 

The existence of the triangulation is unpublished work of Agol-Gu\'eritaud, while transversality of the $2$-skeleton is due to Landry-Minsky-Taylor \cite{LMT21}. To achieve this, they had to show that one can place the edges appropriately while constructing the triangulation. An important point for us is that there is no canonical choice for this placement of edges, and we shall utilize this freedom to construct some transverse surfaces to the flow.

\subsection{Method of constructing Birkhoff sections}

We introduce the notion of a broken transverse surface to a flow. These are surfaces which have a vertical boundary and a horizontal boundary. The vertical boundary is tangent to the flow whereas the interior of the surface and the horizontal boundary are transverse to the flow. See \Cref{defn:brokentransversesurface} for details. Provided that their horizontal boundaries match up, one can glue up a collection of broken transverse surfaces into a transverse surface that only has vertical boundary. If in addition every orbit meets one of the broken transverse surfaces in finite time, then the glued up transverse surface would be a Birkhoff section.

For a given pseudo-Anosov flow $\phi$, we construct two types of broken transverse surfaces out of the combinatorial data of an associated veering triangulation $\Delta$. 

The first type of surfaces are helicoids. We construct these from winding edge paths, which are edge paths in the universal cover $\widetilde{\Delta}$ that wind around an orbit. The vertical boundary of such a helicoidal surface is the orbit which the edge path winds around, while the horizontal boundary is the collection of edges in the edge path (see \Cref{fig:helicoid}). It is here that the edge placements mentioned in the previous subsection come up. We have to argue that we can arrange for this winding behavior to occur where we expect it to. This is done in a fairly technical trace through Landry-Minsky-Taylor's work in \Cref{sec:LMT}.

The second type of surfaces comes from the shearing decomposition of veering triangulations, introduced by Schleimer and Segerman in \cite{SS22a}. These surfaces are obtained by putting together faces of the triangulation, and essentially have no vertical boundary. See \Cref{subsec:veeringsolidtori} for details. 

For the proof of \Cref{thm:closedBirkhoffsection}, we first choose sufficiently many surfaces of the second type in order to intersect all the orbits, then show that we can construct two helicoids with matching horizontal boundary. This allows us to glue the surfaces up into an immersed Birkhoff section. A standard trick of Fried \cite{Fri83} then allows us to resolve the self-intersections and get a honest Birkhoff section. The boundary of this Birkhoff section comes from the vertical boundary of the two helicoids, hence consists of exactly two closed orbits.

In our construction, we will also keep track of the complexity of the various objects involved, so as to obtain the explicit bounds in \Cref{thm:closedBirkhoffsection}. The reader who is ultimately only interested in \Cref{thm:twoboundarycomponents} can skip these parts.

\subsection{Outline of the paper}

In \Cref{sec:background}, we recall some background knowledge about pseudo-Anosov flows, Birkhoff sections, and veering triangulations. In \Cref{sec:LMT}, we recall the construction in \cite{LMT21} of the veering triangulation associated to a pseudo-Anosov flow, explaining how we can arrange for winding edge paths along the way. In \Cref{sec:closedorbits}, we recall work in \cite{LMT21} on encoding closed orbits of the flow using the dual graph and flow graph of the veering triangulation. This is so that we can define the flow graph complexity of closed orbits and establish some lemmas for keeping track of complexities. 

In \Cref{sec:helicoid}, we explain how to construct the helicoidal broken transverse surfaces as mentioned above. The construction goes through objects which we call edge sequences. These lift up to winding edge paths in the universal cover which bound the desired helicoids. In \Cref{sec:Birkhoffsection}, we recall the shearing decomposition, and we prove \Cref{thm:cuspedBirkhoffsection} and \Cref{thm:closedBirkhoffsection}. In \Cref{sec:questions}, we include some extra discussion of our theorems and present some future questions coming out of this paper. 

{\bf Acknowledgements.} I would like to thank Ian Agol and Michael Landry for their support and encouragement. I would like to thank Pierre Dehornoy, Th\'eo Marty, Saul Schleimer, Henry Segerman, and Mario Shannon for enlightening conversations. I would like to thank Michael Landry, Yair Minsky, and Sam Taylor for allowing me to reproduce some of the figures in \cite{LMT21}. Finally, I would like to thank the anonymous referee for their careful reading of the paper and for their helpful suggestions.

{\bf Notational conventions.} Throughout this paper, 
\begin{itemize}
    \item $X \cut Y$ will denote the metric completion of $X \backslash Y$ with respect to the induced path metric from $X$. In addition, we will call the components of $X \cut Y$ the complementary regions of $Y$ in $X$.
    \item $\widetilde{X}$ will denote the universal cover of $X$, unless otherwise stated.
    \item Suppose $\alpha$ is a path, then $-\alpha$ will denote the path traversed in the opposite direction.
    \item Suppose $\alpha$ and $\beta$ are paths, where the ending point of $\alpha$ equals to the starting point of $\beta$, then $\alpha * \beta$ will denote the concatenated path obtained by traversing $\alpha$ then $\beta$.
    \item Suppose $\mathcal{C}$ is a collection of sets, then $\bigcup \mathcal{C}$ will denote the union over all elements of $\mathcal{C}$.
\end{itemize}

\section{Background} \label{sec:background}

\subsection{Pseudo-Anosov flows} \label{subsec:pAflow}

We recall the definition of a pseudo-Anosov flow.

\begin{defn} \label{defn:phorbit}
Let $n \geq 2$ be an integer. Let $p_n: \mathbb{R}^2 \to \mathbb{R}^2$ be the map defined by identifying $\mathbb{R}^2 \cong \mathbb{C}$ and sending $z$ to $z^{\frac{n}{2}}$. When $n$ is odd, one has to choose a branch; any choice here would work. Consider the foliations of $\mathbb{R}^2$ by vertical and horizontal lines respectively. Let $l_n^s, l_n^u$ be the singular foliations of $\mathbb{R}^2$ obtained by pulling back these foliations under $p_n$ respectively. We refer to lifts of the quadrants in $\mathbb{R}^2$ under $p_n$ as \textit{quadrants} as well.

Let $\lambda>1$. Consider the map $\begin{bmatrix} \lambda & 0\\ 0 & \lambda^{-1} \end{bmatrix}: \mathbb{R}^2 \to \mathbb{R}^2$. Let $\phi_{n,0,\lambda}:\mathbb{R}^2 \to \mathbb{R}^2$ be the lift of this map over $p_n$ that preserves the quadrants. Let $\phi_{n,k,\lambda}: \mathbb{R}^2 \to \mathbb{R}^2$ be the composition of $\phi_{n,0,\lambda}$ and rotation by $\frac{2\pi k}{n}$ anticlockwise. Since $\begin{bmatrix} \lambda & 0\\ 0 & \lambda^{-1} \end{bmatrix}$ preserves the foliations of $\mathbb{R}^2$ by vertical and horizontal lines respectively, $l_n^s$ and $l_n^u$ are preserved by $\phi_{n,k,\lambda}$. We depict the dynamics of $\phi_{n,0,\lambda}$ for $n=3$ in \Cref{fig:phorbit}. 

\begin{figure}
    \centering
    \resizebox{!}{4.5cm}{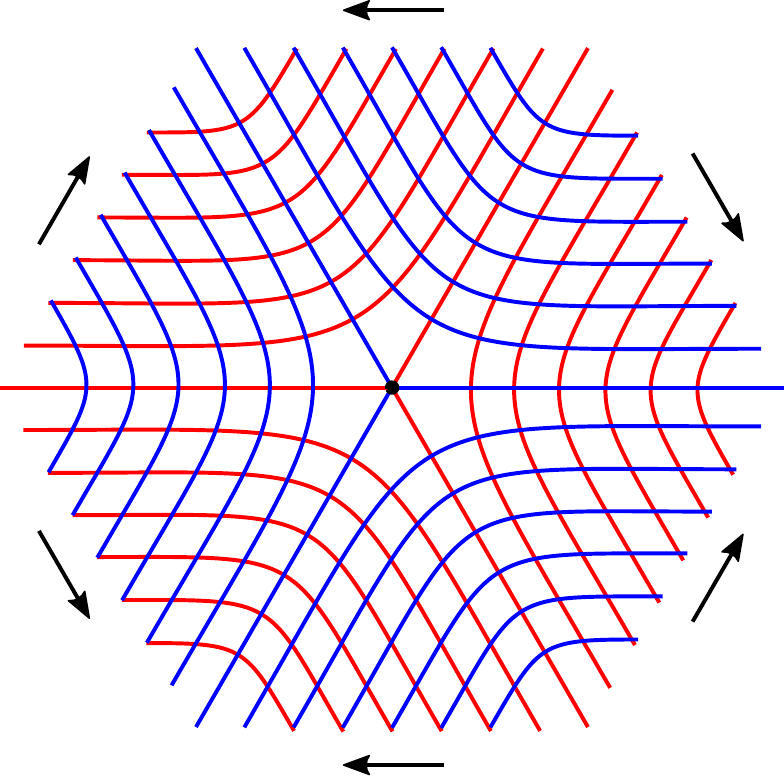}
    \caption{The dynamics of $\phi_{n,0,\lambda}$ for $n=3$.}
    \label{fig:phorbit}
\end{figure}

Let $\Phi_{n,k,\lambda}$ be the mapping torus of $\phi_{n,k, \lambda}$, let $\Lambda^s, \Lambda^u$ be the suspensions of $l^s_n, l^u_n$ respectively, and consider the suspension flow on $\Phi_{n,k, \lambda}$. Call the suspension of the origin the \textit{singular orbit} of $\Phi_{n,k,\lambda}$.
\end{defn}

\begin{defn} \label{defn:pAflow}
A \textit{pseudo-Anosov flow} on a closed 3-manifold $M$ is a $C^1$-flow $\phi_t$ satisfying:
\begin{itemize}
    \item There is a finite collection of closed orbits $\{\gamma_1, ..., \gamma_s \}$, called the \textit{singular orbits}, such that $\phi_t$ is smooth away from the singular orbits.
    \item There is a path metric $d$ on $M$, which is induced from a Riemannian metric $g$ away from the singular orbits.
    \item Away from the singular orbits, there is a splitting of the tangent bundle into three $\phi_t$-invariant line bundles $TM=E^s \oplus E^u \oplus T\phi_t$, such that $$|d\phi_t(v)| < C \lambda^{-t} |v|$$ for every $v \in E^s, t>0$, and $$|d\phi_t(v)| < C \lambda^t |v|$$ for every $v \in E^u, t<0$, for some $C, \lambda>1$.
    \item Each singular orbit $\gamma_i$ has a neighborhood $N_i$ and a map $f_i$ sending $N_i$ to a neighborhood of the singular orbit in $\Phi_{n_i, k_i, \lambda}$, for some $n_i \geq 3$, such that $f_i$ is bi-Lipschitz on $N_i$ and smooth away from $\gamma_i$, preserves the orbits, and sends $E^s, E^u$ to line bundles tangent to $\Lambda^s, \Lambda^u$ respectively. In this case, we say that $\gamma_i$ is \textit{$n_i$-pronged}. By extension, we also say that a non-singular orbit is \textit{$2$-pronged}.
\end{itemize}

We call the (possibly singular) foliation which is tangent to $E^s \oplus T\phi_t$ away from the singular orbits and given by the image of $\Lambda^s \subset \Phi_{n_i,k_i,\lambda}$ under $f_i$ near the singular orbits the \textit{stable foliation} $\Lambda^s$. We define the \textit{unstable foliation} $\Lambda^u$ similarly.

An \textit{Anosov flow} is a pseudo-Anosov flow without singular orbits.
\end{defn}

\begin{defn} \label{defn:transitive}
A flow on a closed 3-manifold $M$ is said to be \textit{transitive} if it has an orbit that is dense in $M$.  
\end{defn}

\begin{defn} \label{defn:orbitequivalence}
Let $\phi_i$ be a flow on a 3-manifold $M_i$ for $i=1,2$. We say that $\phi_1$ is \textit{orbit equivalent} to $\phi_2$ if there is a homeomorphism $h:M_1 \to M_2$ sending orbits of $\phi_1$ to orbits of $\phi_2$ in an orientation preserving way (but not necessarily preserving the parametrizations by the flows). In this case we say that $h$ is an \textit{orbit equivalence}.
\end{defn}

We next recall the definition of no perfect fits. This was introduced by Fenley in \cite{Fen99} and slightly generalized in \cite{AT22}. Here we use the generalized definition.

\begin{defn} \label{defn:perfectfit}
Let $\phi$ be a pseudo-Anosov flow on a closed 3-manifold $M$, and let $\mathcal{C}$ be a nonempty finite collection of closed orbits of $\phi$ which includes all the singular orbits of $\phi$. Lift $\phi$ up to a flow $\widetilde{\phi}$ on the universal cover $\widetilde{M}$. Let $\widetilde{\mathcal{C}}$ be the set of orbits of $\widetilde{\phi}$ which cover the orbits in $\mathcal{C}$.

Let $\mathcal{O}$ be the space of orbits of $\widetilde{\phi}$, endowed with the quotient topology. We refer to $\mathcal{O}$ as the \textit{orbit space} of $\phi$. It is shown in \cite[Proposition 4.2]{FM01} that $\mathcal{O}$ is homeomorphic to $\mathbb{R}^2$, and the images of $\Lambda^s, \Lambda^u$ under the projection $\widetilde{M} \to \mathcal{O}$ are two (possibly singular) 1-dimensional foliations $\mathcal{O}^s, \mathcal{O}^u$ respectively.

A \textit{perfect fit rectangle} is a rectangle-with-one-ideal-vertex properly embedded in $\mathcal{O}$ such that the restrictions of $\mathcal{O}^s$ and $\mathcal{O}^u$ to the rectangle foliate it as a product, i.e. conjugate to the foliations of $[0,1]^2 \backslash \{(1,1)\}$ by vertical and horizontal lines. See \Cref{fig:perfectfitdefn}. 

\begin{figure}
    \centering
    \resizebox{!}{3cm}{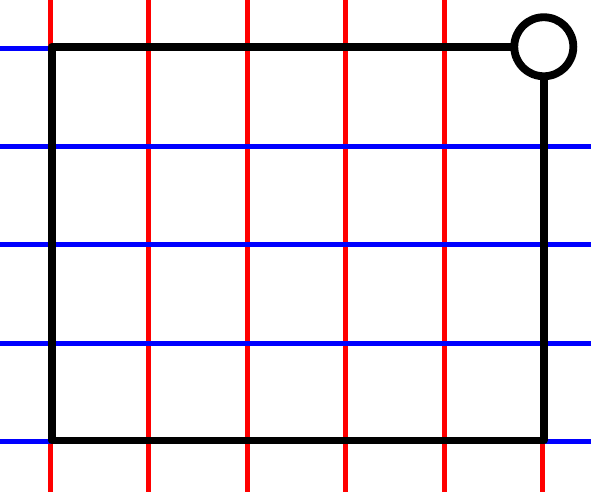}
    \caption{A perfect fit rectangle.}
    \label{fig:perfectfitdefn}
\end{figure}

We say that $\phi$ \textit{has no perfect fits relative to $\mathcal{C}$} if every perfect fit rectangle in $\mathcal{O}$ intersects $\widetilde{\mathcal{C}}$.

If $\phi$ has no perfect fits relative to the collection of singular orbits, then we say that $\phi$ \textit{has no perfect fits}.
\end{defn}

\begin{constr} \label{constr:GFsurgery}
Let $\phi$ be a pseudo-Anosov flow on a closed 3-manifold $M$ and let $\gamma$ be a closed orbit of $\phi$. Let $N(\gamma)$ be a small tubular neighborhood of $\gamma$. The leaf of the restricted foliation $\Lambda^s|_{N(\gamma)}$ containing $\gamma$ intersects $\partial N(\gamma)$ in a collection of closed curves. Let $l$ be the union of all these curves. 

Let $s$ be a slope on $\partial N(\gamma)$ such that $|\langle s,l \rangle| \geq 2$, and let $M'$ be the closed 3-manifold obtained by Dehn filling $M \backslash N(\gamma)$ along $s$. Then there exists a pseudo-Anosov flow $\phi'$ on $M'$ with a closed orbit $\gamma'$ isotopic to the core of the Dehn filling, such that $\phi$ restricted to $M \backslash \gamma$ is orbit equivalent to $\phi'$ restricted to $M' \backslash \gamma'$. Moreover, $\gamma'$ will be $|\langle s,l \rangle|$-pronged. 

Such a construction is commonly known in the literature as \textit{Goodman-Fried surgery}. The history behind this is rather interesting. Goodman, in \cite{Goo83}, introduced a way of performing this construction: one excises a round handle neighborhood of $\gamma$ and inserts another round handle neighborhood which achieves the correct filling slope $s$ to get $\phi'$. Fried, in \cite{Fri83}, introduced another way of performing this construction: one `blows up' the flow along $\gamma$ and collapses the torus boundary component along the filling slope $s$ to get $\phi'$. It was assumed early on that the flows produced by the two methods are orbit equivalent, but it was later realized that there was no rigorous proof of this. It was only recently shown by Shannon in his thesis \cite{Sha20} that this is indeed the case when $\phi$ is transitive, and even now, it is still open whether this is true when $\phi$ is not transitive. See \cite{Sha20} for a more in-depth discussion. For the purposes of this paper, one can just choose their preferred way of performing the surgery when we say perform Goodman-Fried surgery.
\end{constr}

\begin{prop} \label{prop:noperfectfitexists}
Let $\phi$ be a transitive pseudo-Anosov flow on a closed 3-manifold $M$. There exists a collection of orbits $\mathcal{C}$ such that $\phi$ is without perfect fits relative to $\mathcal{C}$. In fact, $\mathcal{C}$ can be chosen to be the set of singular orbits and one other orbit.
\end{prop}
\begin{proof}
We first make the following definition:

\begin{defn} \label{defn:transversepolygon}
Recall the notation in \Cref{defn:phorbit}. The preimage of the square $[-1,1]^2$ under $p_n$ is a $2n$-gon. We call the copy of this $2n$-gon in a fiber of the mapping torus $\Phi_{n,k,\lambda}$ a \textit{standard transverse $2n$-gon}. 

Now suppose $x \in M$. Suppose $f$ is a homeomorphism sending a neighborhood of $x$ to a neighborhood in $\Phi_{n,k,\lambda}$ containing a standard transverse $2n$-gon $R$ such that:
\begin{itemize}
    \item $f$ preserves the foliations $\Lambda^s$ and $\Lambda^u$
    \item $f$ sends $x$ to a point $y \in R$
\end{itemize}
Then we call the preimage $f^{-1}(R)$ a \textit{transverse polygon} at $x$. The leaves of $l^s_n$ and $l^u_n$ that contain $y$ divide $R$ into certain regions. We call the preimage of one of these regions that contains $y$ a \textit{sector} of the transverse polygon.

Notice that we do not require $x$ or $y$ to lie on a singular orbit in this definition. Hence, for example, a transverse polygon at a point lying on a nonsingular orbit can be a $(2n \geq 6)$-gon, but in such a case the transverse polygon will still only have 4 sectors.
\end{defn}

Returning to the proof of the proposition, we fix some metric on $M$. There exists $\epsilon > 0$ such that for every $x \in M$, there is a transverse polygon $R$ at $x$ for which each sector $S$ of $R$ intersects every orbit passing through the $\epsilon$-neighborhood of some point $z \in S$. Such $\epsilon$ can be chosen for $x$ in small neighborhoods, then using compactness of $M$ we obtain a uniform $\epsilon$.

Meanwhile, by transitivity and the shadowing lemma (see \cite[Lemma 1.3]{Man98}), there exists a closed orbit $\gamma$ of $\phi$ that intersects every $\epsilon$-neighborhood in $M$. By the choice of $\epsilon$ above, we know that for every $x \in M$, there is a transverse polygon $R$ at $x$ such that $\gamma$ intersects every sector of $R$. We take $\mathcal{C}$ to be the set of singular orbits and $\gamma$, and claim that $\phi$ has no perfect fits relative to $\mathcal{C}$.

To show the claim, we recall the notion of lozenges, which were also introduced by Fenley in \cite{Fen99}.
 
\begin{defn} \label{defn:lozenge}
A \textit{lozenge} is a rectangle-with-two-opposite-ideal-vertices properly embedded in $\mathcal{O}$ such that the restrictions of $\mathcal{O}^s$ and $\mathcal{O}^u$ to the rectangle foliate it as a product, i.e. conjugate to the foliations of $[0,1]^2 \backslash \{(0,0), (1,1)\}$ by vertical and horizontal lines. See \Cref{fig:lozengedefn}.
\end{defn}

\begin{figure}
    \centering
    \resizebox{!}{3cm}{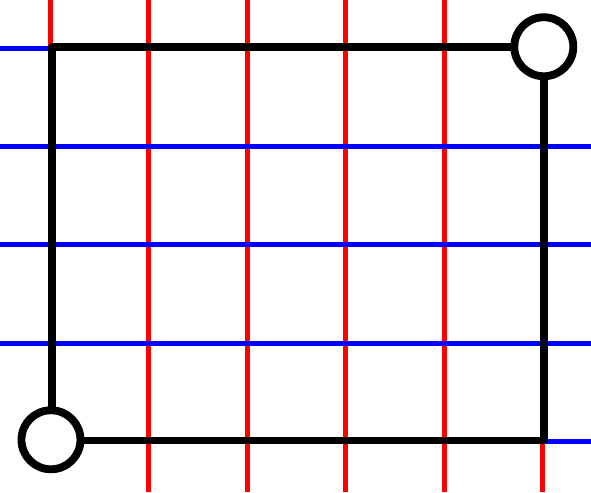}
    \caption{A lozenge.}
    \label{fig:lozengedefn}
\end{figure}

The result we need to use is the following.

\begin{prop}[{\cite[Proposition 5.5]{Fen16}}] \label{prop:perfectfitimplylozenge}
Let $\phi$ be a pseudo-Anosov flow on a closed 3-manifold $M$. If the orbit space of $\phi$ contains a perfect fit rectangle then it contains a lozenge.
\end{prop}

Returning to the proof of our claim, we first perform Goodman-Fried surgery on $\mathcal{\gamma}$ to make it singular, that is, we get a pseudo-Anosov flow $\phi'$ on a closed 3-manifold $M'$ with a singular orbit $\gamma'$, such that $\phi$ restricted to $M \backslash \gamma$ is orbit equivalent to $\phi'$ restricted to $M' \backslash \gamma'$. If we let $\mathcal{O}$ and $\mathcal{O}'$ be the orbit spaces of $\phi$ and $\phi'$ respectively, then the orbit equivalence $M \backslash \gamma \cong M' \backslash \gamma'$ induces a homeomorphism $\widetilde{\mathcal{O} \backslash \widetilde{\{\gamma\}}} \cong \widetilde{\mathcal{O}' \backslash \widetilde{\{\gamma'\}}}$ which maps the lifted stable/unstable foliations of one flow to the other.

Now suppose that $\mathcal{O}$ contains a perfect fit rectangle disjoint from $\widetilde{\{\gamma\}}$, then we can lift it to $\widetilde{\mathcal{O} \backslash \widetilde{\{\gamma\}}}$, transfer it to $\widetilde{\mathcal{O}' \backslash \widetilde{\{\gamma'\}}}$, then project down to $\mathcal{O}'$ to get a perfect fit rectangle in $\mathcal{O}'$. By \Cref{prop:perfectfitimplylozenge}, $\mathcal{O}'$ contains a lozenge. We can then run the above reasoning backwards to obtain a lozenge $L$ in $\mathcal{O}$ disjoint from $\widetilde{\{\gamma\}}$. 

Let $\alpha$ be one of the (non-ideal) corners of $L$. Recall that $\alpha$ is an orbit of $\widetilde{\phi}$. Let $x$ be a point on $\alpha$. We know from above that there is a transverse polygon $R$ at the image of $x$ in $M$ such that $\gamma$ intersects every sector of $R$. Lift $R$ to $\widetilde{R} \subset \widetilde{M}$ containing $x$. One of the sectors of $\widetilde{R}$ projects down to a region contained in $L \subset \mathcal{O}$. But some element of $\widetilde{\{\gamma\}}$ lies within such a region, contradicting the fact that $L$ is disjoint from $\widetilde{\{\gamma\}}$.
\end{proof}

\begin{rmk} \label{rmk:noperfectfitimpliestransitive}
We note that, conversely, if $\phi$ is a pseudo-Anosov flow on a closed 3-manifold $M$ without perfect fits relative to some collection of closed orbits $\mathcal{C}$, then $\phi$ must be transitive. 

To show this, we first perform Goodman-Fried surgery on orbits in $\mathcal{C}$ to make them singular, that is, we get a pseudo-Anosov flow $\phi'$ on a closed 3-manifold $M'$ with a collection of singular orbits $\mathcal{C'}$, such that $\phi$ restricted to $M \backslash \bigcup \mathcal{C}$ is orbit equivalent to $\phi'$ restricted to $M' \backslash \bigcup \mathcal{C'}$.

We claim that $\phi'$ has no perfect fits. Otherwise using the argument in the proof above, we can transfer a perfect fit rectangle in $\mathcal{O}'$ to a perfect fit rectangle in $\mathcal{O}$ that is disjoint from $\widetilde{\mathcal{C}}$, contradicting the hypothesis. Now by \cite[Corollary E]{Fen12}, $M'$ is atoroidal, and by \cite[Proposition 2.7]{Mos92a}, $\phi'$ is transitive. This implies that $\phi$ is transitive.
\end{rmk}

\begin{rmk} \label{shorterproofofnoperfectfitexists}
Using \Cref{thm:classicexistence}, one can obtain a much simpler proof of the first part of \Cref{prop:noperfectfitexists}. Indeed, the existence of a Birkhoff section, say, with boundary along a collection of closed orbits $\mathcal{C}$, implies that there exists a pseudo-Anosov flow $\phi'$ on a closed 3-manifold $M'$ with a collection of closed orbits $\mathcal{C'}$, such that $\phi'$ is the suspension flow on some pseudo-Anosov mapping torus, and such that $\phi$ restricted to $M \backslash \bigcup \mathcal{C}$ is orbit equivalent to $\phi'$ restricted to $M' \backslash \bigcup \mathcal{C'}$. 

But it is well-known that suspension flows have no perfect fits, see, for example, \cite[Theorem G]{Fen12}. Hence $\phi$ has no perfect fits relative to $\mathcal{C}$, for otherwise we can apply the argument in the proof above to transfer a perfect fit rectangle in $\mathcal{O}$ disjoint from $\widetilde{\mathcal{C}}$ to a perfect fit rectangle in $\mathcal{O}'$.

One of the drawbacks of this proof, however, is that, as pointed out in the introduction, Fried's and Brunella's proofs of \Cref{thm:classicexistence} do not offer any control over the complexity of $\mathcal{C}$. In particular, we do not know of a way of recovering the second statement of \Cref{prop:noperfectfitexists} using this proof. Another reason for using the longer proof is that we aim to offer an independent proof of \Cref{thm:classicexistence} in this paper, and so we should avoid a circular argument.
\end{rmk}

One of the convenient features of a pseudo-Anosov flow without perfect fits is the following generalization of \cite[Theorem 4.8]{Fen99}.

\begin{lemma} \label{lemma:nohomotopicorbits}
Let $\phi$ be a pseudo-Anosov flow on a closed 3-manifold $M$ without perfect fits relative to a collection of orbits $\mathcal{C}$. Suppose $\gamma_1$ and $\gamma_2$ are two closed orbits of $\phi$ which are not elements of $\mathcal{C}$. If $[\gamma_1]^{k_1} = [\gamma_2]^{k_2}$ in $\pi_1(M \backslash \bigcup \mathcal{C})$, then $\gamma_1 = \gamma_2$.
\end{lemma}
\begin{proof}
Assume otherwise, we apply Goodman-Fried surgery on the orbits in $\mathcal{C}$ to make them singular, that is, we get a pseudo-Anosov flow $\phi'$ on a closed 3-manifold $M'$ with a collection of singular orbits $\mathcal{C'}$, such that $\phi$ restricted to $M \backslash \bigcup \mathcal{C}$ is orbit equivalent to $\phi'$ restricted to $M' \backslash \bigcup \mathcal{C'}$. $\gamma_1$ and $\gamma_2$ are sent by the orbit equivalence to closed orbits of $\phi'$ in $M' \backslash \bigcup \mathcal{C'}$, which we still call $\gamma_1$ and $\gamma_2$ respectively, such that $\gamma_1^{k_1}$ is homotopic to $\gamma_2^{k_2}$ in $M'$. By \cite[Theorem 4.8]{Fen99}, $\phi'$ must have perfect fits. But then one can transfer a perfect fit rectangle from the orbit space of $\phi'$ to that of $\phi$ as in the proof of \Cref{prop:noperfectfitexists} and obtain a contradiction.
\end{proof}

\subsection{Birkhoff sections} \label{subsec:Birkhoffsection}

\begin{defn} \label{defn:Birkhoffsection}
Let $\phi$ be a pseudo-Anosov flow on a closed 3-manifold $M$. An \textit{immersed Birkhoff section} is an immersed cooriented compact surface with boundary $S$ such that:
\begin{itemize}
    \item The interior of $S$ is positively transverse to the orbits of $\phi$
    \item The boundary of $S$ is a union of closed orbits of $\phi$.
    \item Every orbit of $\phi$ intersects $S$ in finite foward and finite backward time, that is, for every $x \in M$, there exists $t_1, t_2 > 0$ such that $\phi_{t_1}(x) \in S$ and $\phi_{-t_2}(x) \in S$.
\end{itemize}
When $M$ is oriented, we orient the boundary components of $S$ using the induced orientation on $S$. We say that a boundary component of $S$ is \textit{positive} if its orientation agrees with the flow direction, otherwise it is \textit{negative}. See \Cref{fig:Birkhoffsectionboundary}.

\begin{figure} 
    \centering
    \fontsize{40pt}{40pt}\selectfont
    \resizebox{!}{4.5cm}{
\begingroup%
  \makeatletter%
  \providecommand\color[2][]{%
    \errmessage{(Inkscape) Color is used for the text in Inkscape, but the package 'color.sty' is not loaded}%
    \renewcommand\color[2][]{}%
  }%
  \providecommand\transparent[1]{%
    \errmessage{(Inkscape) Transparency is used (non-zero) for the text in Inkscape, but the package 'transparent.sty' is not loaded}%
    \renewcommand\transparent[1]{}%
  }%
  \providecommand\rotatebox[2]{#2}%
  \newcommand*\fsize{\dimexpr\f@size pt\relax}%
  \newcommand*\lineheight[1]{\fontsize{\fsize}{#1\fsize}\selectfont}%
  \ifx\svgwidth\undefined%
    \setlength{\unitlength}{306.51822996bp}%
    \ifx\svgscale\undefined%
      \relax%
    \else%
      \setlength{\unitlength}{\unitlength * \real{\svgscale}}%
    \fi%
  \else%
    \setlength{\unitlength}{\svgwidth}%
  \fi%
  \global\let\svgwidth\undefined%
  \global\let\svgscale\undefined%
  \makeatother%
  \begin{picture}(1,0.79739748)%
    \lineheight{1}%
    \setlength\tabcolsep{0pt}%
    \put(0,0){\includegraphics[width=\unitlength,page=1]{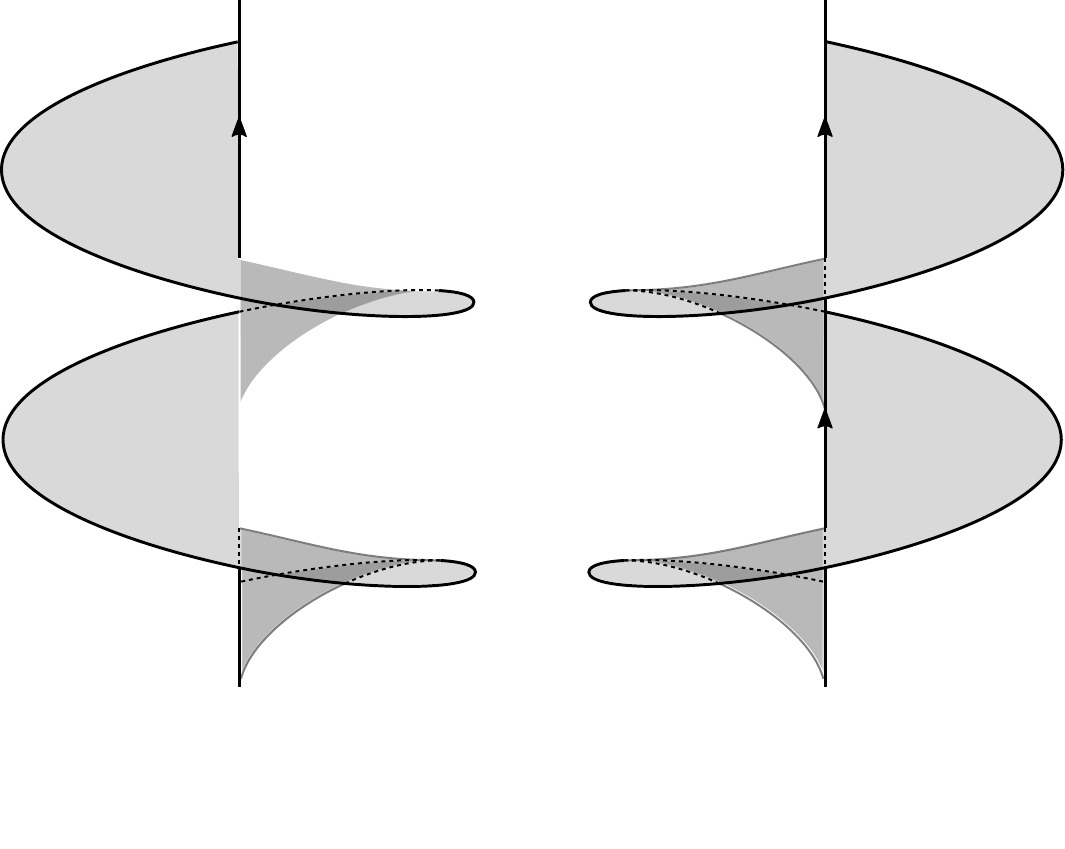}}%
    \put(0.16371587,0.02300281){\color[rgb]{0,0,0}\makebox(0,0)[lt]{\lineheight{1.25}\smash{\begin{tabular}[t]{l}$+$\end{tabular}}}}%
    \put(0.73794245,0.02300874){\color[rgb]{0,0,0}\makebox(0,0)[lt]{\lineheight{1.25}\smash{\begin{tabular}[t]{l}$-$\end{tabular}}}}%
    \put(0,0){\includegraphics[width=\unitlength,page=2]{Birkhoffsectionboundary.pdf}}%
  \end{picture}%
\endgroup%
}
    \caption{A Birkhoff section near a positive/negative boundary component.} 
    \label{fig:Birkhoffsectionboundary}
\end{figure}

Let $\mathcal{C}$ be the set of closed orbits for which some element of $\partial S$ lies along. The \textit{complexity} of $S$ is defined to be $\mathfrak{c}(S) = -\chi_{\top} (S \backslash \mathcal{C})$, where $\chi_{\top}$ denotes the Euler characteristic of the underlying space. In other words, we puncture the immersed surface at the points where it intersects its boundary and compute the negative of its Euler characteristic.

A \textit{Birkhoff section} is an immersed Birkhoff section that is embedded in its interior. Note that the complexity of a Birkhoff section is equals to negative of its Euler characteristic.
\end{defn}

Having an immersed Birkhoff section is essentially as good as having a Birkhoff section, since we have the following resolution trick, introduced by Fried in \cite{Fri83}.

\begin{constr} \label{constr:Friedresolution}
Let $\phi$ be a pseudo-Anosov flow on a closed 3-manifold $M$. Let $S$ be an immersed Birkhoff section. By a slight perturbation, we can assume that $S$ is in general position. In this case, the self-intersection set of $S$ is a graph that can be described as follows:

We call the interior points that are identified with a boundary point the \textit{cut points}, and call the interior points that are identified with two other interior points the \textit{triple points}. Then the self-intersection set of $S$ is a union of some boundary components and some curves and arcs that have endpoints on boundary components and cut points, and which are disjoint except for intersecting three at a time at triple points. We call each such curve or arc a \textit{double curve or arc}, respectively.

For each double curve or arc $l$, we cut and paste along $l$ as in \Cref{fig:Friedresolution} first row. The local picture around a cut point is as in \Cref{fig:Friedresolution} second row, and the local picture around a triple point is as in \Cref{fig:Friedresolution} third row. This results in a surface that is embedded in its interior but may not be immersed along its boundary. We call the boundary points that are not immersed the \textit{turning points}. Inductively, for each turning point, we cut and paste along the arcs that are identified on either side of the turning point as in \Cref{fig:Friedresolution} last row. We eventually get rid of all the turning points and get a Birkhoff section $S'$. We call $S'$ the \textit{Fried resolution} of $S$.

We caution that a closed orbit $\gamma$ in $\partial S$ may not be in $\partial S'$. This disappearance of boundary orbits happens exactly when the homology class of $S \cap \partial N(\gamma)$ is a multiple of the meridian, where in the last step of getting rid of the turning points, we end up collapsing the boundary components that lie on $\gamma$.

We also note that if $\partial S$ is embedded along a closed orbit $\gamma$, then the same is true for $\partial S'$. In general, the (signed) intersection number of $S \cap \partial N(\gamma)$ with the meridian at $\gamma$ is preserved under Fried resolution, since the effect of the operation on the homology class of $S \cap \partial N(\gamma)$ is summing with multiples of the meridian. 
\end{constr}

\begin{figure} 
    \centering
    \resizebox{!}{13cm}{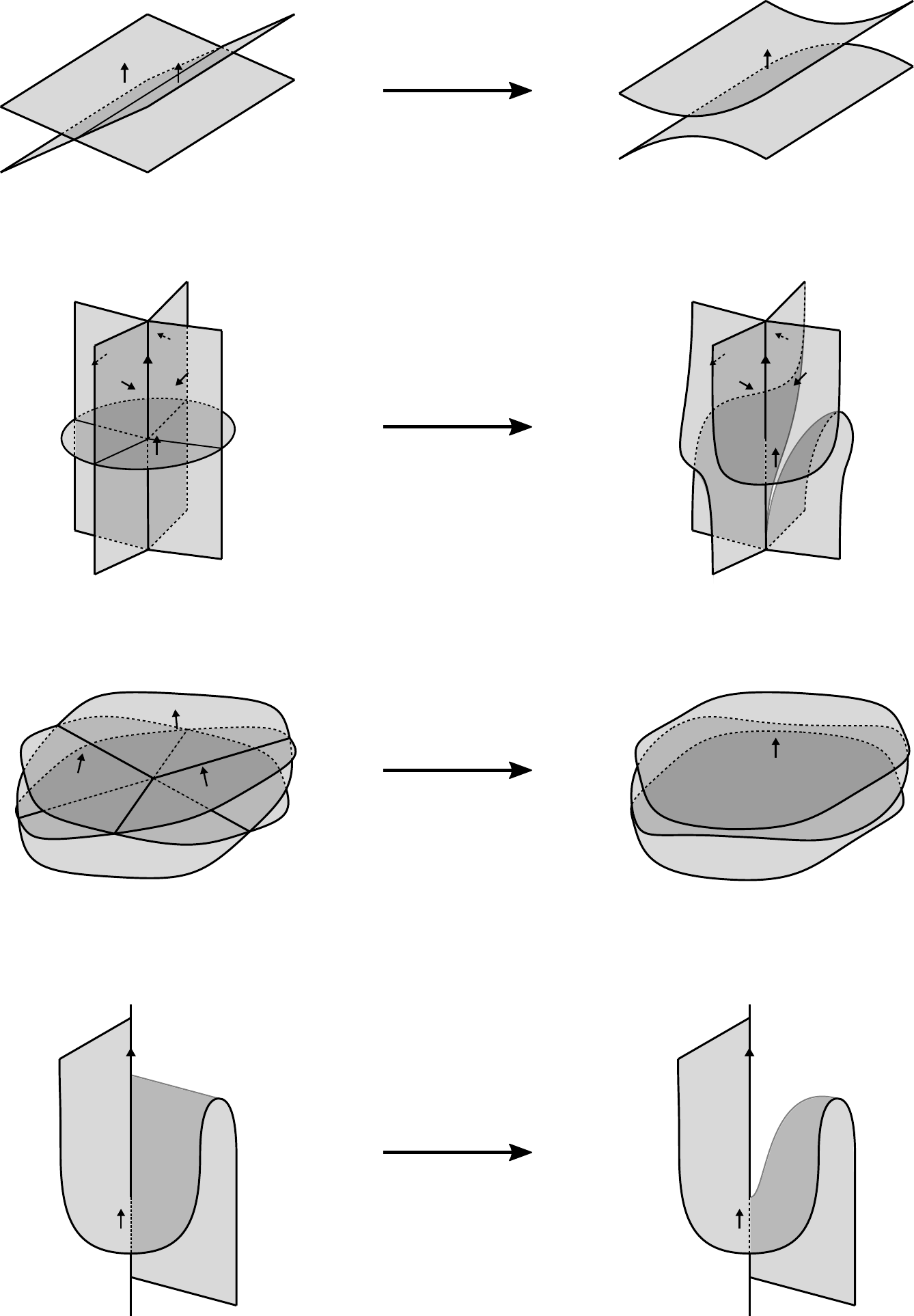}
    \caption{Performing Fried resolution on an immersed Birkhoff section.} 
    \label{fig:Friedresolution}
\end{figure}

\begin{lemma} \label{lemma:Friedresolutioncomplexity}
Let $S$ be an immersed Birkhoff section and $S'$ be its Fried resolution. Then $\mathfrak{c}(S) \geq \mathfrak{c}(S')$
\end{lemma}

To explain this proof, we make the following definition.

\begin{defn} \label{defn:surfacewithcorners}
A \textit{surface with corners} is a surface with boundary $S$ along with a finite collection of points on $\partial S$, which we call the \textit{corners}. The complementary regions of the corners in $\partial S$ are called the \textit{sides}. A \textit{punctured surface with corners} is a surface with corners $S$ with a finite collection of interior points and corners removed.

The \textit{index} of a surface with corners $S$ is defined to be $\ind(S) = \chi_{\top}(S) - \frac{1}{4} \text{\# corners}$. If we remove $x$ interior points and $y$ corners from $S$ to get a punctured surface with corners $S'$, then the \textit{index} of $S'$ is defined to be $\ind(S') = \ind(S)-x-\frac{1}{4} y$.
\end{defn}

Note that the index is additive, in the sense that if a punctured surface with corners $S$ is divided into $S'$ by a collection of disjoint curves and properly embedded arcs, then $\ind(S)=\ind(S')$.

\begin{rmk} \label{rmk:indexhypmetric}
A good motivation and mnemonic for the definition of the index comes from the Gauss-Bonnet formula: If a surface with corners $S$ can be endowed with a hyperbolic metric such that the sides are geodesic and the corners form right angles, then $\ind(S)=-\frac{\text{area($S$)}}{2\pi}$. Similarly, if a punctured surface with corners $S'$ is obtained by removing interior points and corners from $S$, and if $S'$ admits a hyperbolic metric such that the sides are geodesics, the unremoved corners form right angles, and the removed points form cusps, then $\ind(S)=-\frac{\text{area($S$)}}{2\pi}$.
\end{rmk}

\begin{proof}[Proof of \Cref{lemma:Friedresolutioncomplexity}]
Let $I$ be the self-intersection set of $S$ and let $V$ be the set of cut points and triple points of $S$. Let $S_1$ be the surface obtained after cutting and pasting along the double curves and arcs of $S$, and let $I_1$ and $V_1$ be the images of $I$ and $V$ in $S_1$ respectively. Notice that $S_1$ can be obtained from $S$ by cutting along $I$ and gluing back the pieces in a different way.

Now observe that $I \backslash V$ is a collection of curves and properly embedded arcs in the punctured surface $S \backslash V$ that divides it into punctured surface with corners $(S \backslash V) \cut (I \backslash V)$. Similarly, $I_1 \backslash V_1$ is a collection of curves and properly embedded arcs in the punctured surface $S_1 \backslash V_1$ that divides it into punctured surface with corners $(S_1 \backslash V_1) \cut (I_1 \backslash V_1)$, and we have $(S \backslash V) \cut (I \backslash V)=(S_1 \backslash V_1) \cut (I_1 \backslash V_1)$. Hence $\ind(S \backslash V)=\ind(S_1 \backslash V_1)$, and $$\mathfrak{c}(S)=-\ind(S \backslash V)-\text{\# triple points} = -\ind(S_1 \backslash V_1)-\text{\# triple points}=-\chi_{\top}(S_1).$$

Meanwhile, $S'$ is obtained from $S_1$ by resolving the turning points on $\partial S_1$. This changes the topology of $S_1$ by possibly collapsing some boundary components. This collapsing happens when a boundary component of $S_1$ is immersed along a closed orbit via a null-homotopic map; after resolving turning points on this boundary component, it is mapped to a single point in $M$. Topologically, this is equivalent to filling in certain boundary components of $S_1$ by discs, hence 
$$\mathfrak{c}(S)=-\chi_{\top}(S_1) \geq -\chi_{\top}(S') = \mathfrak{c}(S').$$
\end{proof}

As mentioned in the introduction, the way we will construct Birkhoff sections in this paper is to assemble multiple pieces of surfaces. To that end, we define the notion of broken transverse surfaces.

\begin{defn} \label{defn:brokentransversesurface}
Let $\phi$ be a pseudo-Anosov flow on a closed 3-manifold $M$. A \textit{broken transverse surface} is an immersed cooriented surface with corners $S$ such that:
\begin{itemize}
    \item The interior of $S$ is positively transverse to the orbits of $\phi$.
    \item Every boundary component $\alpha$ of $S$ has an even number of corners. The sides along $\alpha$ lie along closed orbits or are transverse to the orbits of $\phi$ alternatingly. 
\end{itemize}
We denote the union of sides of $S$ that lie along closed orbits by $\partial_v S$ and call it the \textit{vertical boundary} of $S$. We denote the union of sides of $S$ that are transverse to orbits of $\phi$ by $\partial_h S$ and call it the \textit{horizontal boundary} of $S$.

When $M$ is oriented, we orient the sides of $S$ using the induced orientation on $S$. A side in $\partial_v S$ is said to be \textit{positive} if its orientation agrees with the flow direction, otherwise it is \textit{negative}.
\end{defn}

If $S$ is a broken transverse surface whose sides in $\partial_h S$ can be grouped together in pairs $(e_1,e_2)$ such that $e_1=-e_2$ as paths, and if every orbit of $\phi$ intersects $S$ in finite forward and backward time, then $S$ can be glued along each pair of sides to give a surface immersed in its interior. The surface may not be immersed along its boundary if the signs of the sides in $\partial_v S$ do not match up, but we can apply the last step of Fried resolution to resolve any turning points and get an immersed Birkhoff section $S'$.

Let $\{ \gamma_1,...,\gamma_k \}$ be the set of orbits for which some element of $\partial S'$ lies along. Using additivity of the index, we get the bound $\mathfrak{c}(S') \leq - \ind(S) + \sum_{i=1}^k \langle S, \gamma_i \rangle$. As in \Cref{lemma:Friedresolutioncomplexity}, strict inequality holds if we have to collapse a boundary component while resolving turning points.

\subsection{Veering triangulations} \label{subsec:vt}

We recall the definition of a veering triangulation.

\begin{defn} \label{defn:tautstructure}
An \textit{ideal tetrahedron} is a tetrahedron with its 4 vertices removed. The removed vertices are called the \textit{ideal vertices}. An \textit{ideal triangulation} of a 3-manifold $M$ is a decomposition of $M$ into ideal tetrahedra glued along pairs of faces.

A \textit{taut structure} on an ideal triangulation is a labelling of the dihedral angles by $0$ or $\pi$, such that 
\begin{itemize}
    \item Each tetrahedron has exactly two dihedral angles labelled $\pi$, and they are opposite to each other.
    \item The angle sum around each edge in the triangulation is $2\pi$.
\end{itemize}

A \textit{transverse taut structure} is a taut structure along with a coorientation on each face, such that for any 0-labelled edge in a tetrahedron, exactly one of the faces adjacent to it is cooriented inwards. A \textit{transverse taut ideal triangulation} is an ideal triangulation with a transverse taut structure.

In this paper, we will take the convention that the coorientations are always pointing upwards.
\end{defn}

\begin{defn} \label{defn:vt}
A \textit{veering structure} on a transverse taut ideal triangulation of an oriented 3-manifold is a coloring of the edges by red or blue, so that if we look at each tetrahedron with a $\pi$-labelled edge in front, the four outer $0$-labelled edges, starting from an end of the front edge and going counter-clockwise, are colored red, blue, red, blue, respectively. See \Cref{fig:veertet}.

A \textit{veering triangulation} is a transverse taut ideal triangulation with a veering structure.
\end{defn}

\begin{rmk} \label{rmk:redblue}
In this paper, as the reader would have noticed in \Cref{subsec:pAflow}, we take the convention of drawing stable foliations in red and unstable foliations in blue, which is common in literature. The reader is cautioned that this usage of red/blue has no relation with using the same colors for the edges of veering triangulations; it is simply an unfortunate coincidence, perhaps due to the fact that there are only so many conspicuous colors to the human eye. To avoid confusion, we will reserve the colors red/blue in our figures for the latter usage when both contexts are present.
\end{rmk}

\begin{figure} 
    \centering
    \fontsize{14pt}{14pt}\selectfont
    \resizebox{!}{3.5cm}{
\begingroup%
  \makeatletter%
  \providecommand\color[2][]{%
    \errmessage{(Inkscape) Color is used for the text in Inkscape, but the package 'color.sty' is not loaded}%
    \renewcommand\color[2][]{}%
  }%
  \providecommand\transparent[1]{%
    \errmessage{(Inkscape) Transparency is used (non-zero) for the text in Inkscape, but the package 'transparent.sty' is not loaded}%
    \renewcommand\transparent[1]{}%
  }%
  \providecommand\rotatebox[2]{#2}%
  \newcommand*\fsize{\dimexpr\f@size pt\relax}%
  \newcommand*\lineheight[1]{\fontsize{\fsize}{#1\fsize}\selectfont}%
  \ifx\svgwidth\undefined%
    \setlength{\unitlength}{325.82965819bp}%
    \ifx\svgscale\undefined%
      \relax%
    \else%
      \setlength{\unitlength}{\unitlength * \real{\svgscale}}%
    \fi%
  \else%
    \setlength{\unitlength}{\svgwidth}%
  \fi%
  \global\let\svgwidth\undefined%
  \global\let\svgscale\undefined%
  \makeatother%
  \begin{picture}(1,0.40856788)%
    \lineheight{1}%
    \setlength\tabcolsep{0pt}%
    \put(0,0){\includegraphics[width=\unitlength,page=1]{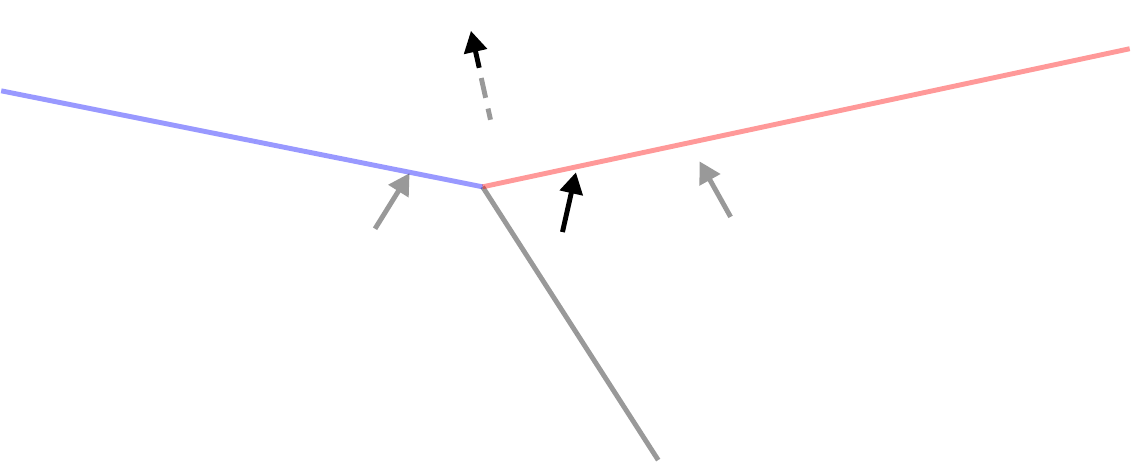}}%
    \put(0.45702311,0.36605622){\color[rgb]{0,0,0}\makebox(0,0)[lt]{\lineheight{1.25}\smash{\begin{tabular}[t]{l}$\pi$\end{tabular}}}}%
    \put(0.49904816,0.14324942){\color[rgb]{0,0,0}\transparent{0.40000001}\makebox(0,0)[lt]{\lineheight{1.25}\smash{\begin{tabular}[t]{l}$\pi$\end{tabular}}}}%
    \put(0.76168485,0.11769163){\color[rgb]{0,0,1}\makebox(0,0)[lt]{\lineheight{1.25}\smash{\begin{tabular}[t]{l}$0$\end{tabular}}}}%
    \put(0.31075696,0.09433859){\color[rgb]{1,0,0}\makebox(0,0)[lt]{\lineheight{1.25}\smash{\begin{tabular}[t]{l}$0$\end{tabular}}}}%
    \put(0.25420744,0.29133508){\color[rgb]{0,0,1}\transparent{0.40000001}\makebox(0,0)[lt]{\lineheight{1.25}\smash{\begin{tabular}[t]{l}$0$\end{tabular}}}}%
    \put(0.62556736,0.30455855){\color[rgb]{1,0,0}\transparent{0.40000001}\makebox(0,0)[lt]{\lineheight{1.25}\smash{\begin{tabular}[t]{l}$0$\end{tabular}}}}%
    \put(0,0){\includegraphics[width=\unitlength,page=2]{veertet.pdf}}%
  \end{picture}%
\endgroup%
}
    \caption{A tetrahedron in a veering triangulation. There are no restrictions on the colors of the top and bottom edges.} 
    \label{fig:veertet}
\end{figure}

We recall some basic combinatorial facts about veering triangulations.

An edge $e$ in a veering triangulation $\Delta$ is the bottom edge of a unique tetrahedron and the top edge of a unique tetrahedron. We say that these tetrahedra lie \textit{above} and \textit{below} $e$ respectively. In between these tetrahedra, on either side of $e$, there is an (a priori, possible empty) collection of tetrehedra incident to $e$, each having $e$ as a side edge. We refer to this collection as a \textit{stack} of tetrahedra on a side of $e$. \Cref{defn:fantet} and \Cref{prop:fantet} below describe the structure of these stacks.

\begin{defn} \label{defn:fantet}
A tetrahedron in $\Delta$ is called a \textit{toggle tetrahedron} if the colors on its top and bottom edges differ. It is called a \textit{red/blue fan tetrahedron} if both its top and bottom edges are red/blue respectively.

Note that some authors call toggle and fan tetrahedra \textit{hinge} and \textit{non-hinge} respectively.
\end{defn}

\begin{prop}[{\cite[Observation 2.6]{FG13}}] \label{prop:fantet}
Let $e$ be an edge in a veering triangulation $\Delta$. The stacks of tetrahedra on each side of $e$ must be nonempty. Suppose $e$ is blue/red. If there is exactly one tetrahedron in one stack, then that tetrahedron is a blue/red fan tetrahedron respectively. If there are $n>1$ tetrahedron in one stack, then going from bottom to top in that stack, the tetrahedra are: one toggle tetrahedron, $n-2$ red/blue fan tetrahedra, and one toggle tetrahedron. 
\end{prop}

The number of tetrahedra in the stack of tetrahedra to a side of an edge is referred to as the \textit{length} of the stack.

For any veering triangulation $\Delta$, one can complete the ideal triangulation by adding in the removed vertices of the ideal tetrahedra. The resulting space will not be a manifold since the link of each vertex is never a ball. Regardless, we will denote this completed triangulation as $\overline{\Delta}$. Let $T$ be a vertex of $\overline{\Delta}$ and let $N(T)$ be a small neighborhood of $T$ in $\overline{\Delta}$. $\partial N(T)$ inherits a \textit{boundary triangulation} $\partial \Delta$, where the vertices/edges/faces of $\partial \Delta$ correspond to vertices of edges/faces/tetrahedra of $\overline{\Delta}$ at $T$ respectively. In particular, each vertex of $\partial \Delta$ inherits the color of the corresponding edge of $\Delta$, and each edge of $\partial \Delta$ inherits the coorientation of the corresponding face of $\Delta$.

As a consequence of the veering structure, the boundary triangulation has to take on a particular form. Namely, there exist edge paths $\{l_i\}_{i \in \mathbb{Z}}$ in $\partial \Delta$ such that:
\begin{itemize}
    \item The vertices along $l_{2i}$ are all colored blue, while the vertices along $l_{2i+1}$ are all colored red.
    \item The faces between $l_{2i}$ and $l_{2i+1}$ form a stack of upward pointing triangles, while the faces between $l_{2i}$ and $l_{2i+1}$ form a stack of downward pointing triangles.
\end{itemize}
See \Cref{fig:boundarytriangulation}, and see \cite[Section 2]{FG13} for explanations. In this paper, we will take the convention that the indices are increasing from left to right, when we look at $\partial \Delta$ from inside $N(T)$.

\begin{figure}
    \centering
    \resizebox{!}{5.5cm}{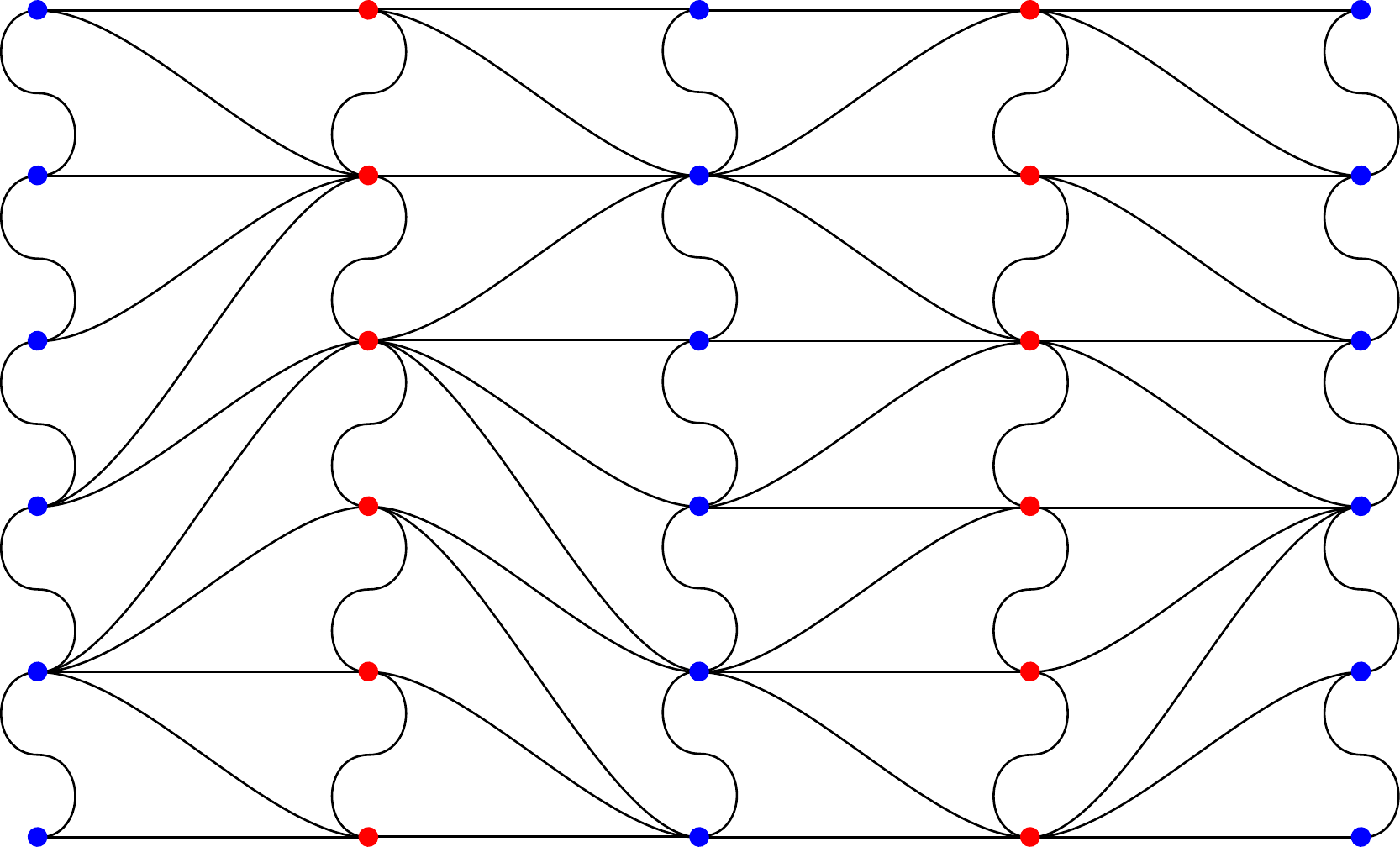}
    \caption{The boundary triangulation at a vertex of a veering triangulation.}
    \label{fig:boundarytriangulation}
\end{figure}

We orient each $l_i$ to go upwards, that is, in the direction that agrees with the coorientations on the edges of $\partial \Delta$. We call these oriented curves the \textit{ladderpole curves} at $T$. If we want to be more specific, we will call $l_i$ a \textit{blue/red ladderpole curve} if it contains only blue/red vertices respectively.

When $\overline{\Delta}$ is compact, $\partial \Delta$ is a triangulation of the torus $\partial N(T)$. In particular, there are only a finite, even number of ladderpole curves, that is, we have $l_i=l_{i+2n}$ for some $n \geq 1$, and $l_1,...,l_{2n}$ are distinct. In this case we say that $n$ is the \textit{ladderpole multiplicity} of $T$. We can complete a ladderpole curve into a basis for $H_1(\partial N(T))$ by defining a \textit{ladderpole transversal} at $T$ to be an embedded edge path $t$ of $\partial \Delta$ such that $\langle t,l_1 \rangle = 1$. Ladderpole transversals always exist: starting from $l_1$, one can move from each $l_i$ to $l_{i+1}$ (indices taken mod $n$) using one edge, and when one returns to $l_1$, one can complete the path by going up or down along $l_1$. We call a choice of a ladderpole transversal at each vertex of $\overline{\Delta}$ a \textit{system of ladderpole transversals}.

For convenience, we use the following shorthand in the rest of the paper:
\begin{itemize}
    \item We refer to a vertex of $\overline{\Delta}$ as a vertex of $\Delta$.
    \item We refer to a vertex of $\partial \Delta$ by the same name as the edge of $\overline{\Delta}$ that it corresponds to.
\end{itemize}

We make one more definition that will play a role in \Cref{subsec:veeringsolidtori}.

\begin{defn} \label{defn:quad}
Let $t$ be a tetrahedron in $\Delta$. An \textit{equatorial square} (or \textit{square} in short) in $t$ is a quadrilateral-with-4-ideal-vertices properly embedded in $t$ with its 4 sides along the side edges of $t$.
\end{defn}

To establish the bounds in \Cref{thm:cuspedBirkhoffsection} and \Cref{thm:closedBirkhoffsection}, we will need to compute bounds on the complexity of various objects along the way. To that end, we fix the notation for the following parameters of a veering triangulation $\Delta$:
\begin{itemize}
    \item $N$ will denote the number of tetrahedra in $\Delta$, which equals to the number of edges in $\Delta$.
    \item $\delta$ will denote the maximum length of the stack of tetrahedra to a side of an edge.
    \item $\nu$ will denote the maximum ladderpole multiplicity over all vertices of $\Delta$.
    \item $\lambda$ will denote the maximum length of all ladderpole curves (as edge paths of $\partial \Delta$).
\end{itemize}

\section{From pseudo-Anosov flows to veering triangulations} \label{sec:LMT}

The following theorem is one of the main results in \cite{LMT21}.

\begin{thm}[{\cite[Theorem 5.1]{LMT21}}] \label{thm:LMT}
Let $\phi$ be a pseudo-Anosov flow on an oriented closed 3-manifold $M$ without perfect fits relative to a collection of closed orbits $\mathcal{C}$. Then there exists a veering triangulation $\Delta$ on $M \backslash \bigcup \mathcal{C}$ whose 2-skeleton is positively transverse to $\phi$.
\end{thm}

In the sequel we will refer to such a $\Delta$ as a veering triangulation \textit{associated to $\phi$} on $M \backslash \bigcup \mathcal{C}$. In this section, we will recall the proof of the theorem, with the goal of pointing out how the technical property of having winding edge paths, as briefly explained in the introduction, can be arranged for in the proof.

\subsection{Rectangles} \label{subsec:rectangles}

We fix the following setting: Let $\phi$ be a pseudo-Anosov flow on an oriented closed 3-manifold $M$ without perfect fits relative to a collection of closed orbits $\mathcal{C}$. Let $\widetilde{\phi}$ be the lift of $\phi$ to the universal cover $\widetilde{M}$ and let $\widetilde{\mathcal{C}}$ be the set of orbits of $\widetilde{\phi}$ which cover the orbits in $\mathcal{C}$. Let $\mathcal{O}$ be the orbit space of $\phi$. 

We define the \textit{completed flow space} $\mathcal{P}$ to be the infinite branched cover of $\mathcal{O}$ over $\widetilde{\mathcal{C}}$ which restricts to the universal cover of $\mathcal{O} \backslash \widetilde{\mathcal{C}}$. We denote the branch points on $\mathcal{P}$ by $\mathcal{S}$. The foliations $\mathcal{O}^s$ and $\mathcal{O}^u$ lift to foliations $\mathcal{P}^s$ and $\mathcal{P}^u$ respectively. 

Finally, let $M^\circ = M \backslash \bigcup \mathcal{C}$. Lift the restricted flow $\phi|_{M^\circ}$ to $\widehat{\phi}$ on $\widetilde{M^\circ}$. Note that $\mathcal{P} \backslash \mathcal{S}$ is the space of orbits of $\widehat{\phi}$. In particular, $\pi_1(M^\circ)$ acts on $\mathcal{P}$, preserving the foliations $\mathcal{P}^s$ and $\mathcal{P}^u$. Let $q: \widetilde{M^\circ} \to \mathcal{P} \backslash \mathcal{S}$ be the quotient map. 

\begin{defn} \label{defn:rectangles}
A \textit{rectangle} $R$ in $\mathcal{P}$ is a rectangle embedded in $\mathcal{P}$ such that the restrictions of $\mathcal{P}^s$ and $\mathcal{P}^u$ foliate the rectangle as a product, i.e. conjugate to the foliations of $[0,1]^2$ by vertical and horizontal lines, and such that no element of $\mathcal{S}$ lies in the interior of $R$.

Let $R_1$ and $R_2$ be rectangles in $\mathcal{P}$. $R_1$ is said to be \textit{taller} than $R_2$ if every leaf of $\mathcal{P}^u$ that intersects $R_2$ intersects $R_1$. $R_1$ is said to be \textit{wider} than $R_2$ if every leaf of $\mathcal{P}^s$ that intersects $R_2$ intersects $R_1$. 

An \textit{edge rectangle} in $\mathcal{P}$ is a rectangle with two opposite corners on $\mathcal{S}$. A \textit{face rectangle} in $\mathcal{P}$ is a rectangle with one corner on $\mathcal{S}$ and the two opposite sides to the corner containing elements of $\mathcal{S}$ in their interior. A \textit{tetrahedron rectangle} in $\mathcal{P}$ is a rectangle all of whose sides contain elements of $\mathcal{S}$ in their interior. See \Cref{fig:rectangles}.

\begin{figure}
    \centering
    \resizebox{!}{3cm}{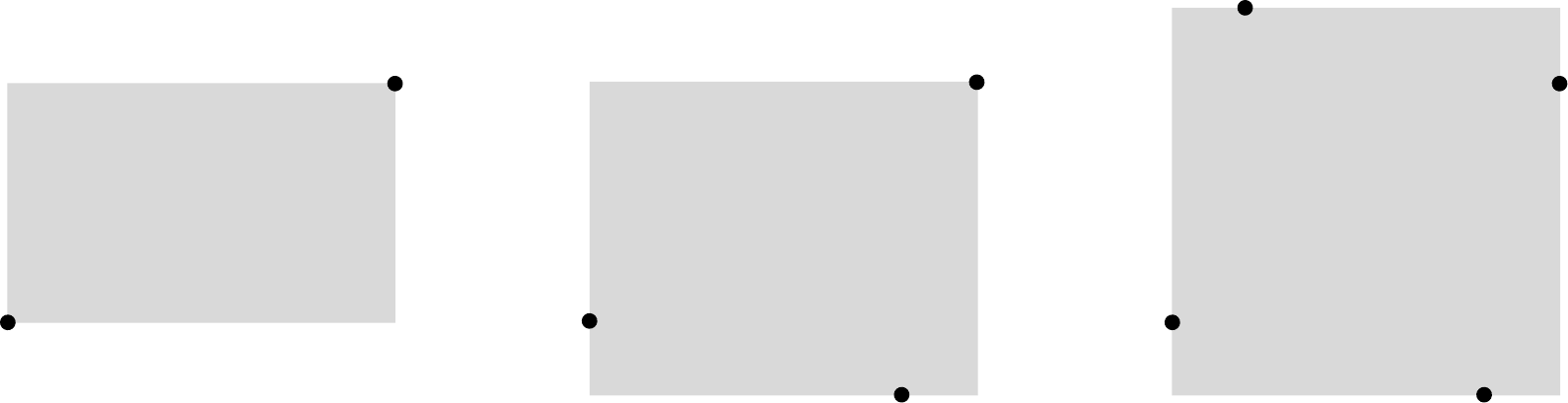}
    \caption{From left to right: an edge rectangle, a face rectangle, and a tetrahedron rectangle.}
    \label{fig:rectangles}
\end{figure}

In this paper, we will take the convention of drawing leaves of $\mathcal{P}^s$ as vertical lines and leaves of $\mathcal{P}^u$ as horizontal lines. As such, we will call the sides of a rectangle $R$ that lie along leaves of $\mathcal{P}^s$ as the \textit{vertical sides} of $R$, and the sides that lie along leaves of $\mathcal{P}^u$ as the \textit{horizontal sides} of $R$.

Note that by our assumption of no perfect fits, we could have alternatively defined a face rectangle to be a rectangle maximal with respect to the property that one corner lies on $\mathcal{S}$, and a tetrahedron rectangle to be a maximal rectangle.
\end{defn}

Recall that $M$ is oriented. This determines an orientation on $\mathcal{O}$, hence on $\mathcal{P}$, and allows us to make the following definition.

\begin{defn} \label{defn:edgerectcolor}
Let $R$ be a rectangle in $\mathcal{P}$ and let $p$ be a corner of $R$. Let $s$ be the vertical side of $R$ which contains $p$. Orient $s$ to point inwards at $p$. Similarly, let $u$ be the horizontal side of $R$ which contains $p$ and orient $u$ to point inwards at $p$. $p$ is said to be \textit{the SW or NE corner} of $R$ if (orientation of $u$, orientation of $s$) agrees with the orientation of $\mathcal{P}$ at $\gamma$. Otherwise $p$ is said to be \textit{the NW or SE corner} of $R$.

Let $R$ be an edge rectangle. We say that $R$ is \textit{red} if its SW and NE corners lie on $\mathcal{S}$. We say that $R$ is \textit{blue} if its NW and SE corners lie on $\mathcal{S}$. 
\end{defn}

\begin{defn}
Let $R$ be an edge rectangle in $\mathcal{P}$, let $p$ and $q$ be the two corners of $R$ that lie in $\mathcal{S}$. A \textit{veering diagonal} is an arc in $R$ between $p$ and $q$ that is transverse to $\mathcal{P}^s$ and $\mathcal{P}^u$.
\end{defn}

With these definitions in place, we can briefly discuss the history and recall an outline of the proof of \Cref{thm:LMT}, as a warm up to the more technical work we will do in the rest of this section. 

It was first discovered by Agol and Gu\'{e}ritaud that if $\phi$ is a pseudo-Anosov flow on an orientable closed 3-manifold $M$ without perfect fits relative to $\mathcal{C}$, then there is a veering triangulation on $M \backslash \bigcup \mathcal{C}$.

The idea of the proof is simple: Associate to each tetrahedron rectangle $R$ a taut ideal tetrahedron $t_R$. Identify the ideal vertices of $t_R$ with the elements of $\mathcal{S}$ on the sides of $R$ so that the edge rectangles contained in $R$ correspond to edges of $t_R$ and the face rectangles contained in $R$ correspond to faces of $t_R$. The upper/lower $\pi$-labelled edges of $t_R$ are the ones that correspond to edge rectangles as tall as, and as wide as $R$, respectively, and we color the edges of $t_R$ by the same color as the corresponding edge rectangle. See \Cref{fig:tetrectangle}.

\begin{figure}
    \centering
    \resizebox{!}{3.5cm}{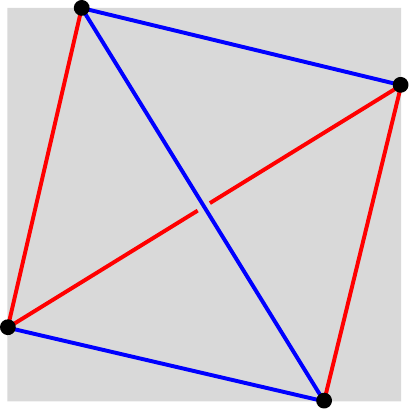}
    \caption{The boundary triangulation at a vertex of a veering triangulation.}
    \label{fig:tetrectangle}
\end{figure}

We glue along two faces of two tetrahedra whenever they correspond to the same face rectangle. This gives a 3-manifold $\widetilde{N}$ with a veering triangulation. Notice that there is a natural properly discontinuous action of $\pi_1(M^\circ)$ on $\widetilde{N}$. Let $N$ be the quotient of this action. The veering triangulation on $\widetilde{N}$ quotients down to a veering triangulation on $N$. Meanwhile, by a theorem of Waldhausen (\cite[Corollary 6.5]{Wal68}), $N$ is homeomorphic to $M \backslash \bigcup \mathcal{C}$, hence there is a veering triangulation on the latter.

An unsatisfying feature of this construction, however, is that it is unclear how the veering triangulation interacts with the flow $\phi$. This was rectified by Landry, Minsky, and Taylor in \cite{LMT21}, where they explicitly construct a homeomorphism $N \to M^\circ$ for which one can compare the image of the veering triangulation with the flow on $M$. In particular, they show that the image of the 2-skeleton is positively transverse to the flow, i.e. the flow lines are transverse to the faces of the triangulation and the flow direction agrees with the coorientation on the faces.

The idea of their construction is to first choose a veering diagonal for each edge rectangle, such that for every face rectangle $R$, the three veering diagonals in the three edge rectangles in $R$ are disjoint. This choice of veering diagonal is quite technical, and our task is to show that there is enough room in the construction so that a prescribed finite collection of edge paths can be made to be winding. Recall that this is in turn necessary for constructing the helicoidal broken transverse surfaces mentioned in the introduction.

Once this is accomplished, one defines a fibration $p:\widetilde{N} \to \mathcal{P} \backslash \mathcal{S}$ by sending the edges of the veering triangulation to the veering diagonals in the corresponding edge rectangles, then sending the faces to the regions bounded by the three veering diagonals of the corresponding face rectangles, and the tetrahedra to the regions bounded by the four outer veering diagonals of the corresponding tetrahedron rectangles. The preimages of $p$ are lines that are transverse to the 2-skeleton of the triangulation.

One then constructs a $\pi_1(M^\circ)$-equivariant map $h: \widetilde{N} \to \widetilde{M^\circ}$ such that $p = q \circ h$. In particular $h$ sends the preimages of $p$ to flow lines of $\widetilde{\phi}$, but not necessarily by homeomorphisms. Nonetheless, one can straighten out $h$ on these preimages to get a $\pi_1(M^\circ)$-equivariant homeomorphism $\widetilde{f}: \widetilde{N} \to \widetilde{M^\circ}$ such that $p = q \circ \widetilde{f}$. Quotient by $\pi_1(M^\circ)$ to get a homeomorphism $f: N \to M^\circ$. The image of the veering triangulation in $M^\circ = M \backslash \bigcup \mathcal{C}$ is the desired veering triangulation in \Cref{thm:LMT}.

\subsection{Winding edge paths} \label{subsec:edgepath}

In this subsection, we will define the technical property of winding edge paths that we want to arrange for.

\begin{defn} \label{defn:quadrant}
Let $\gamma$ be a point on $\mathcal{O}$. Recall \Cref{defn:phorbit}. There exists a neighborhood of $\gamma$ and a map from the neighborhood to $\mathbb{R}^2$ sending $\gamma$ to $0$ and sending the foliations $\mathcal{O}^s$ and $\mathcal{O}^u$ to $l^s_n$ and $l^u_n$ respectively. We call the preimage of a quadrant under the map a \textit{quadrant} at $\gamma$. 

Let $p$ be a point on $\mathcal{P}$. Suppose $p$ maps to $\gamma$ in $\mathcal{O}$. Then we call the lift of a quadrant at $\gamma$ to $p$ a \textit{quadrant} at $p$.

Recall again that $\mathcal{P}$ is oriented. If $q_1$ and $q_2$ are quadrants at $p$, we say that $q_1$ is \textit{on the left} of $q_2$ if there exists an orientation preserving local homeomorphism $\mathbb{R} \times [0,\infty)$ mapping $(0,0)$ to $p$, $(-\infty,0] \times [0,\infty)$ to $q_1$, and $[0,\infty) \times [0,\infty)$ to $q_2$. If there exists quadrants $q_1,...,q_{n+1}$ at $p$ such that $q_i$ is on the left of $q_{i+1}$ for each $i$, then we say that $q_1$ is \textit{$n$ quadrants to the left} of $q_{n+1}$.
\end{defn}

Note that if $p$ is a corner of a rectangle $R$, then $R$ determines a quadrant at $p$.

\begin{defn} \label{defn:staircase}
Let $R_1,...,R_k$ be edge rectangles such that:
\begin{itemize}
    \item There are elements $s_0,...,s_k \in \mathcal{S}$ such that $R_i$ has corners at $s_{i-1}$ and $s_i$, for each $i$.
    \item The quadrant determined by $R_{i+1}$ at $s_i$ is $2$ quadrants to the left of that determined by $R_i$, for each $i$.
\end{itemize}

Consider the set $S_k=\bigcup_{i=1}^k [i-1,i] \times [0,i]$ in $\mathbb{R}^2$. If there exists an orientation preserving embedding of $S_k$ into $\mathcal{P}$ sending $[i-1,i] \times [i-1,i]$ to $R_i$ for each $i$ and sending the foliations of $\mathbb{R}^2$ by vertical and horizontal lines to $\mathcal{P}^s$ and $\mathcal{P}^u$, we call the image of the embedding a \textit{staircase} for $R_1,..,R_k$. See \Cref{fig:staircase} left.

\begin{figure} 
    \centering
    \fontsize{14pt}{14pt}\selectfont
    \resizebox{!}{4.5cm}{
\begingroup%
  \makeatletter%
  \providecommand\color[2][]{%
    \errmessage{(Inkscape) Color is used for the text in Inkscape, but the package 'color.sty' is not loaded}%
    \renewcommand\color[2][]{}%
  }%
  \providecommand\transparent[1]{%
    \errmessage{(Inkscape) Transparency is used (non-zero) for the text in Inkscape, but the package 'transparent.sty' is not loaded}%
    \renewcommand\transparent[1]{}%
  }%
  \providecommand\rotatebox[2]{#2}%
  \newcommand*\fsize{\dimexpr\f@size pt\relax}%
  \newcommand*\lineheight[1]{\fontsize{\fsize}{#1\fsize}\selectfont}%
  \ifx\svgwidth\undefined%
    \setlength{\unitlength}{542.84997692bp}%
    \ifx\svgscale\undefined%
      \relax%
    \else%
      \setlength{\unitlength}{\unitlength * \real{\svgscale}}%
    \fi%
  \else%
    \setlength{\unitlength}{\svgwidth}%
  \fi%
  \global\let\svgwidth\undefined%
  \global\let\svgscale\undefined%
  \makeatother%
  \begin{picture}(1,0.38385025)%
    \lineheight{1}%
    \setlength\tabcolsep{0pt}%
    \put(0,0){\includegraphics[width=\unitlength,page=1]{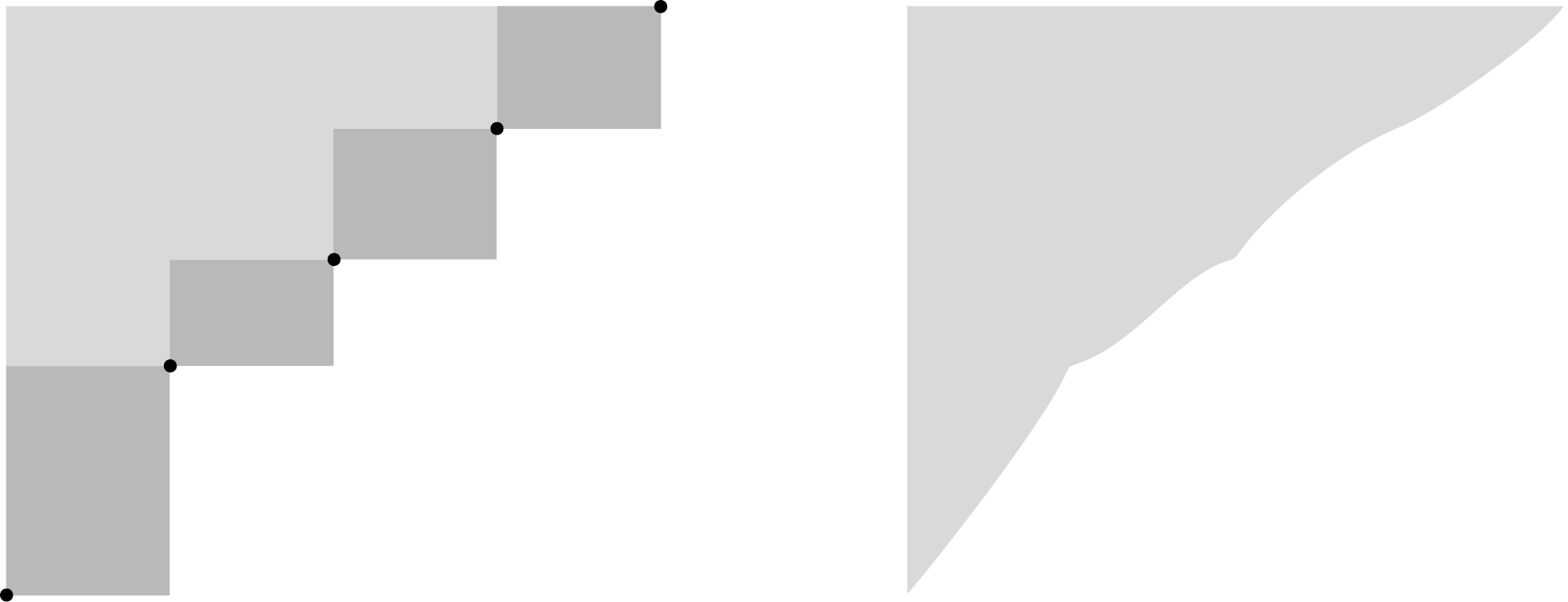}}%
    \put(0.3512037,0.33355356){\color[rgb]{0,0,0}\makebox(0,0)[lt]{\lineheight{1.25}\smash{\begin{tabular}[t]{l}$R_1$\end{tabular}}}}%
    \put(0.24676801,0.25261588){\color[rgb]{0,0,0}\makebox(0,0)[lt]{\lineheight{1.25}\smash{\begin{tabular}[t]{l}$R_2$\end{tabular}}}}%
    \put(0.14233231,0.17690002){\color[rgb]{0,0,0}\makebox(0,0)[lt]{\lineheight{1.25}\smash{\begin{tabular}[t]{l}$R_3$\end{tabular}}}}%
    \put(0.03789662,0.06985344){\color[rgb]{0,0,0}\makebox(0,0)[lt]{\lineheight{1.25}\smash{\begin{tabular}[t]{l}$R_4$\end{tabular}}}}%
    \put(0,0){\includegraphics[width=\unitlength,page=2]{staircase.pdf}}%
  \end{picture}%
\endgroup%
}
    \caption{Left: A staircase for $R_1,..,R_k$. Right: A slope for $R_1,...,R_k$.} 
    \label{fig:staircase}
\end{figure}

Note that if $k=1$, then the edge rectangle $R_1$ is a staircase for itself.

Extending the definition for rectangles, for staircases $S_1$ and $S_2$, $S_1$ is said to be \textit{taller} than $S_2$ if every leaf of $\mathcal{P}^u$ that intersects $S_2$ intersects $S_1$; $S_1$ is said to be \textit{wider} than $S_2$ if every leaf of $\mathcal{P}^s$ that intersects $S_2$ intersects $S_1$. 

Suppose we have chosen a veering diagonal $e_i$ for each $R_i$, then the complementary region of $\bigcup_i e_i$ in the staircase which contains the image of $(k,0)$ is called a \textit{slope} for $R_1,..,R_k$ with respect to the choice of veering diagonals. See \Cref{fig:staircase} right.

Note that if $k=1$, then there are two slopes for $R_1$. We will specify which one we are referring to in such a case.
\end{defn}

\begin{defn} \label{defn:edgepath}
Let $p$ be a point on $\mathcal{P}$ that is fixed by $g \in \pi_1(M^\circ)$. We define a \textit{red $g$-edge path} to be an infinite sequence of red edge rectangles $(R_i)_{i \in \mathbb{Z}}$ satisfying:
\begin{itemize}
    \item There is a collection of elements $(s_i)_{i \in \mathbb{Z}}$ of $\mathcal{S}$ such that each $R_i$ has corners at $s_{i-1}$ and $s_i$.
    \item For each $i$, we have one of the following two cases
    \begin{enumerate}
        \item The quadrants determined by $R_i$ and $R_{i+1}$ at $s_i$ are the same and $R_{i+1}$ is taller than $R_i$
        \item The quadrant determined by $R_{i+1}$ at $s_i$ is two quadrants to the left of the quadrant determined by $R_i$ 
    \end{enumerate}
    and case (1) occurs for at least one value of $i$.
    \item If $R_j, ..., R_k$ is a maximal subsequence such that $R_i$ and $R_{i+1}$ do not determine the same quadrant at $s_i$ for $j \leq i < k$, then there is a staircase $S_{j,k}$ for $R_j, ..., R_k$, such that $p$ lies in $S_{j,k}$.
    \item There exists some positive integer $P$ such that $g \cdot R_i=R_{i+P}$ for every $i$.
\end{itemize}

See \Cref{fig:edgepath}. We call $P$ the \textit{period} of the red $g$-edge path.

\begin{figure} 
    \centering
    \fontsize{16pt}{16pt}\selectfont
    \resizebox{!}{5.5cm}{
\begingroup%
  \makeatletter%
  \providecommand\color[2][]{%
    \errmessage{(Inkscape) Color is used for the text in Inkscape, but the package 'color.sty' is not loaded}%
    \renewcommand\color[2][]{}%
  }%
  \providecommand\transparent[1]{%
    \errmessage{(Inkscape) Transparency is used (non-zero) for the text in Inkscape, but the package 'transparent.sty' is not loaded}%
    \renewcommand\transparent[1]{}%
  }%
  \providecommand\rotatebox[2]{#2}%
  \newcommand*\fsize{\dimexpr\f@size pt\relax}%
  \newcommand*\lineheight[1]{\fontsize{\fsize}{#1\fsize}\selectfont}%
  \ifx\svgwidth\undefined%
    \setlength{\unitlength}{285.59540433bp}%
    \ifx\svgscale\undefined%
      \relax%
    \else%
      \setlength{\unitlength}{\unitlength * \real{\svgscale}}%
    \fi%
  \else%
    \setlength{\unitlength}{\svgwidth}%
  \fi%
  \global\let\svgwidth\undefined%
  \global\let\svgscale\undefined%
  \makeatother%
  \begin{picture}(1,0.90755962)%
    \lineheight{1}%
    \setlength\tabcolsep{0pt}%
    \put(0,0){\includegraphics[width=\unitlength,page=1]{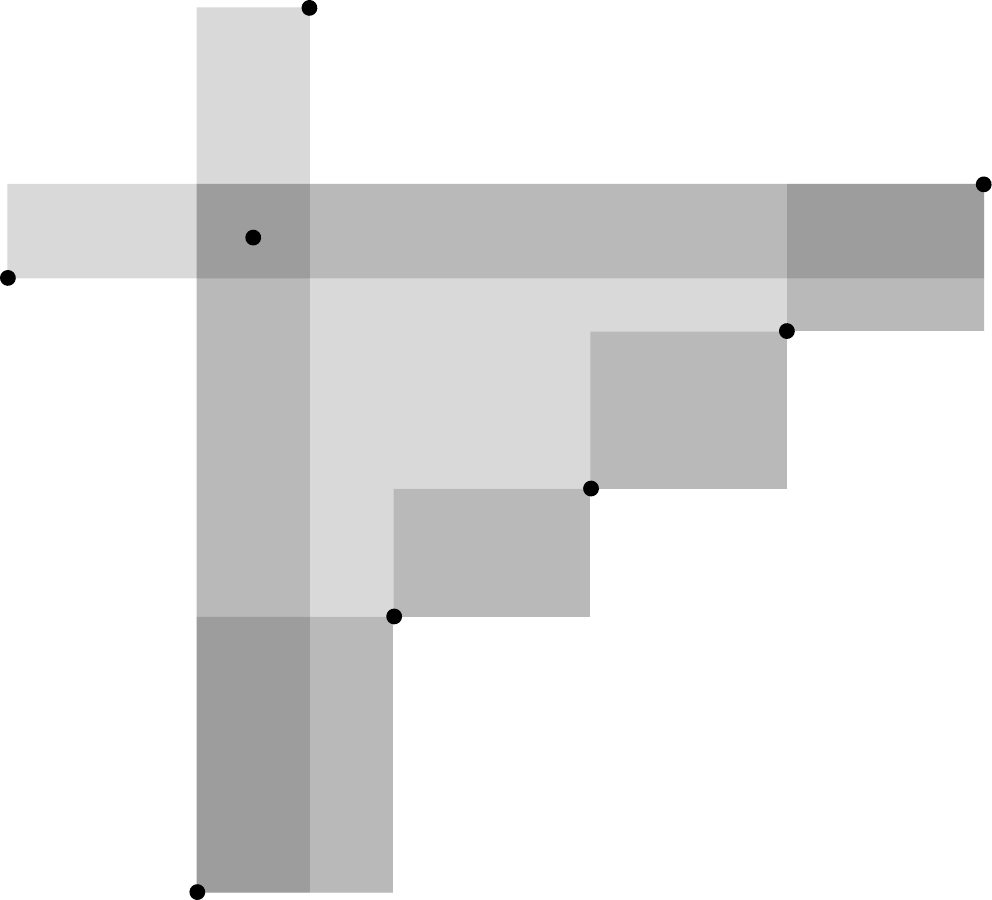}}%
    \put(0.23414441,0.60166018){\color[rgb]{0,0,0}\makebox(0,0)[lt]{\lineheight{1.25}\smash{\begin{tabular}[t]{l}$p$\end{tabular}}}}%
  \end{picture}%
\endgroup%
}
    \caption{A red $g$-edge path.} 
    \label{fig:edgepath}
\end{figure}

A \textit{blue $g$-edge path} is defined similarly. When $g$ and the color red/blue is clear from context, we will abbreviate these as \textit{edge paths}.

Intuitively, one can think of an edge path as a directed path obtained by transversing some choice of veering diagonals $e_i$ in $R_i$ in the direction of increasing $i$. Such a directed path is $g$-invariant and `winds around' $p$ with increasing height and decreasing width. The desirable case is when this intuitive picture holds.

Suppose we have chosen a veering diagonal $e_i$ for each $R_i$. Then the red $g$-edge path $(R_i)$ is said to be \textit{winding} with respect to the choice of veering diagonals if for every maximal subsequence $R_j, ..., R_k$ such that $R_i$ and $R_{i+1}$ do not determine the same quadrant at $s_i$ for $j \leq i < k$, $p$ lies in the slope for $R_j, ..., R_k$ on the side of $R_{j-1}$ and $R_{k+1}$.

See \Cref{fig:windingedgepath} left for a choice of veering diagonals for which the edge path in \Cref{fig:edgepath} is winding, and \Cref{fig:windingedgepath} right for a choice of veering diagonals for which the edge path is not winding.

Notice that the condition is automatically true when $k-j \geq 1$, since $p$ must lie outside of $R_j,...,R_k$ in this case. So one only needs to check the condition for $R_i$ where $R_{i-1}$ and $R_i$ determine the same quadrant at $s_{i-1}$ and $R_i$ and $R_{i+1}$ determine the same quadrant at $s_i$.

\begin{figure} 
    \centering
    \resizebox{!}{5cm}{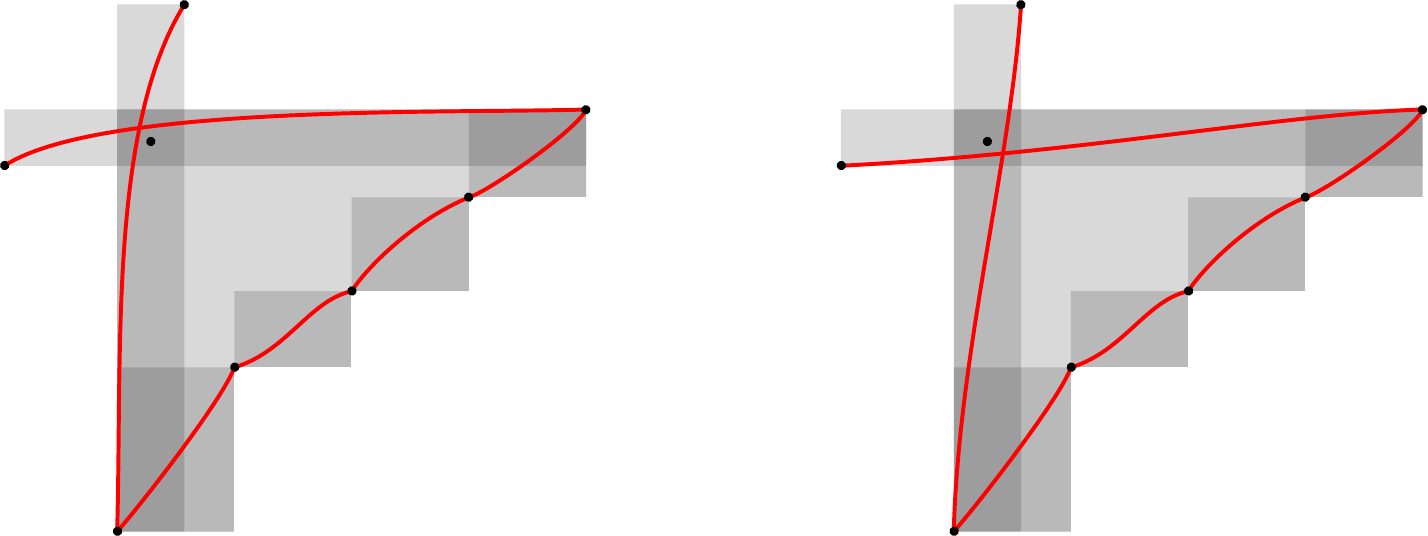}
    \caption{Left: A choice of veering diagonals for which the edge path in \Cref{fig:edgepath} is winding. Right: A choice of veering diagonals for which the edge path in \Cref{fig:edgepath} is not winding.} 
    \label{fig:windingedgepath}
\end{figure}

\end{defn}

The technical property we want is for specified edge paths to be winding, with respect to the choice of veering diagonals determined by where the edges of the veering triangulation $\widetilde{\Delta}$ project to. We are able to satisfy this technical property under an extra condition on the edge paths, which we now define.

\begin{defn}[{\cite[Section 5.1.2]{LMT21}}] \label{defn:corepoint}
Let $R$ be an edge rectangle. Set $R_0=R$. We define $R_i$ for $i \geq 1$ inductively in the following way: There exists a unique tetrahedron rectangle $Q_i$ whose two elements of $\mathcal{S}$ on its vertical sides are the corners of $R_{i-1}$. Consider the sub-rectangle of $Q_i$ which has corners at the elements of $\mathcal{S}$ on its horizontal sides. This is an edge rectangle which we call $R_i$. Similarly, for each $i \leq -1$, there exists a unique tetrahedron rectangle $Q_{i+1}$ whose two elements of $\mathcal{S}$ on its horizontal sides are corners of $R_{i+1}$. Consider the sub-rectangle of $Q_{i+1}$ which has corners at the elements of $\mathcal{S}$ on its vertical sides. This is an edge rectangle which we call $R_i$.

We call the bi-infinite sequence $(R_i)$ the \textit{core sequence} of $R$. The intersection $\bigcap_{i=-\infty}^\infty R_i$ is a single point, which we call the \textit{core point} of $R$ and denote as $c(R)$. See \Cref{fig:corepoint}.

\begin{figure} 
    \centering
    \fontsize{16pt}{16pt}\selectfont
    \resizebox{!}{6cm}{
\begingroup%
  \makeatletter%
  \providecommand\color[2][]{%
    \errmessage{(Inkscape) Color is used for the text in Inkscape, but the package 'color.sty' is not loaded}%
    \renewcommand\color[2][]{}%
  }%
  \providecommand\transparent[1]{%
    \errmessage{(Inkscape) Transparency is used (non-zero) for the text in Inkscape, but the package 'transparent.sty' is not loaded}%
    \renewcommand\transparent[1]{}%
  }%
  \providecommand\rotatebox[2]{#2}%
  \newcommand*\fsize{\dimexpr\f@size pt\relax}%
  \newcommand*\lineheight[1]{\fontsize{\fsize}{#1\fsize}\selectfont}%
  \ifx\svgwidth\undefined%
    \setlength{\unitlength}{451.35609041bp}%
    \ifx\svgscale\undefined%
      \relax%
    \else%
      \setlength{\unitlength}{\unitlength * \real{\svgscale}}%
    \fi%
  \else%
    \setlength{\unitlength}{\svgwidth}%
  \fi%
  \global\let\svgwidth\undefined%
  \global\let\svgscale\undefined%
  \makeatother%
  \begin{picture}(1,0.66429981)%
    \lineheight{1}%
    \setlength\tabcolsep{0pt}%
    \put(0,0){\includegraphics[width=\unitlength,page=1]{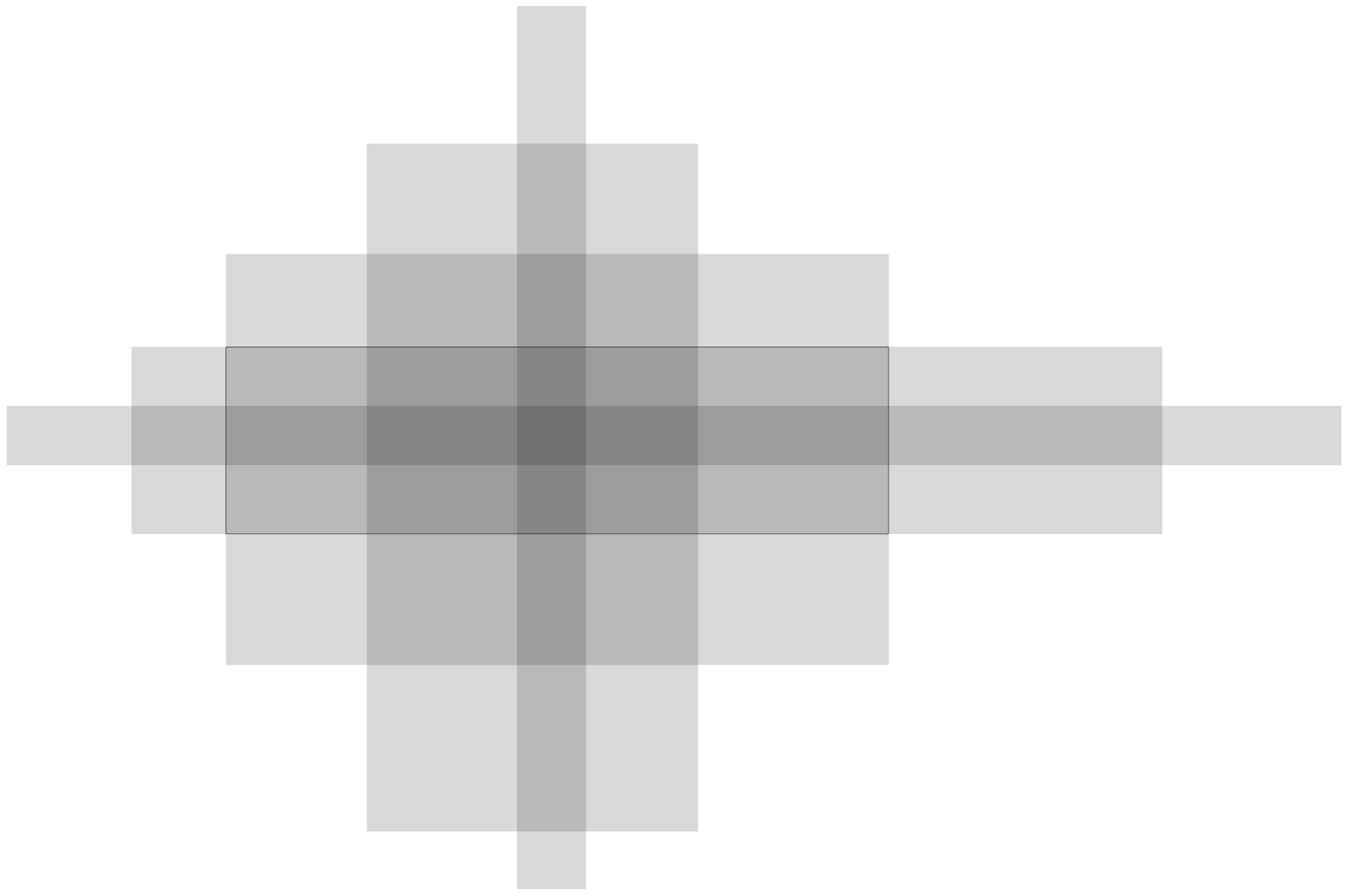}}%
    \put(0.59695626,0.23074079){\color[rgb]{0,0,0}\makebox(0,0)[lt]{\lineheight{1.25}\smash{\begin{tabular}[t]{l}$R$\end{tabular}}}}%
    \put(0,0){\includegraphics[width=\unitlength,page=2]{corepoint.pdf}}%
    \put(0.38334409,0.29480443){\color[rgb]{0,0,0}\makebox(0,0)[lt]{\lineheight{1.25}\smash{\begin{tabular}[t]{l}$c(R)$\end{tabular}}}}%
    \put(0,0){\includegraphics[width=\unitlength,page=3]{corepoint.pdf}}%
  \end{picture}%
\endgroup%
}
    \caption{The core sequence and core point of an edge rectangle $R$.} 
    \label{fig:corepoint}
\end{figure}

\end{defn}

\begin{defn} \label{defn:niceedgepath}
A red $g$-edge path is said to be \textit{nice} if whenever $R_{i-1}$ and $R_i$ determine the same quadrant at $s_{i-1}$ and $R_i$ and $R_{i+1}$ determine the same quadrant at $s_i$, $p$ lies closer to the vertical side of $R_i$ containing $s_i$ than the core point of $R_i$.

A nice blue $g$-edge path is defined similarly.
\end{defn}

Let $\gamma$ be a closed orbit of $\phi$ that is not an element of $\mathcal{C}$. Suppose for every $p \in \mathcal{P}$ that corresponds to an orbit of $\widehat{\phi}$ covering $\gamma$ and which is fixed by $h [\gamma] h^{-1} \in \pi_1(M^\circ)$, we are given a nice red $h [\gamma] h^{-1}$-edge path $(R_{i,h})$. If these nice red edge paths are $\pi_1(M^\circ)$-equivariant, or more preicisely, if $R_{i,gh}=g \cdot R_{i,h}$ for every $g \in \pi_1(M^\circ)$, then we call this collection a \textit{nice red $\gamma$-edge path}. A \textit{nice blue $\gamma$-edge path} is defined similarly.

We can finally state our result:

\begin{prop} \label{prop:windingedgepath}
In the setting of \Cref{thm:LMT}, suppose we are given a finite collection of closed orbits of $\phi$ which is disjoint from $\mathcal{C}$, and a nice $\gamma$-edge path for every $\gamma$ in the collection. Then there exists a veering triangulation $\Delta$ on $M \backslash \bigcup \mathcal{C}$ with 2-skeleton positively transverse to the flow $\phi$ and such that each given edge path is winding, with respect to the choice of veering diagonals determined by where the edges of $\widetilde{\Delta}$ project to.
\end{prop}

\subsection{Proof of \Cref{prop:windingedgepath}} \label{subsec:proofofwindingedgepath}

As remarked before, we only need to modify the first part of the proof in \cite{LMT21}, namely their Proposition 5.2. So we will only explain in depth the parts that we need to modify, and refer the reader to \cite{LMT21} for details concerning the rest of the proof.

We recall some of the definitions found in \cite{LMT21}.

\begin{defn} \label{defn:anchorsystem}
An \textit{anchor system} is a pair $(A,\alpha)$ where $\alpha$ is a bijection from the set of edge rectangles onto a subset $A \subset \mathcal{P}$ satisfying:
\begin{itemize}
    \item For each edge rectangle $R$, $\alpha(R)$ lies in the interior of $R$.
    \item $g \cdot \alpha(R) = \alpha(g \cdot R)$ for each edge rectangle $R$ and each $g \in \pi_1(M^\circ)$
    \item For edge rectangles $R_1$ and $R_2$ sharing a corner $s \in \mathcal{S}$ and determining the same quadrant at $s$, if $R_1$ is wider than $R_2$, then the rectangle with corners at $s$ and $\alpha(R_1)$ is wider and no taller than the rectangle with corners at $s$ and $\alpha(R_2)$.
\end{itemize}

Let $(A,\alpha)$ be an anchor system. Let $F$ be a face rectangle. Let $s$ be the corner of $F$ that lies in $\mathcal{S}$, let $x$ be the element of $\mathcal{S}$ that lies in the interior of a vertical side of $F$, and let $y$ be the element of $\mathcal{S}$ that lies in the interior of a horizontal side of $F$. Let $a_x$ be the image of the edge rectangle with corners at $s$ and $x$ under $\alpha$, and $a_y$ be the image of the edge rectangle with corners at $s$ and $y$ under $\alpha$. If the rectangle $R$ with corners at $x$ and $a_x$ intersects the rectangle $Q$ of $F$ with corners at $y$ and $a_y$, then $F$ is said to be \textit{busy}. If $F$ is busy, let $R'$ be the maximal sub-rectangle of $R$ with the property that the stable and unstable leaves through each point in $R'$ do not intersect the interior of $Q$. A \textit{$F$-buoy} is a point in $R'$ which corresponds to an orbit of $\widehat{\phi}$ covering a closed orbit of $\phi$.
\end{defn}

\begin{defn} \label{defn:pinched}
Let $s$ be an element of $\mathcal{S}$, and let $q$ be a quadrant at $\gamma$. A \textit{$s$-staircase} is the union of all edge rectangles that have a corner at $s$ and determine the quadrant $q$. Generally, a \textit{$\mathcal{S}$-staircase} is a $s$-staircase for some $s \in \mathcal{S}$. Notice that these were simply called staircases in \cite{LMT21} but in this paper we need to distinguish them from the objects in \Cref{defn:staircase}. 

\begin{figure} 
    \centering
    \fontsize{14pt}{14pt}\selectfont
    \resizebox{!}{5.5cm}{
\begingroup%
  \makeatletter%
  \providecommand\color[2][]{%
    \errmessage{(Inkscape) Color is used for the text in Inkscape, but the package 'color.sty' is not loaded}%
    \renewcommand\color[2][]{}%
  }%
  \providecommand\transparent[1]{%
    \errmessage{(Inkscape) Transparency is used (non-zero) for the text in Inkscape, but the package 'transparent.sty' is not loaded}%
    \renewcommand\transparent[1]{}%
  }%
  \providecommand\rotatebox[2]{#2}%
  \newcommand*\fsize{\dimexpr\f@size pt\relax}%
  \newcommand*\lineheight[1]{\fontsize{\fsize}{#1\fsize}\selectfont}%
  \ifx\svgwidth\undefined%
    \setlength{\unitlength}{244.62448982bp}%
    \ifx\svgscale\undefined%
      \relax%
    \else%
      \setlength{\unitlength}{\unitlength * \real{\svgscale}}%
    \fi%
  \else%
    \setlength{\unitlength}{\svgwidth}%
  \fi%
  \global\let\svgwidth\undefined%
  \global\let\svgscale\undefined%
  \makeatother%
  \begin{picture}(1,0.86102678)%
    \lineheight{1}%
    \setlength\tabcolsep{0pt}%
    \put(0,0){\includegraphics[width=\unitlength,page=1]{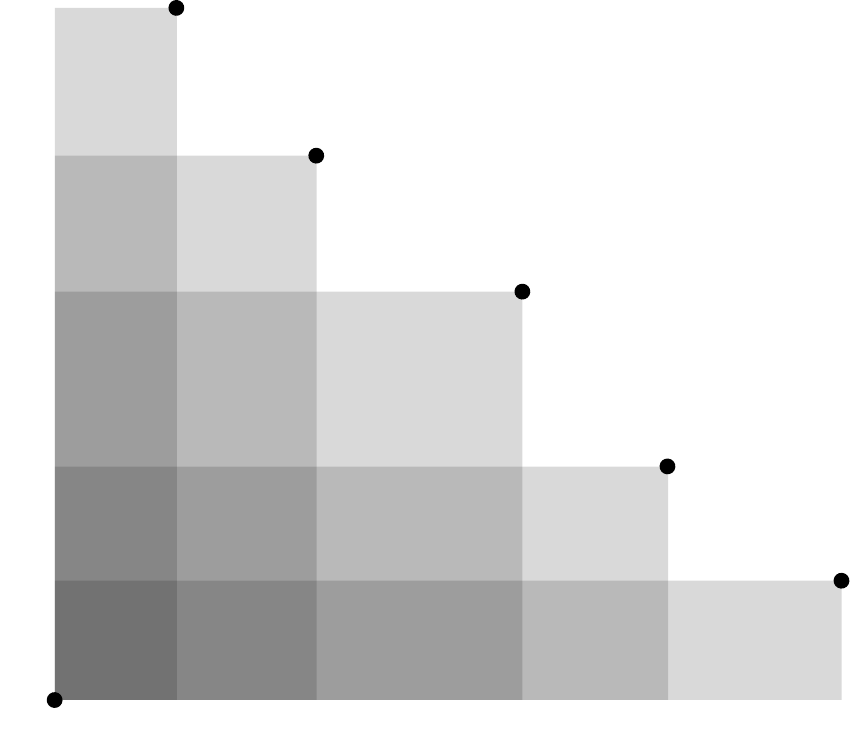}}%
    \put(-0.00556896,0.01297029){\color[rgb]{0,0,0}\makebox(0,0)[lt]{\lineheight{1.25}\smash{\begin{tabular}[t]{l}$s$\end{tabular}}}}%
  \end{picture}%
\endgroup%
}
    \caption{A $s$-staircase.} 
    \label{fig:LMTstaircase}
\end{figure}

An edge rectangle $R$ is \textit{pinched} if there exists another edge rectangle $Q$ such that $R$ and $Q$ lie in the same $\mathcal{S}$-staircase and $c(R)=c(Q)$. A core point $c$ is \textit{pinched} if it is the core point of a pinched edge rectangle.

The \textit{preimage} of a core point $c$ is the union of edge rectangles whose core point is $c$.
\end{defn}

\begin{defn} \label{defn:corebox}
A family of rectangles $\{b(R):\text{edge rectangles $R$}\}$ is a choice of \textit{core boxes} if it satisfies the following properties.
\begin{enumerate}
    \item For every edge rectangle $R$, $b(R)$ lies in the interior of $R$ and contains the core point of $R$.
    \item If $R_1$ and $R_2$ are edge rectangles that lie in the same $\mathcal{S}$-staircase with distinct core points and $R_1$ is taller than $R_2$, then for any points $x_i \in b(R_i)$, the rectangle $R'_1$ with corners at $s$ and $x_1$ is strictly taller than the rectangle $R'_2$ with corners at $s$ and $x_2$, and $R'_2$ is strictly wider than $R'_1$.
    \item $b(g \cdot R) = g \cdot b(R)$ for all edge rectangles $R$ and $g \in \pi_1(M^\circ)$
\end{enumerate}
\end{defn}

We follow \cite{LMT21} to construct an anchor system. First construct a choice of core boxes: From each $\pi_1(M^\circ)$-orbit of $\mathcal{S}$-staircases, choose a particular $s$-staircase $S$, then choose \textit{preliminary boxes} $b_S(R)$ for edge rectangles $R \subset S$ satisfying:
\begin{enumerate}
    \item For every edge rectangle $R$, $b_S(R)$ lies in the interior of $R$ and contains the core point of $R$.
    \item If $R_1$ and $R_2$ are edge rectangles that lie in $S$ with distinct core points and $R_1$ is taller than $R_2$, then for any points $x_i \in b(R_i)$, the rectangle $R'_1$ with corners at $s$ and $x_1$ is strictly taller than the rectangle $R'_2$ with corners at $s$ and $x_2$, and $R'_2$ is strictly wider than $R'_1$.
    \item $b_S(g \cdot R) = g \cdot b_S(R)$ for all edge rectangles $R$ in $S$ and all $g \in \pi_1(M^\circ)$ that preserve $S$.
\end{enumerate}
Then define $b_{g \cdot S}(g \cdot R) = g \cdot b_S(R)$ for each $g \in \pi_1(M^\circ)$. Finally define $b(R)=b_{S_1}(R) \cap b_{S_2}(R)$ for the two $\mathcal{S}$-staircases $S_1$ and $S_2$ in which $R$ lies in.

To proceed, we need the following fact:

\begin{lemma}[{\cite[Claim 5.8]{LMT21}}] \label{lemma:LMTpinchanchor}
Let $c$ be a pinched core point, let $g$ be the primitive element of $\pi_1(M)$ that preserves $c$, and let $P_c$ be the preimage of $c$. Then for any $\lambda>1$, there exists an embedding $\Psi_{c,\lambda}:P_c \to \mathbb{R}^2$ that conjugates the action of $g$ with $\phi_{2,k,\lambda}$ for some $k$.
\end{lemma}

From each $\pi_1(M^\circ)$-orbit of pinched core points, choose a particular one $c$. Apply \Cref{lemma:LMTpinchanchor} to obtain a $\Psi_{c,\lambda}$, then define $\Psi_{g \cdot c,\lambda} = g \cdot \Psi_{c,\lambda}$ for every $g \in \pi_1(M^\circ)$. For each pinched edge rectangle $R$, in $\Psi_{c(R),\lambda}$ coordinates, draw the straight line between the corners of $R$ that lie in $\mathcal{S}$, and let the point where this straight line passes through the $x$-axis be $a_\lambda(R)$. 

It is argued in \cite[Claim 5.9]{LMT21} that for small enough $\lambda$, $a_\lambda(R) \in b(R)$ all edge rectangles $R$. Fix such a $\lambda$. Define $\alpha(R)=c(R)$ for non-pinched $R$ and $\alpha(R)=a_\lambda(R)$ for pinched $R$. \cite[Claim 5.10]{LMT21} shows that this defines an anchor system.

In this construction, notice that $b_S(R)$, hence $b(R)$, can be chosen to be arbitrarily small, and so $a_\lambda(R)$, hence $\alpha(R)$, can be chosen to be arbitrarily close to $c(R)$ for each edge rectangle $R$. In particular, we may assume that for a given $g$-edge path $(R_i)$, if $R_{i-1}$ and $R_i$ lie in the same quadrant at $s_{i-1}$ and $R_i$ and $R_{i+1}$ lie in the same quadrant at $s_i$, then $p$ lies closer to the vertical side of $R_i$ containing $s_i$ than $\alpha(R_i)$.

Now with respect to this choice of anchor system, choose a $F$-buoy for each busy face $F$. Let $B$ be the collection of all $F$-buoys and all points that correspond to an orbit of $\widehat{\phi}$ covering a closed orbit of $\phi$ in the given collection. Notice that $B$ is discrete.

For every $\mathcal{S}$-staircase $S$, say $S$ is a $s$-staircase, let $g$ be a primitive element of $\pi_1(M^\circ)$ that preserves $S$. Choose a $g$–equivariant family of paths from $s$ to the anchors of the edge rectangles in $S$ with the following three properties:
\begin{enumerate}
    \item For each edge rectangle $R \subset S$, let $Q$ be the sub-rectangle with corners at $s$ and $\alpha(R)$. The path from $s$ to $\alpha(R)$ is homotopic rel endpoints to the first-horizontal-then-vertical path in $Q \backslash B$.
    \item The paths are disjoint except at $s$
    \item The paths are transverse to $\mathcal{P}^s$ and $\mathcal{P}^u$
\end{enumerate}

See \Cref{fig:anchorsystem}, which is a reproduction of \cite[Figure 21]{LMT21}, for an illustration of what these paths look like.

\begin{figure} 
    \centering
    \fontsize{14pt}{14pt}\selectfont
    \resizebox{!}{5.5cm}{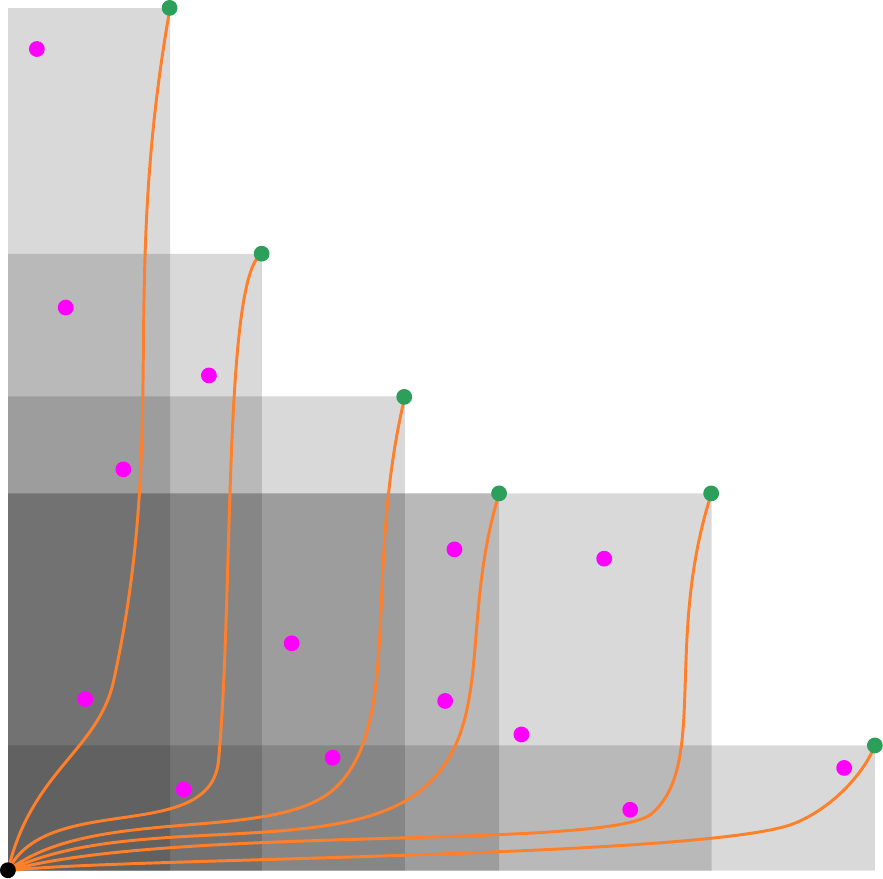}
    \caption{Using the anchor system (green) and the points in $B$ (pink) to choose paths (orange) that will form the veering diagonals. This figure is a reproduction of \cite[Figure 21]{LMT21}.} 
    \label{fig:anchorsystem}
\end{figure}

Construct the veering diagonal for each edge rectangle $R$ by concatenating the paths from each of the corners at $\mathcal{S}$ to $\alpha(R)$. In \cite[Lemma 5.3]{LMT21}, it is argued that the three veering diagonals of every face rectangle have disjoint interiors. We further claim that each of the given edge paths is winding with respect to this choice of veering diagonals.

To see this, fix a given red $g$-edge path $(R_i)$. If $R_{i-1}$ and $R_i$ determine the same quadrant at $s_{i-1}$ and $R_i$ and $R_{i+1}$ determine the same quadrant at $s_i$, then $p \in B$ lies closer to the vertical side of $R_i$ containing $s_i$ than $\alpha(R_i)$. Therefore by property (1) above, $p$ lies in the slope for $R_i$ to the side of $R_{i-1}$ and $R_{i+1}$.

The rest of the proof proceeds exactly as in \cite{LMT21}, as outlined in \Cref{subsec:rectangles}. The edges of the lift of the constructed triangulation $\widetilde{\Delta}$ project down to the chosen veering diagonals, hence each of the given edge paths are winding.

\section{Closed orbits of the flow} \label{sec:closedorbits}

One of the advantages to studying pseudo-Anosov flows using veering triangulations is that one can encode orbits of the flow using discrete, combinatorial objects coming from the triangulation. This is the idea we will be exploring in this section. 

We define the dual graph and the flow graph of a veering triangulation. In the setting of \Cref{thm:LMT}, cycles of these encode the closed orbits of $\phi$. This was described in \cite{LMT21}, but we will recall how this works in \Cref{subsec:dualgraphflowgraph}. Using this fact, we can define the complexity of a closed orbit according to how long of a flow graph cycle we need to use to encode it. This complexity then provides a bound on the number of interactions between the closed orbit and objects of the triangulation, which we explain in \Cref{subsec:orbitcomplexity}.

\subsection{Dual graph and flow graph} \label{subsec:dualgraphflowgraph}

\begin{defn} \label{defn:dualgraph}


Let $\Delta$ be a veering triangulation on a 3-manifold $M$. We define the \textit{stable branched surface} $B$ of $\Delta$: As a $2$-complex, $B$ is the dual cell complex to $\Delta$. The branched surface structure is then determined by declaring that within each tetrahedra, $B$ contains a smooth quadrilateral with vertices on the top and bottom edges and the two side edges of the same color as the top edge. See \Cref{fig:branchsurf} left.

The $2$-cells of $B$ are known as the \textit{sectors}, while the $1$-skeleton of $B$ is known as the \textit{branch locus}. We orient the 1-cells of $B$ to be positively transverse to the faces of $\Delta$. Then the 1-skeleton of $B$ becomes a directed graph embedded in $M$, which we call the \textit{dual graph} of $\Delta$ and denote by $\Gamma$.

Suppose $c$ is a directed path of $\Gamma$, then at a vertex $v$ of $c$, we say that $c$ takes a \textit{branching turn} at $v$ if it can be realized by a smooth arc on $B$ near $v$, otherwise we say that $c$ takes an \textit{anti-branching turn} at $v$. A cycle of $\Gamma$ that only takes branching turns is called a \textit{branch cycle}. A cycle of $\Gamma$ that only takes anti-branching turns is called an \textit{AB cycle}.
\end{defn}

\begin{defn}[{\cite{LMT20}}] \label{defn:flowgraph}
Let $\Delta$ be a veering triangulation on a 3-manifold $M$. Define the \textit{flow graph} $\Phi$ to be a directed graph with the set of vertices equals to the set of edges of $\Delta$, and adding 3 edges for each tetrahedron, going from the bottom edge to the top edge and the two side edges of opposite color to the top edge.

$\Phi$ can be naturally embedded in the stable branched surface $B$, hence in $M$, by placing each vertex at the top corner of the sector of $B$ that its corresponding edge of $\Delta$ meets, and placing the edges that enter that vertex within that sector of $B$. See \Cref{fig:branchsurf} right. We will always consider the flow graph to be embedded in $M$ in this way.
\end{defn}

\begin{figure} 
    \centering
    \resizebox{!}{3.5cm}{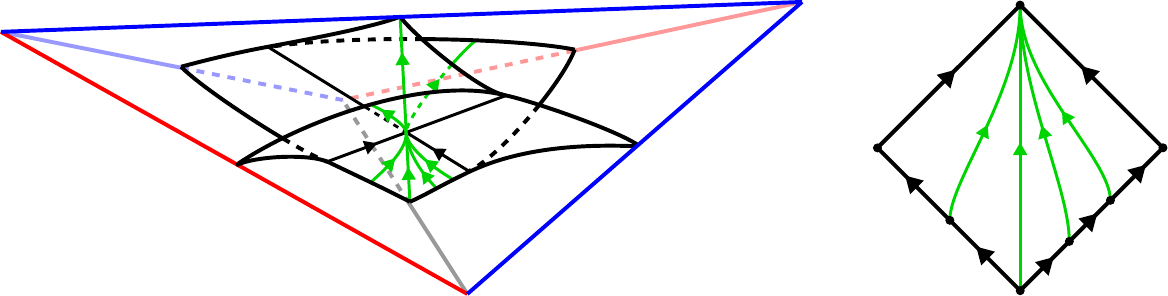}
    \caption{Left: The portion of the stable branched surface and the flow graph within each tetrahedron. Right: The portion of the flow graph on each sector of the stable branched surface.}
    \label{fig:branchsurf}
\end{figure}

We recall the notion of a dynamic plane, which was introduced in \cite{LMT21}.

\begin{defn} \label{defn:dynamicplane}
Let $\Delta$ be a veering triangulation on a 3-manifold $M$. Lift the triangulation, its stable branched surface $B$, its dual graph $\Gamma$ and its flow graph $\Phi$ to the universal cover $\widetilde{M}$ to get $\widetilde{\Delta}$, $\widetilde{B}$, $\widetilde{\Gamma}$, and $\widetilde{\Phi}$ respectively. 

A \textit{descending path} on $\widetilde{B}$ is a path that intersects the branch locus of $\widetilde{B}$ transversely and induces the maw coorientation at each intersection, that is, it goes from a side with more sectors to a side with less sectors. Let $x$ be a point on $\widetilde{B}$. The \textit{descending set} of $x$, denoted by $\Delta(x)$, is the set of points on $\widetilde{B}$ that can be reached from $x$ via a descending path. By \cite[Lemma 3.1]{LMT21}, each descending set is a union of sectors that forms a quarter-plane.

Now let $c$ be a cycle of the dual graph $\Gamma$ that is not a branch cycle. Lift $c$ to a bi-infinite path $\widetilde{c}$ of $\widetilde{\Gamma}$. The \textit{dynamic plane} associated to $\widetilde{c}$, denoted by $D(\widetilde{c})$, is the set of points on $\widetilde{B}$ that can be reached from a point on $\widetilde{c}$ via a descending path. A dynamic plane is a union of sectors that forms a plane. In fact, if the vertices of $\widetilde{c}$ are $(v_i)_{i \in \mathbb{Z}}$, then $D(\gamma) = \bigcup_i \Delta(v_i)$. 

Consider the restriction of $\widetilde{\Phi}$ to a dynamic plane $D$. This is an oriented train track with only converging switches. In particular the forward $\widetilde{\Phi}$-path starting at any given point $x \in \widetilde{\Phi}$ is well-defined. 

In \Cref{fig:dynamicplane}, which is a reproduction of \cite[Figure 6]{LMT21}, we illustrate a descending set in a dynamic plane, and the restriction of $\widetilde{\Phi}$ to it.

\begin{figure} 
    \centering
    \fontsize{14pt}{14pt}\selectfont
    \resizebox{!}{5.5cm}{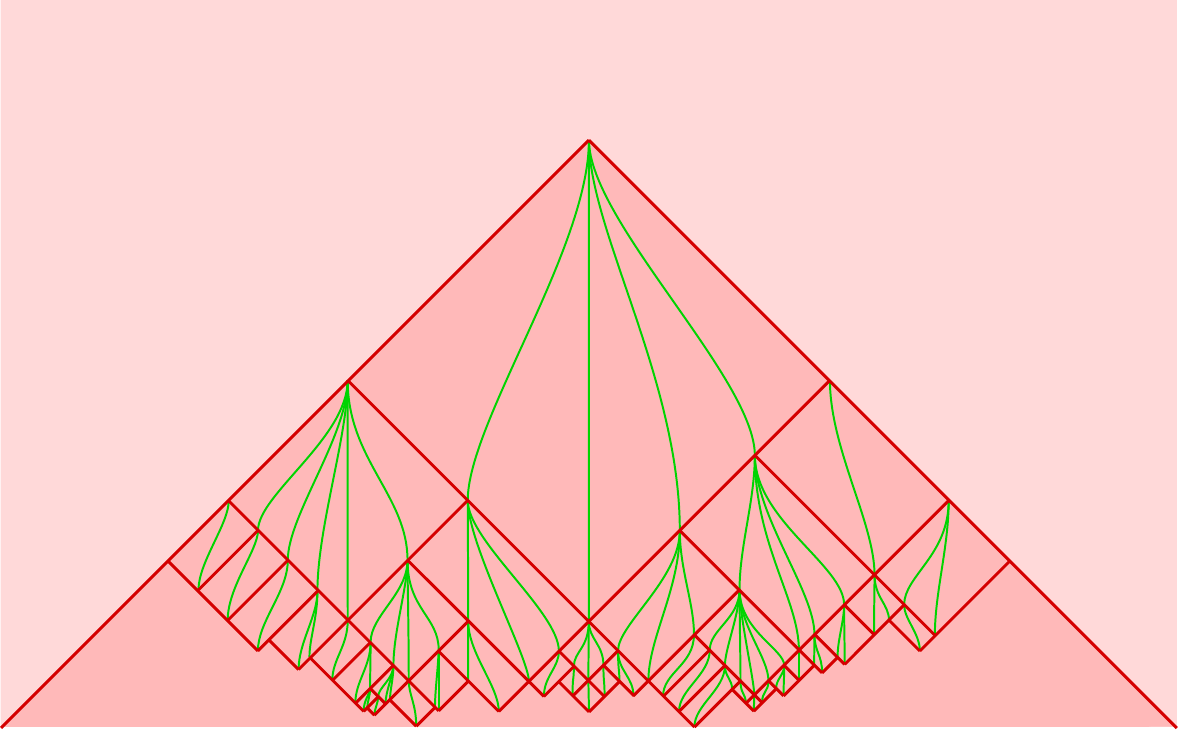}
    \caption{A descending set in a dynamic plane and the restriction of $\widetilde{\Phi}$ (green). This figure is a reproduction of \cite[Figure 6]{LMT21}.} 
    \label{fig:dynamicplane}
\end{figure}

Consider the case when $c$ is an AB cycle. Let $\widetilde{c}$ be a lift of $c$. Then there are two $\widetilde{\Phi}$-paths on the dynamic plane $D(\widetilde{c})$ for which every vertex of $\widetilde{c}$ lies along. The region bounded by the two $\widetilde{\Phi}$-paths is called an \textit{AB strip}. By \cite[Proposition 3.10]{LMT21}, for a fixed dynamic plane $D$, all the AB strips, if any, must be adjacent. We call the union of them the \textit{AB region} of $D$.

In \Cref{fig:ABregion}, which is a reproduction of \cite[Figure 6]{LMT21}, we illustrate a case where the AB region of a dynamic plane contains two AB strips.

\begin{figure} 
    \centering
    \fontsize{10pt}{10pt}\selectfont
    \resizebox{!}{5.5cm}{
\begingroup%
  \makeatletter%
  \providecommand\color[2][]{%
    \errmessage{(Inkscape) Color is used for the text in Inkscape, but the package 'color.sty' is not loaded}%
    \renewcommand\color[2][]{}%
  }%
  \providecommand\transparent[1]{%
    \errmessage{(Inkscape) Transparency is used (non-zero) for the text in Inkscape, but the package 'transparent.sty' is not loaded}%
    \renewcommand\transparent[1]{}%
  }%
  \providecommand\rotatebox[2]{#2}%
  \newcommand*\fsize{\dimexpr\f@size pt\relax}%
  \newcommand*\lineheight[1]{\fontsize{\fsize}{#1\fsize}\selectfont}%
  \ifx\svgwidth\undefined%
    \setlength{\unitlength}{270.41778062bp}%
    \ifx\svgscale\undefined%
      \relax%
    \else%
      \setlength{\unitlength}{\unitlength * \real{\svgscale}}%
    \fi%
  \else%
    \setlength{\unitlength}{\svgwidth}%
  \fi%
  \global\let\svgwidth\undefined%
  \global\let\svgscale\undefined%
  \makeatother%
  \begin{picture}(1,0.57554375)%
    \lineheight{1}%
    \setlength\tabcolsep{0pt}%
    \put(0,0){\includegraphics[width=\unitlength,page=1]{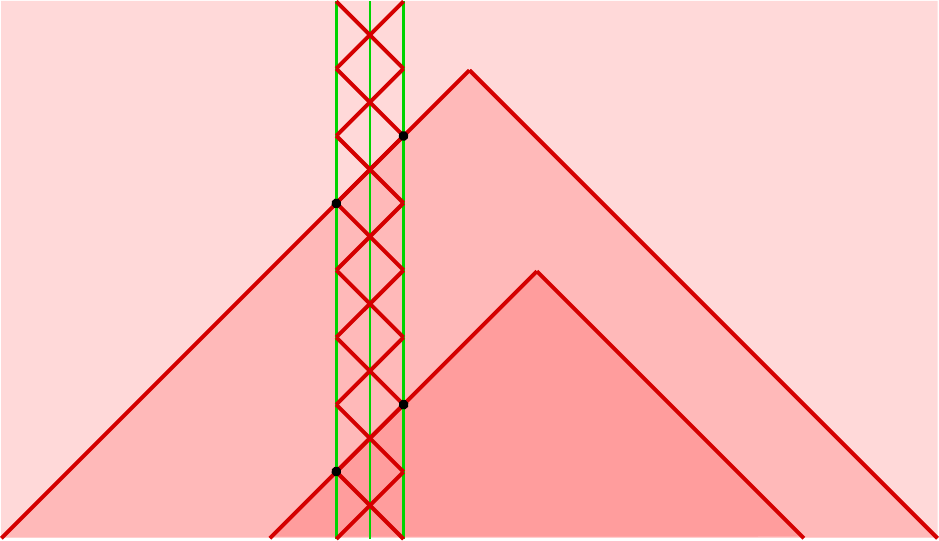}}%
    \put(0.30598609,0.07456086){\color[rgb]{0,0,0}\makebox(0,0)[lt]{\lineheight{1.25}\smash{\begin{tabular}[t]{l}$x_1$\end{tabular}}}}%
    \put(0.30640835,0.36089565){\color[rgb]{0,0,0}\makebox(0,0)[lt]{\lineheight{1.25}\smash{\begin{tabular}[t]{l}$y_1$\end{tabular}}}}%
    \put(0.43738514,0.39451537){\color[rgb]{0,0,0}\makebox(0,0)[lt]{\lineheight{1.25}\smash{\begin{tabular}[t]{l}$y_2$\end{tabular}}}}%
    \put(0.43598438,0.10524562){\color[rgb]{0,0,0}\makebox(0,0)[lt]{\lineheight{1.25}\smash{\begin{tabular}[t]{l}$x_2$\end{tabular}}}}%
  \end{picture}%
\endgroup%
}
    \caption{A dynamic plane containing two AB strips in its AB region. The $\widetilde{\Phi}$-paths starting at $x_1$ and $x_2$ never converge. This figure is a reproduction of \cite[Figure 10]{LMT21}.} 
    \label{fig:ABregion}
\end{figure}
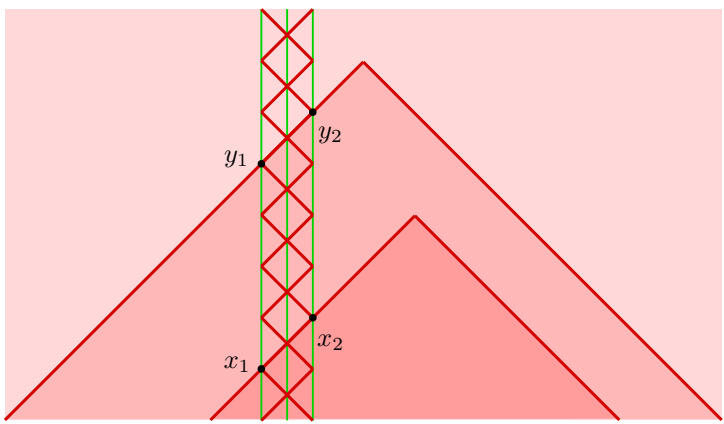

\end{defn}

As promised, we will explain how these objects can be used to study closed orbits of pseudo-Anosov flows. For the rest of this section, we fix the following setting: Let $\phi$ be a pseudo-Anosov flow on an oriented closed 3-manifold $M$ without perfect fits relative to $\mathcal{C}$. Let $\Delta$ be a veering triangulation associated to $\phi$ on $M^\circ = M \backslash \bigcup \mathcal{C}$. We will use the notation in \Cref{sec:LMT}. 

Let $B$, $\Gamma$, $\Phi$ be the stable branched surface, the dual graph, and the flow graph of $\Delta$. Let $\widetilde{\Delta}$, $\widetilde{B}$, $\widetilde{\Gamma}$, and $\widetilde{\Phi}$ be the lift of the corresponding objects to the universal cover $\widetilde{M^\circ}$.

Let $O$ be the set of closed orbits of $\phi$, let $Z_\Gamma$ be the set of cycles of the dual graph $\Gamma$, and let $Z_\Phi$ be the set of cycles of the flow graph $\Phi$. 

\begin{rmk} \label{rmk:primitiveorbits}
In \cite{LMT21}, closed orbits include non-primitive orbits and cycles include non-primitive cycles. We will follow their convention in this section since it is more convenient for the discussion. However, to be consistent with this, we will temporarily abuse notation and include all the multiples of the orbits in $\mathcal{C}$ inside $\mathcal{C}$ as well.
\end{rmk}

Define a map $F_\Gamma : Z_\Gamma \to O$ as follows: Given a cycle $c$ of the dual graph $\Gamma$, let $g=[c] \in \pi_1(M^\circ)$, and let $\widetilde{c}$ be the lift of $c$ that is preserved by $g$. $\widetilde{c}$ is a bi-infinite path in $\widetilde{\Gamma}$. Suppose the vertices of $\widetilde{c}$ are, in order, $t_i$, for $i \in \mathbb{Z}$. Let $R_i$ be the tetrahedron rectangle in $\mathcal{P}$ corresponding to the tetrahedron of $\widetilde{\Delta}$ dual to $t_i$. Then since for each $i$, $R_{i+1}$ is taller than $R_i$ and $R_i$ is wider than $R_{i+1}$, the intersection $\bigcap_i R_i$ is a single point that is invariant under $g$. This point corresponds to an orbit of $\widehat{\phi}$. We take the quotient of this orbit by $g$ to get a closed orbit $\gamma$ of $\phi$ and set $F_\Gamma(c)=\gamma$. 

Similarly, we define a map $F_\Phi : Z_\Phi \to O$ as follows: Given a cycle $c$ of the flow graph $\Phi$, let $g=[c] \in \pi_1(M^\circ)$, and let $\widetilde{c}$ be the lift of $c$ that is preserved by $g$. $\widetilde{c}$ is a bi-infinite path in $\widetilde{\Phi}$. Suppose the vertices of $\widetilde{c}$ are, in order, $e_i$, for $i \in \mathbb{Z}$. Let $R_i$ be the tetrahedron rectangle in $\mathcal{P}$ corresponding to the tetrahedron of $\widetilde{\Delta}$ which has $e_i$ as its bottom edge. Then since for each $i$, $R_{i+1}$ is taller than $R_i$ and $R_i$ is wider than $R_{i+1}$, the intersection $\bigcap_i R_i$ is a single point that is invariant under $g$. This point corresponds to an orbit of $\widehat{\phi}$. We take the quotient of this orbit by $g$ to get a closed orbit $\gamma$ of $\phi$ and set $F_\Phi(c)=\gamma$. 

By construction, $c$ is homotopic to $F_\Gamma(c)$ for every cycle $c$ of $\Gamma$. Similarly for cycles of $\Phi$ and $F_\Phi$. Also note that if $c$ is a branch cycle of $\Gamma$, then $F_\Gamma(c)$ is some multiple of the element of $\mathcal{C}$ corresponding to the vertex of $\Delta$ that $c$ is homotopic into.

\begin{prop}[{\cite[Theorem 6.1]{LMT21}}] \label{prop:flowgraphencode}
We have the following properties of $F_\Phi$:
\begin{enumerate}
    \item If $\gamma \in \mathcal{C}$, then $|F_\Phi^{-1}(\gamma)| \leq 2\nu$
    \item If $\gamma \not\in \mathcal{C}$ and $F_\Gamma^{-1}(\gamma)$ does not contain AB-cycles, then $|F_\Phi^{-1}(\gamma)| = 1$.
    \item If $\gamma \not\in \mathcal{C}$ and $F_\Gamma^{-1}(\gamma)$ contains AB-cycles, then $1 \leq |F_\Phi^{-1}(\gamma^2)| \leq \delta$ and the elements of $F_\Phi^{-1}(\gamma^2)$ have the same length.
\end{enumerate}
\end{prop}
\begin{proof}
Let $O^+$ be the set of closed orbits of $\phi$ except for those in $\mathcal{C}$, union the set of multiples of red ladderpole curves at all the vertices of $\Delta$. There is a natural projection $p: O^+ \to O$ defined by sending orbits to themselves and multiples of a red ladderpole curve at a vertex of $\Delta$ to the corresponding element of $\mathcal{C}$ that it is homotopic to. In \cite[Section 6]{LMT21}, a map $\mathcal{F}: Z \to O^+$ is defined, such that $p \circ \mathcal{F} = F$; indeed, the definition of our $F$ is essentially the same as $\mathcal{F}$, except for a way of `lifting' $F(c)$ when it is an element of $\mathcal{C}$. 

We also point out that an orbit $\gamma$ of $\phi$ being homotopic to an AB cycle is equivalent to $F_\Gamma^{-1}(\gamma)$ containing AB-cycles. The backward implication is clear. For the forward implication, if $\gamma$ is homotopic to an AB-cycle $c$, then $\gamma$ is homotopic to $F_\Gamma(c)$, hence by \Cref{lemma:nohomotopicorbits}, $\gamma=F_\Gamma(c)$.

With this understanding, most of the proposition follows from the statement of \cite[Theorem 6.1]{LMT21}. The last part of (3) follows from the fact that the elements of $F_\Phi^{-1}(\gamma^2)$ are exactly some common multiple of the images of the boundary components of the AB strips on some dynamic plane. This fact is in turn established in the proof of \cite[Theorem 6.1]{LMT21}.
\end{proof}

We will use a partially defined multi-function $h: Z_\Gamma \dashrightarrow Z_\Phi$ to complete the commutative diagram 

\begin{center}
\begin{tikzcd}
Z_\Gamma \arrow[rd, "F_\Gamma"'] \arrow[rr, "h", dashed] & & Z_\Phi \arrow[ld, "F_\Phi"] \\
 & O & 
\end{tikzcd}
\end{center}

Given a cycle $c$ of $\Gamma$ that is not a branch cycle, let $g=[c] \in \pi_1(M^\circ)$, and let $\widetilde{c}$ be the lift of $c$ that is preserved by $g$. Consider the dynamic plane $D$ associated to $\widetilde{c}$ and consider the restriction of $\widetilde{\Phi}$ to $D$. If there is a $g$-invariant bi-infinite $\widetilde{\Phi}$-path on $D$, then its quotient by $g$ is a cycle $c'$ of $\Phi$ that is homotopic to $c$. Such a $g$-invariant bi-infinite $\widetilde{\Phi}$-path may not exist and may not be unique in general. In any case, we set $h(c)$ to be the set of cycles $c'$ that are obtained this way. 

\begin{lemma} \label{lemma:encodecommute}
Suppose $c$ is a cycle of the dual graph $\Gamma$ that is not a branch cycle. For every $c' \in h(c)$, $F_\Gamma(c)=F_\Phi(c')$.
\end{lemma}
\begin{proof}
By construction, $c$ is homotopic to any element $c'$ in $h(c)$, so $F_\Gamma(c)$ is homotopic to $F_\Phi(c')$ in $M^\circ$. By \Cref{lemma:nohomotopicorbits}, $F_\Gamma(c)=F_\Phi(c')$. 
\end{proof}

\subsection{Complexity of closed orbits} \label{subsec:orbitcomplexity}

\begin{defn} \label{defn:orbitcomplexity}
Let $\gamma$ be a closed orbit of $\phi$ which is not an element of $\mathcal{C}$. Take $c \in F_\Phi^{-1}(\gamma^2)$, which exists by \Cref{prop:flowgraphencode}. We define the \textit{flow graph complexity of $\gamma$ with respect to $\mathcal{C}$}, denoted by $\mathfrak{c}_\mathcal{C}(\gamma)$, to be $\frac{1}{2}$ times the length of $c$. By \Cref{prop:flowgraphencode}, $\mathfrak{c}_\mathcal{C}(\gamma)$ is well-defined.

When the collection $\mathcal{C}$ is clear from context, we will just write $\mathfrak{c}(\gamma)$ and call it the \textit{flow graph complexity} of $\gamma$.
\end{defn}

There is a natural motivation for making this definition. In \cite[Section 5]{AT22}, it is shown that using the veering triangulation $\Delta$ on $M^\circ$, one can construct a pseudo-Anosov flow $\phi'$ on $M$ with a collection of \textit{core orbits} $\mathcal{C}'$, such that $M \backslash \bigcup \mathcal{C}' = M^\circ$, and such that $\phi'$ is without perfect fits relative to $\mathcal{C}'$. Moreover, by \cite[Proposition 5.15]{AT22}, $\phi'$ admits a Markov partition encoded by the \textit{reduced flow graph} $\Phi_{\red}$ of $\Delta$, which is defined by deleting the infinitesimal cycles of $\Phi$. 

It is highly speculated, even though a complete proof has not been written down, that $\phi$ and $\phi'$ are orbit equivalent. Assuming that this is true for the moment, then $\Phi_{\red}$ encodes a Markov partition for $\phi$. By \cite[Lemma 3.8]{AT22}, $c$ in \Cref{defn:orbitcomplexity} can be chosen to be a cycle in $\Phi_{\red}$. Thus the flow graph complexity of $\gamma$ essentially records the length of a cycle needed to represent $\gamma$ in this particular Markov partition.

Furthermore, such a Markov partition is canonically associated to the flow if it is without perfect fits (for one can choose $\mathcal{C}$ to be the set of singular orbits), and canonically associated to the flow and the choice of $\mathcal{C}$ in general. Hence this flow graph complexity would be, at least in the case of no perfect fits, a canonical way of measuring the complexity of closed orbits.

\begin{rmk} \label{rmk:complexityofsingorbit}
A natural way to extend the definition of flow graph complexity to all closed orbits might be to define the flow graph complexity of $\gamma \in \mathcal{C}$ to be the sum over the lengths of all elements in $F_\Phi^{-1}(\gamma^n)$ then divide by $n$, where $n$ is the number of prongs at $\gamma$. The intuition is that there should be cycles in the Markov partition mentioned above that together $n$-fold cover $\gamma$, hence we can average them out.

However, we will not need to deal with this case in this paper, so we will leave the definition open for future interpretation.
\end{rmk}

A significance of the flow graph complexity is that it controls how many times the closed orbit can interact with the rectangles, in the sense of \Cref{lemma:complexityrectbound} below. To establish that, we show the following proposition, which is a quantitative upgrade of \cite[Proposition 3.15]{LMT21}.

\begin{prop} \label{prop:dualgraphflowgraph}
Let $c$ be a cycle of $\Gamma$ that is not a branch cycle. Suppose $c$ is of length $L$, then the length of any element in $h(c)$ is at most $(2L-1)L$ and at least $\frac{L}{\delta^2-\delta+1}$.
\end{prop}
\begin{proof}
Let $g=[c] \in \pi_1(M^\circ)$, and let $\widetilde{c}$ be the lift of $c$ that is preserved by $g$. Consider the dynamic plane $D$ determined by $\widetilde{c}$ and consider the restriction of $\widetilde{\Phi}$ on $D$. Recall that if the vertices of $\widetilde{c}$ are, in order, $v_i$, for $i \in \mathbb{Z}$, then $D = \bigcup_i \Delta(v_i)$. The boundary of each $\Delta(v_i)$ is a union of two $\widetilde{\Gamma}$-rays, and we can measure the distance between two points on $\partial \Delta(v_i)$ by the number of edges between them on $\partial \Delta(v_i)$. It is argued in \cite[Lemma 3.7]{LMT21} that if $x_1,x_2$ are two points on $\partial \Delta(v_i)$, and if $y_1,y_2$ are the intersections of the $\widetilde{\Phi}$-paths starting at $x_1, x_2$ with $\partial \Delta(v_{i+1})$ respectively, then the distance between $y_1,y_2$ is less than or equal to that between $x_1,x_2$. We refer to this property as `following $\widetilde{\Phi}$-paths contracts distances'.

We first make the following claim.

\begin{lemma} \label{lemma:upperbounds}
If we have a $\widetilde{\Gamma}$-path $(v_0,...,v_l)$ on $D$, then the $\widetilde{\Phi}$-path starting at $v_0$ and ending on $\partial \Delta(v_l)$ has length at most $(l-1)^2$ and has endpoint at most $l$ edges away from $v_0$. 
\end{lemma}

\begin{proof}[Proof of \Cref{lemma:upperbounds}]
We apply induction on $l$. For $l=1$, this is clear. For $l \geq 2$, by applying the lemma to the path $(v_0,...,v_{l-1})$, we know that the $\widetilde{\Phi}$-path $\alpha$ starting at $v_0$ and ending on $\partial \Delta(v_{l-1})$ has length $\leq (l-2)^2$ and has endpoint $\leq l-1$ edges away from $v_{l-1}$. If the endpoint of $\alpha$ on $\partial \Delta(v_{l-1})$ lies on $\partial \Delta(v_l)$ as well, then $\alpha$ is also the $\widetilde{\Phi}$-path starting at $v_0$ and ending on $\partial \Delta(v_l)$, hence has length $\leq (l-2)^2$ and has endpoint $\leq 1+(l-1)=l$ edges away from $v_0$. Otherwise, since following $\widetilde{\Phi}$-paths contracts distance, the $\widetilde{\Phi}$-path $\alpha'$ starting at $v_0$ and ending on $\partial \Delta(v_l)$ has endpoint $\leq l-2$ edges away from $v_0$. To compute the length of $\alpha'$, we recall the following definition from \cite{LMT21}.

\begin{defn} \label{defn:chainofsectors}
A \textit{chain} of sectors in a dynamic plane is a collection of sectors $\sigma_1,...,\sigma_n$ such that an entire bottom side of $\sigma_i$ is identified with a top side of $\sigma_{i+1}$ for each $i$, and there is a branch line that contains a top side of each $\sigma_i$. In this case we call $n$ the \textit{length} of the chain of sectors.
\end{defn}

Returning to the proof of the lemma, we can divide the side of $\partial \Delta(v_{l-1})$ not on $\partial \Delta(v_l)$ into bottom sides of chains of sectors. By the proof of \cite[Claim 6.10]{LMT21}, the $\widetilde{\Phi}$-path starting at any point on the $k^{\text{th}}$ chain enters the $(k-1)^{\text{th}}$ chain within $2$ edges. For the reader's convenience, we demonstrate the pictorial proof of this in \Cref{fig:chainofsectors}, which is a reproduction of \cite[Figure 30]{LMT21}.

\begin{figure} 
    \centering
    \resizebox{!}{4.5cm}{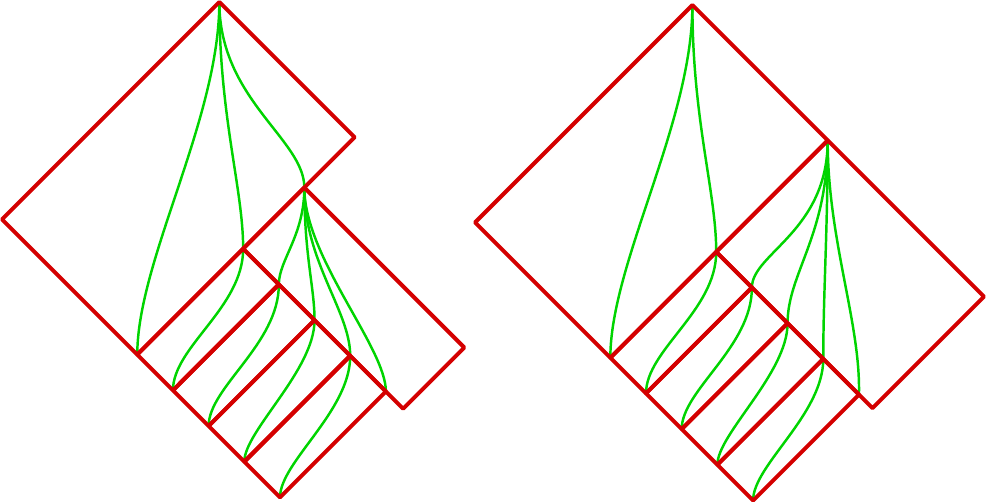}
    \caption{The $\widetilde{\Phi}$-path starting at any point on the $k^{\text{th}}$ chain enters the $(k-1)^{\text{th}}$ chain within $2$ edges. This figure is a reproduction of \cite[Figure 30]{LMT21}.} 
    \label{fig:chainofsectors}
\end{figure}

From this, we see that $\widetilde{\Phi}$-path starting at the endpoint of $\alpha$ must meet $\partial \Delta(v_l)$ within $2(l-2)+1$ edges, hence the length of $\alpha'$ is $\leq (l-2)^2+2(l-2)+1 = (l-1)^2$.
\end{proof}

Applying this lemma to the vertices $(v_0,...,v_L)$ of $\widetilde{c}$, we see that the $\widetilde{\Phi}$-path $\alpha$ starting at $v_0$ and ending on $\partial \Delta(v_L)$ has length $\leq (L-1)^2$ and has endpoint $\leq L$ edges away from $v_0$.

Next, we claim that if $\widetilde{c}$ takes an antibranching turn at $v_{L-1}$, which we can always arrange for by relabeling the vertices, then $\alpha$ has length $\geq \frac{L}{\delta^2-\delta+1}$.

To see this, suppose $\widetilde{c}$ takes an antibranching turn at some $v_i$ but takes a branching turn at $v_{i-k+1},...,v_{i-1}$ for $k \geq 1$. We illustrate this scenario in \Cref{fig:dualgraphflowgraph1}.

\begin{figure} 
    \centering
    \fontsize{16pt}{16pt}\selectfont
    \resizebox{!}{4.5cm}{
\begingroup%
  \makeatletter%
  \providecommand\color[2][]{%
    \errmessage{(Inkscape) Color is used for the text in Inkscape, but the package 'color.sty' is not loaded}%
    \renewcommand\color[2][]{}%
  }%
  \providecommand\transparent[1]{%
    \errmessage{(Inkscape) Transparency is used (non-zero) for the text in Inkscape, but the package 'transparent.sty' is not loaded}%
    \renewcommand\transparent[1]{}%
  }%
  \providecommand\rotatebox[2]{#2}%
  \newcommand*\fsize{\dimexpr\f@size pt\relax}%
  \newcommand*\lineheight[1]{\fontsize{\fsize}{#1\fsize}\selectfont}%
  \ifx\svgwidth\undefined%
    \setlength{\unitlength}{507.99048859bp}%
    \ifx\svgscale\undefined%
      \relax%
    \else%
      \setlength{\unitlength}{\unitlength * \real{\svgscale}}%
    \fi%
  \else%
    \setlength{\unitlength}{\svgwidth}%
  \fi%
  \global\let\svgwidth\undefined%
  \global\let\svgscale\undefined%
  \makeatother%
  \begin{picture}(1,0.50112531)%
    \lineheight{1}%
    \setlength\tabcolsep{0pt}%
    \put(0,0){\includegraphics[width=\unitlength,page=1]{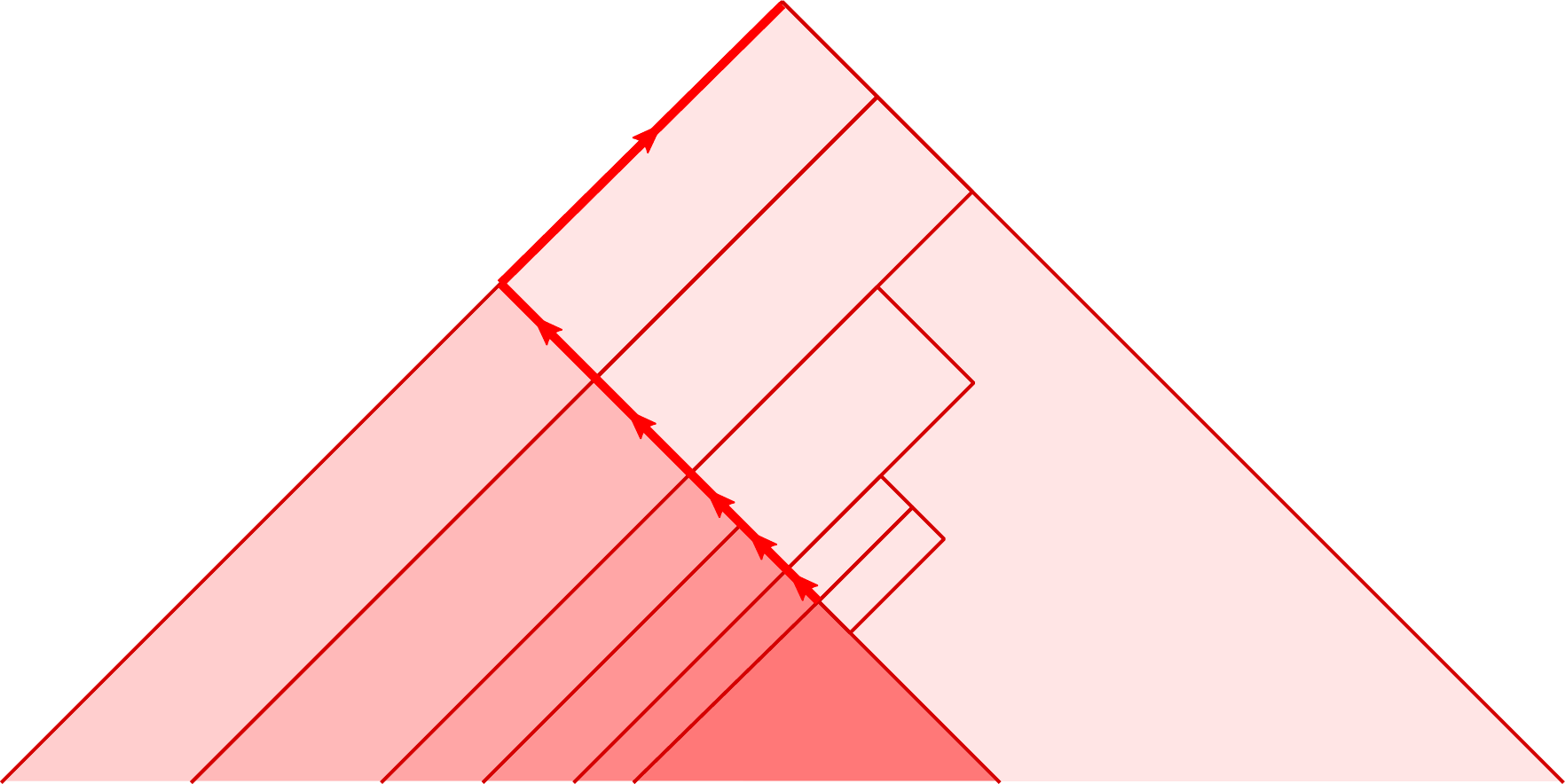}}%
    \put(0.28131456,0.32596929){\color[rgb]{0,0,0}\makebox(0,0)[lt]{\lineheight{1.25}\smash{\begin{tabular}[t]{l}$v_i$\end{tabular}}}}%
  \end{picture}%
\endgroup%
}
    \caption{Bounding the length of $\alpha$ from below by inspecting the situation at each antibranching turn.} 
    \label{fig:dualgraphflowgraph1}
\end{figure}

If $\alpha \cap \partial \Delta(v_{i-k})$ lies on the same branch line as $v_{i-k},...,v_i$, then the portion of $\alpha$ between $\partial \Delta(v_{i-k})$ and $\partial \Delta(v_{i+1})$ has length at least the number of chains between $v_{i-k},...,v_i$. But each sector in such a chain sits on top of $\leq \delta$ edges on the branch line, and it is shown in \cite[Lemma 6.8]{LMT21} that any chain of sectors in a dynamic plane has length $\leq \delta-1$. So this length is $\geq \lceil \frac{k}{\delta(\delta-1)} \rceil \geq \frac{k+1}{\delta^2-\delta+1}$.

If $\alpha \cap \partial \Delta(v_{i-k})$ lies on the different branch line as $v_{i-k},...,v_i$, let $j$ be the smallest positive number so that $\alpha \cap \partial \Delta(v_{i-j})$ lies on the different branch line as $v_{i-k},...,v_i$. Then the portion of $\alpha$ between $\partial \Delta(v_{i-s})$ and $\partial \Delta(v_{i-s+1})$ has length at least one for $k \geq s \geq j$, and the portion of $\alpha$ between $\partial \Delta(v_{i+1})$ and $\partial \Delta(v_{i-j+1})$ is at least $\lceil \frac{j-1}{\delta(\delta-1)} \rceil$ by the above argument. So the portion of $\alpha$ between $\partial \Delta(v_{i-k})$ and $\partial \Delta(v_{i+1})$ has length $\geq (k-j+1)+\lceil \frac{j-1}{\delta(\delta-1)} \rceil \geq \frac{k+1}{\delta^2-\delta+1}$. 

Applying this observation to every antibranching turn of $\widetilde{c}$ between $v_0$ and $v_L$ (or every other antibranching turn in a sequence of consecutive antibranching turns), we get the lower bound on the length of $\alpha$.

We now restrict to the case when $g$ acts on $D$ in an orientation preserving way. Let $\beta$ be the infinite $\widetilde{\Phi}$-path starting at $v_L$. We claim that $\beta$ and $g^{-1} \cdot \beta$ must converge at some point $w$. This follows from the proof of \cite[Lemma 3.7]{LMT21} if $v_0$ lies outside of the AB region of $D$, and follows from the fact that the AB region is $g$-invariant otherwise. The portion of $\beta$ between $w$ and $g \cdot w$ descends down to a $\Phi$-cycle in $h(c)$. Thus it remains to bound the length of this portion of $\beta$.

Consider the descending sets $g^i \cdot \Delta(v_0)$, $i \geq 0$. Suppose $w$ lies between $g^{r-1} \cdot \partial \Delta(v_0)$ and $g^r \cdot \partial \Delta(v_0)$, then $g \cdot w$ lies between $g^r \cdot \partial \Delta(v_0)$ and $g^{r+1} \cdot \partial \Delta(v_0)$. Let $\beta_i$ be the portion of $\beta$ between $g^i \cdot \partial \Delta(v_0)$ and $g^{i+1} \cdot \partial \Delta(v_0)$. Let $\beta'_{r-1}$ be the portion of $\beta$ between $g^{r-1} \cdot \partial \Delta(v_0)$ and $w$, and let $\beta'_r$ be the portion of $\beta$ between $g^r \cdot \partial \Delta(v_0)$ and $g \cdot w$. Then the length of the portion of $\beta$ between $w$ and $g \cdot w$ is 
\begin{align*}
    l(\beta_{r-1})-l(\beta'_{r-1})+l(\beta'_r) &= l(\beta_1) + \sum_{i=1}^{r-2} (l(\beta_{i+1}) - l(\beta_i)) + (l(\beta'_r) - l(\beta'_{r-1})) \\
    &= l(\beta_1) + \sum_{i=1}^{r-1} (l(\beta_{i+1}) - l(\beta_i))
\end{align*}

Notice that each $l(\beta_{i+1}) - l(\beta_i)$ is the difference of the lengths of the portions of $g^{-1} \cdot \beta$ and $\beta$ between $g^i \cdot \partial \Delta(v_0)$ and $g^{i+1} \cdot \partial \Delta(v_0)$. 

We illustrate the situation, in the case when $r=3$, in \Cref{fig:dualgraphflowgraph2}.

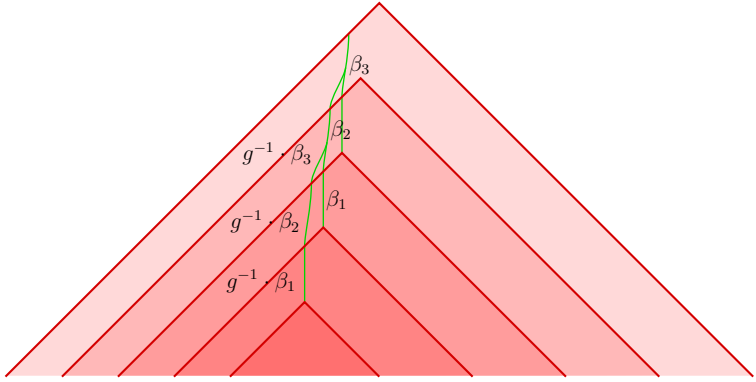
\begin{figure} 
    \centering
    \fontsize{12pt}{12pt}\selectfont
    \resizebox{!}{5cm}{
\begingroup%
  \makeatletter%
  \providecommand\color[2][]{%
    \errmessage{(Inkscape) Color is used for the text in Inkscape, but the package 'color.sty' is not loaded}%
    \renewcommand\color[2][]{}%
  }%
  \providecommand\transparent[1]{%
    \errmessage{(Inkscape) Transparency is used (non-zero) for the text in Inkscape, but the package 'transparent.sty' is not loaded}%
    \renewcommand\transparent[1]{}%
  }%
  \providecommand\rotatebox[2]{#2}%
  \newcommand*\fsize{\dimexpr\f@size pt\relax}%
  \newcommand*\lineheight[1]{\fontsize{\fsize}{#1\fsize}\selectfont}%
  \ifx\svgwidth\undefined%
    \setlength{\unitlength}{401.68119885bp}%
    \ifx\svgscale\undefined%
      \relax%
    \else%
      \setlength{\unitlength}{\unitlength * \real{\svgscale}}%
    \fi%
  \else%
    \setlength{\unitlength}{\svgwidth}%
  \fi%
  \global\let\svgwidth\undefined%
  \global\let\svgscale\undefined%
  \makeatother%
  \begin{picture}(1,0.50209258)%
    \lineheight{1}%
    \setlength\tabcolsep{0pt}%
    \put(0,0){\includegraphics[width=\unitlength,page=1]{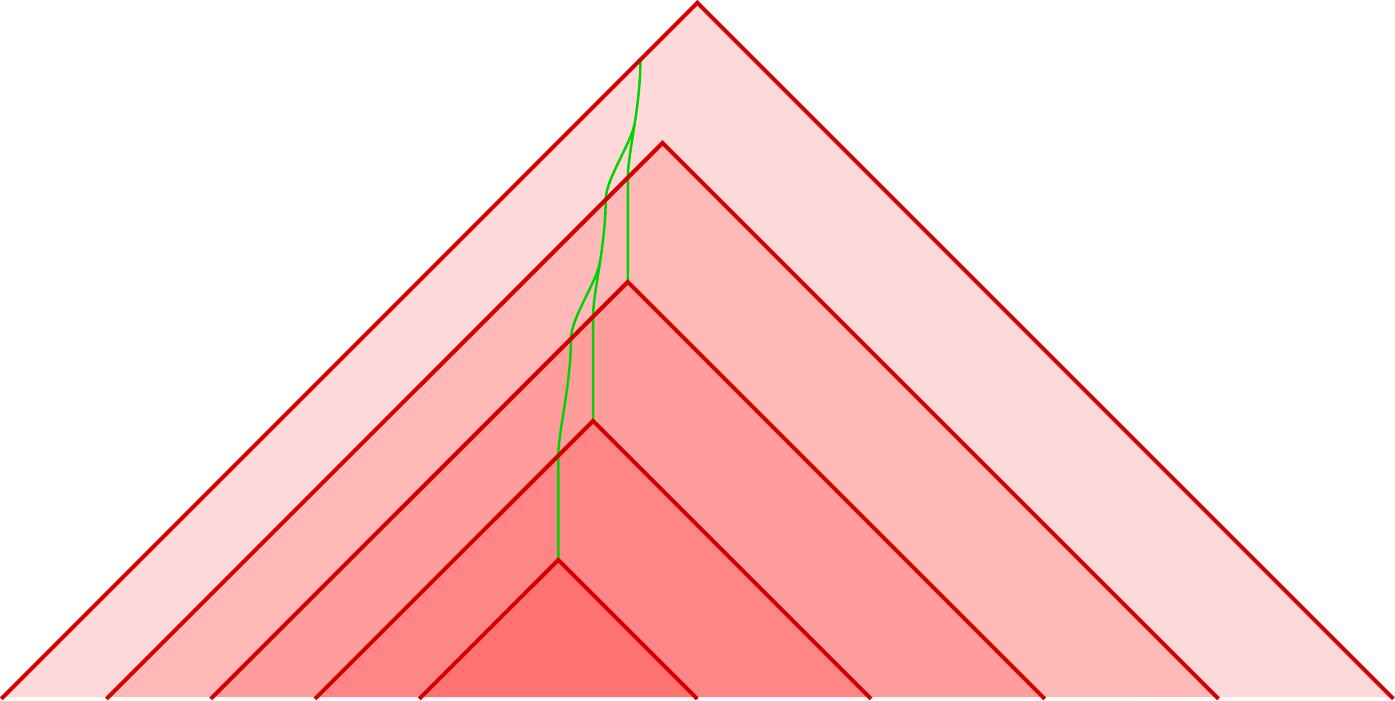}}%
    \put(0.42893707,0.22914219){\color[rgb]{0,0,0}\makebox(0,0)[lt]{\lineheight{1.25}\smash{\begin{tabular}[t]{l}$\beta_1$\end{tabular}}}}%
    \put(0.29584669,0.12011745){\color[rgb]{0,0,0}\makebox(0,0)[lt]{\lineheight{1.25}\smash{\begin{tabular}[t]{l}$g^{-1} \cdot \beta_1$\end{tabular}}}}%
    \put(0.43490096,0.32288698){\color[rgb]{0,0,0}\makebox(0,0)[lt]{\lineheight{1.25}\smash{\begin{tabular}[t]{l}$\beta_2$\end{tabular}}}}%
    \put(0.45993815,0.40939936){\color[rgb]{0,0,0}\makebox(0,0)[lt]{\lineheight{1.25}\smash{\begin{tabular}[t]{l}$\beta_3$\end{tabular}}}}%
    \put(0.3014289,0.19902587){\color[rgb]{0,0,0}\makebox(0,0)[lt]{\lineheight{1.25}\smash{\begin{tabular}[t]{l}$g^{-1} \cdot \beta_2$\end{tabular}}}}%
    \put(0.31733919,0.29023091){\color[rgb]{0,0,0}\makebox(0,0)[lt]{\lineheight{1.25}\smash{\begin{tabular}[t]{l}$g^{-1} \cdot \beta_3$\end{tabular}}}}%
  \end{picture}%
\endgroup%
}
    \caption{Computing the length of an element in $h(c)$ by comparing the infinite $\widetilde{\Phi}$-path $\beta$ starting at $v_L$ and the infinite $\widetilde{\Phi}$-path $g^{-1} \cdot \beta$ starting at $v_0$.} 
    \label{fig:dualgraphflowgraph2}
\end{figure}

Since $g$ acts in an orientation preserving way on $D$, $\beta \cap g^i \cdot \partial \Delta(v_0)$ and $g^{-1} \cdot \beta \cap g^i \cdot \partial \Delta(v_0)$ lie on the same side of $g^i \cdot \partial \Delta(v_0)$, with $\beta \cap g^i \cdot \partial \Delta(v_0)$ closer to $g^i \cdot v_0$ than $g^{-1} \cdot \beta \cap g^i \cdot \partial \Delta(v_0)$, for each $i$.

We make the following general claim.

\begin{lemma} \label{lemma:flowgraphcontraction}
If $\Delta(v) \subset \Delta(v')$ is a inclusion of descending sets, $x_1$ and $x_2$ are two points on the same side of $\partial \Delta(v)$ with $x_1$ closer to $v$ than $x_2$, such that the $\widetilde{\Phi}$ paths $\beta_1,\beta_2$ starting at $x_1,x_2$ end at $y_1, y_2$ respectively, which lie on the same side of $\partial \Delta(v')$ with $y_1$ closer to $v'$ than $y_2$, then $l(\beta_1) \leq l(\beta_2)$ and $l(\beta_2)-l(\beta_1)$ is less than or equal to the decrease in distance between $y_i$ compared to between $x_i$. 
\end{lemma}
\begin{proof}[Proof of \Cref{lemma:flowgraphcontraction}]
It suffices to prove the lemma when $x_i$ are one edge apart. In this case, $\beta_1$ and $\beta_2$ converge exactly when both of them enter the bottom side of a sector through vertices that are not the side vertex of the sector. Before this point the length of $\beta_2$ equals to the length of $\beta_1$ or the length of $\beta_1$ plus one. 
\end{proof}

Applying \Cref{lemma:flowgraphcontraction} to $\beta$ and $g^{-1} \cdot \beta$ between $g^i \cdot \partial \Delta(v_0)$ and $g^{i+1} \cdot \partial \Delta(v_0)$ for each $i$, we deduce that 
$$0 \leq \sum_{i=1}^{r-1} (l(\beta_{i+1}) - l(\beta_i)) \leq \text{Distance between $g^{-1} \cdot \beta$ and $\beta$ on $\partial \Delta(v_L)$} \leq L$$
Hence the length of the $\Phi$-cycle as described above is $\leq (L-1)^2 + L$ and $\geq \frac{L}{\delta^2-\delta+1}$. 

If $g$ acts in an orientation reversing way on $D$, then we apply the entire argument on $c^2$, which has length $2L$, to see that every element of $h(c^2)$ has length $\leq (2L-1)^2+2L$ and $\geq \frac{2L}{\delta^2-\delta+1}$. Hence by \Cref{prop:flowgraphencode}, every element of $h(c)$ has length $\leq \lfloor \frac{1}{2}((2L-1)^2+2L) \rfloor = (2L-1)L$ and $\geq \frac{L}{\delta^2-\delta+1}$.
\end{proof}

\begin{lemma} \label{lemma:complexityfacebound}
Let $\gamma$ be a closed orbit of $\phi$ which is not an element of $\mathcal{C}$. Suppose $\gamma$ has flow graph complexity $\mathfrak{c}$, then $\gamma$ intersects at most $2\delta^2 \mathfrak{c}$ faces of $\Delta$.
\end{lemma}
\begin{proof}
A closed orbit $\gamma$ that positively intersects $M$ faces of $\Delta$ can be perturbed to a path that positively intersect $\geq \frac{1}{2} M$ faces of $\Delta$ in their interiors, since the worst case scenario here is if $\gamma$ intersects an edge, in which case we can push $\gamma$ off to the side with more tetrahedra. Thus $\gamma$ be homotoped to a $\Gamma$-cycle $c$ of length $\geq \frac{1}{2} M$. By definition of $F_\Gamma$, $F_\Gamma(c)$ is homotopic to $c$, hence to $\gamma$, so by \Cref{lemma:nohomotopicorbits}, $F_\Gamma(c)=\gamma$. From this we deduce that $c$ is not a branch curve, since otherwise $\gamma \in \mathcal{C}$. Now by \Cref{prop:dualgraphflowgraph}, the length of any element in $h(c^2)$ is $\geq \frac{M}{\delta^2-\delta+1}$, so $2 \mathfrak{c} = \mathfrak{c}(\gamma^2) = l(h(c^2)) \geq \frac{M}{\delta^2-\delta+1} \geq \frac{M}{\delta^2}$
\end{proof}

Notice that the lemma can be restated as follows: Let $\gamma$ be a closed orbit of $\phi$ which is not an element of $\mathcal{C}$. Let $g=[\gamma]$ and let $p$ be the point in $\mathcal{P}$ which is invariant under $g$. Consider the set $F=\text{\{Projections of faces of $\widetilde{\Delta}$ on $\mathcal{P}$ which contain $p$\}}$, which $\langle g \rangle$ acts on. Suppose $\gamma$ has flow graph complexity $\mathfrak{c}$, then there are $\leq 2\delta^2 \mathfrak{c}$ many $\langle g \rangle$-orbits in $F$.

\begin{lemma} \label{lemma:complexityrectbound}
Let $\gamma$ be a closed orbit of $\phi$ which is not an element of $\mathcal{C}$. Let $g=[\gamma]$ and let $p$ be the point in $\mathcal{P}$ which is invariant under $g$. Consider the following sets
\begin{itemize}
    \item $Q=\text{\{Projections of equatorial squares of $\widetilde{\Delta}$ on $\mathcal{P}$ which contain $p$\}}$
    \item $E=\text{\{Edge rectangles in $\mathcal{P}$ which contain $p$\}}$
    \item $T=\text{\{Tetrahedron rectangles in $\mathcal{P}$ which contain $p$\}}$
\end{itemize}
$\langle g \rangle$ acts on each of these sets. Suppose $\gamma$ has flow graph complexity $\mathfrak{c}$, then there are $\leq 2 \delta^2 \mathfrak{c}$, $\leq 16 \delta^3 \mathfrak{c}$, $\leq 32 \delta^4 \mathfrak{c}$ many $\langle g \rangle$-orbits in $Q, E, T$ respectively.
\end{lemma}
\begin{proof}
A projection of a square $q$ is the union of the projections of the two top faces of the tetrahedron $q$ lies in. See \Cref{fig:complexityrectbound} left. Conversely, any projection of a face appears exactly once in such a union. So the first statement follows from \Cref{lemma:complexityfacebound}

For the second statement, we will first show that an edge rectangle $R$, say, associated to a blue edge $e$, can be covered by projections of certain squares. If a side of $e$ has 2 or more tetrahedra, then the slope of $R$ to that side is covered by projections of squares in the tetrahedra to that side of $e$. See \Cref{fig:complexityrectbound} middle, top right of the center edge. If a side of $e$ only has 1 tetrahedron $t$, then consider one of the red side edges $e'$ of $t$. The slope of $R$ to this side of $e$ is covered by projections of squares in the tetrahedra to the same side of $e'$ as $t$. See \Cref{fig:complexityrectbound} middle, bottom left of the center edge. 

Conversely, any square lies on 4 sides-of-edges and the union of squares in the tetrahedra to a side of an edge can appear in the collection of squares as above for at most $2 \delta - 3$ edges. So the second statement follows from the first statement.

Finally, a tetrahedron rectangle corresponding to a tetrahedron $t$ is covered by the union of edge rectangles corresponding to the 4 side edges and the top edge of $t$. See \Cref{fig:complexityrectbound} right. Conversely, any edge rectangle can appear in such a collection of edge rectangles for at most $2 \delta + 1$ tetrahedron rectangles. So the third statement follows from the second statement. 
\end{proof}

\begin{figure} 
    \centering
    \resizebox{!}{4.5cm}{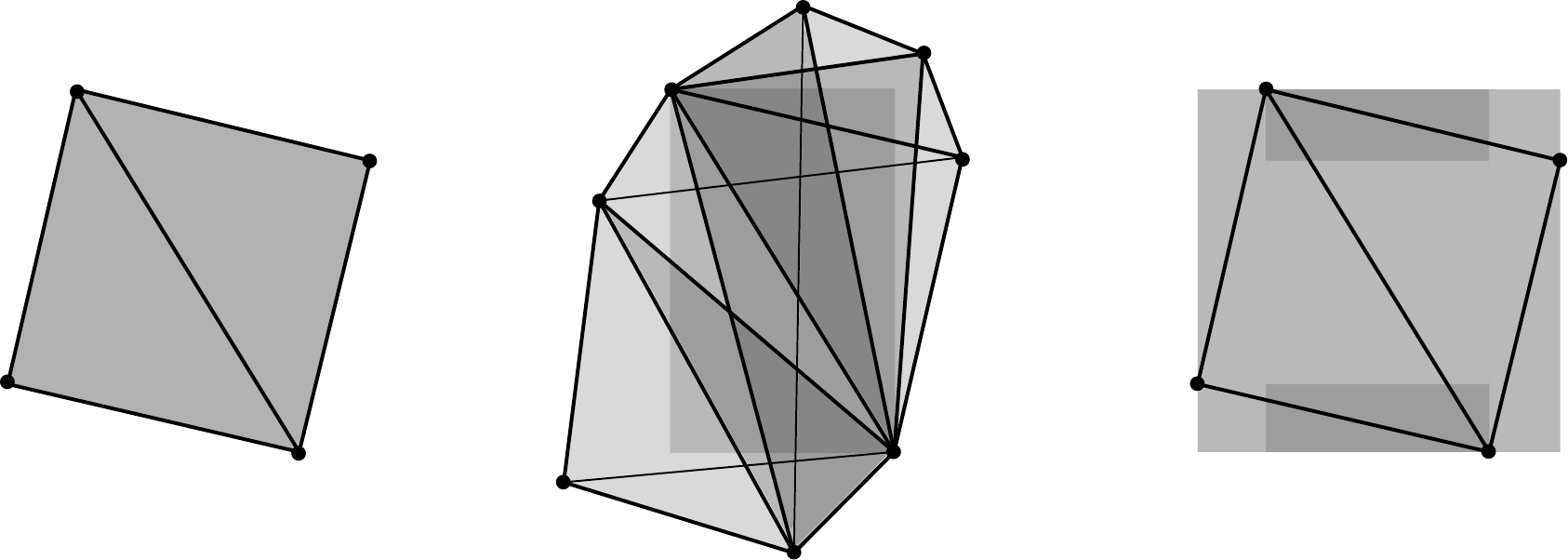}
    \caption{We consider how projections of faces/projections of equatorial squares/edge rectangles can cover projections of equatorial squares/edge rectangles/tetrahedron rectangles respectively to deduce \Cref{lemma:complexityrectbound} from \Cref{lemma:complexityfacebound}.} 
    \label{fig:complexityrectbound}
\end{figure}

\section{Broken transverse helicoids} \label{sec:helicoid}

In this section, we will construct the first type of broken transverse surfaces which we will use to assemble our Birkhoff sections. To briefly outline the construction: we start with an admissible collection of edges in the veering triangulation, then construct an edge sequence containing the given collection, lift this to a winding edge path, which bounds a helicoidal transverse surface then finally quotient this down to the desired broken transverse surface. For convenience we will just describe this for red edges/edge sequences/edge paths; the symmetric construction holds for blue edges/edge sequences/edge paths.

\subsection{Edge sequences} \label{subsec:edgeseqdefn}

\begin{defn} \label{defn:crossladderpole}
Let $\Delta$ be a veering triangulation. Let $T$ be a vertex of $\Delta$. Recall the form of the boundary triangulation $\partial \Delta$ at $T$, as explained in \Cref{subsec:vt}. Suppose $e_1$ and $e_2$ are two vertices of $\partial \Delta$ of the same color, then they each determine a ladderpole curve. We say that $e_1$ is \textit{$2k$ ladderpoles to the left} of $e_2$ at $T$ if the ladderpole curve determined by $e_1$ is $2k$ ladderpole curves to the left of that determined by $e_2$, for $k \geq 0$. When $k=0$ we also say that $e_1$ and $e_2$ \textit{lie in the same ladderpole} at $T$.
\end{defn}

\begin{defn} \label{defn:edgeseq}
A \textit{red edge sequence} is a sequence of red oriented edges $(e_i)_{i \in \mathbb{Z}/P}$ of $\Delta$ satisfying:
\begin{itemize}
    \item There is a collection of vertices $(T_i)_{i \in \mathbb{Z}/P}$ such that each $e_i$ goes from $T_{i-1}$ to $T_i$
    \item For each $i$, we have one of the following two cases
    \begin{enumerate}
        \item $e_i$ and $e_{i+1}$ lie in the same ladderpole at $T_i$
        \item $e_{i+1}$ lies $2$ ladderpoles to the left of $e_i$ at $T_i$
    \end{enumerate}
    and case (1) occurs for at least one value of $i$
\end{itemize}
We call $P$ the \textit{period} of the red edge sequence.

In case (2) above, suppose $l$ is the blue ladderpole curve to the left of the red ladderpole curve determined by $e_i$, then we say that the edge sequence crosses $l$ at $T_i$. If there are $n$ indices $i$ such that the edge sequence crosses $l$ at $T_i$, then we say that the edge sequence \textit{crosses $l$ $n$ times}.

A \textit{blue edge sequence} is defined similarly. When the color red/blue is clear from context, we will abbreviate these as \textit{edge sequences}.

A broken transverse surface $S$ is said to \textit{have boundary along closed orbits $\{\gamma_j\}$ and a red edge sequence $(e_i)$} if the sides in one component of $\partial S$ that are in $\partial_h S$ are $(e_i)$ in that order, while the other components of $\partial S$ lie along $\bigcup \gamma_j$.
\end{defn}

The reader will notice that the definition of an edge sequence bears a strong resemblance to the definition of the an edge path. Indeed, in \Cref{subsec:edgeseqtoedgepath}, we will show how to lift an edge sequence to an edge path. 

To the readers that find our terminology of edge sequences and edge paths confusing: One should think of an edge path as the path of edges of $\widetilde{\Delta}$ corresponding to the edge rectangles, which we will later show to be part of the boundary of a helicoidal transverse surface in $\widetilde{M^\circ}$. Meanwhile, an edge sequence should only be thought of as a sequence of oriented edges, and not as a path, since it is not specified how to travel from $e_i$ to $e_{i+1}$ within the boundary torus of a neighborhood of $T_i$. Indeed, the freedom of choosing how to do this will play a big role in the proof of \Cref{thm:closedBirkhoffsection}. 

\subsection{Constructing edge sequences} \label{subsec:edgeseqconstr}

\begin{defn} \label{defn:admissible}
A collection of red oriented edges $\{d_1,...,d_N\}$ of $\Delta$ is said to be \textit{admissible} if for each vertex of $\Delta$, the number of incoming edges equals to the number of outgoing edges. 
\end{defn}

\begin{lemma} \label{lemma:admissibletoedgeseq}
Let $\{d_1,...,d_M\}$ be an admissible collection of red oriented edges. Then there exists a collection of red oriented edges $\{d'_1,...,d'_k\}$ and a red edge sequence $(e_1,...,e_{M+2k})$ such that $(e_1,...,e_{M+2k})$ is a permutation of $(d_1,...,d_M,d'_1,...,d'_k,-d'_1,...,-d'_k)$. Moreover, if we are given a positive integer $\alpha_T$ for every vertex $T$, then we can arrange it so that the minimum number of times the edge sequence crosses a blue ladderpole curve on $T$ is $\alpha_T$, and in that case $k$ can be arranged to be at most $(\frac{\nu}{2} M + \max_T \alpha_T + 1) N + M$.
\end{lemma}
\begin{proof}
We choose $\{d'_1,...,d'_k\}$ to be $\frac{\nu}{2} M + \max_T \alpha_T + 1$ copies of the set of all red edges (with some choice of orientation). Let $\mathcal{E}=(d_1,...,d_M,d'_1,...,d'_k,-d'_1,...,-d'_k)$. We first show that $\mathcal{E}$ can be arranged into circular sequences $\{(e^m_i)_i \}_m$ such that:
\begin{itemize}
    \item There is a collection of vertices $T^m_i$ such that each $e^m_i$ goes from $T^m_{i-1}$ to $T^m_i$
    \item For each $i$, we have one of the following two cases
    \begin{enumerate}
        \item $e^m_i$ and $e^m_{i+1}$ lie in the same ladderpole at $T^m_i$
        \item $e^m_{i+1}$ lies $2$ ladderpoles to the left of $e^m_i$ at $T^m_i$
    \end{enumerate}
\end{itemize}
In other words, each $(e^m_i)_i$ is almost a red edge sequence, except the property that case (1) occurs for at least one value of $i$ may not be satisfied.

Consider each vertex $T$ of $\Delta$. Suppose the red ladderpoles on $T$ are $l_1,...,l_n$ in circular order. Suppose the elements of $\mathcal{E}$ that enter $T$ through a vertex on $l_j$ are $a_{j,1},...,a_{j,s_j}$ and the elements of $\mathcal{E}$ that exit $T$ through a vertex on $l_j$ are $b_{j,1},...,b_{j,t_j}$. By the choice of $\mathcal{E}$ and the definition of admissibility, we have the following properties of $s_j$ and $t_j$:
\begin{itemize}
    \item For every $j$, $s_j$ and $t_j$ are greater or equal to $\frac{n}{2} M + \max_T \alpha_T + 1$
    \item For every $j$, $|s_j-t_j| \leq M$
    \item $\sum_{j=1}^n s_j = \sum_{j=1}^n t_j$
\end{itemize}

We claim that there exists positive integers $u_1,...,u_n$, each less than or equal to $\frac{n}{2} M + \max_T \alpha_T$, such that $s_j-t_j = u_j - u_{j-1}$, where the indices are taken mod $n$. To see this, consider the integers $\sum_{k=1}^j (s_k-t_k)$ as $j$ ranges from $1$ to $n$. By cycling permuting $l_1,...,l_n$, we may assume that the minimum of these integers is attained at $j=n$, hence the minimum is $0$. The difference between two adjacent such integers is $s_k-t_k$, which has absolute value at most $M$, so the maximum over all these integers is bounded above by $\frac{n}{2} M$. Hence we can simply set $u_j = \alpha_T + \sum_{k=1}^j (s_k-t_k)$.

Now we arrange $\mathcal{E}$ by requiring, for each vertex $T$ of $\Delta$ as above, $a_{j,i}$ be followed by $b_{j,i}$ for $i=1,...,s_j-u_{j-1}$, and $a_{j,i}$ be followed by $b_{j-1,i-s_j+t_{j-1}}$ for $i=s_j-u_{j-1}+1,...,s_j$. Intuitively, we make the edge sequences stay within $l_j$ for the first $s_j-u_{j-1}=t_j-u_j$ indices and make them cross the blue ladderpole curve between $l_j$ and $l_{j+1}$ for the last $u_j$ indices. We draw a schematic picture of this in \Cref{fig:admissibletoedgeseq}. This divides $\mathcal{E}$ into sequences as claimed.

\begin{figure} 
    \centering
    \resizebox{!}{3.5cm}{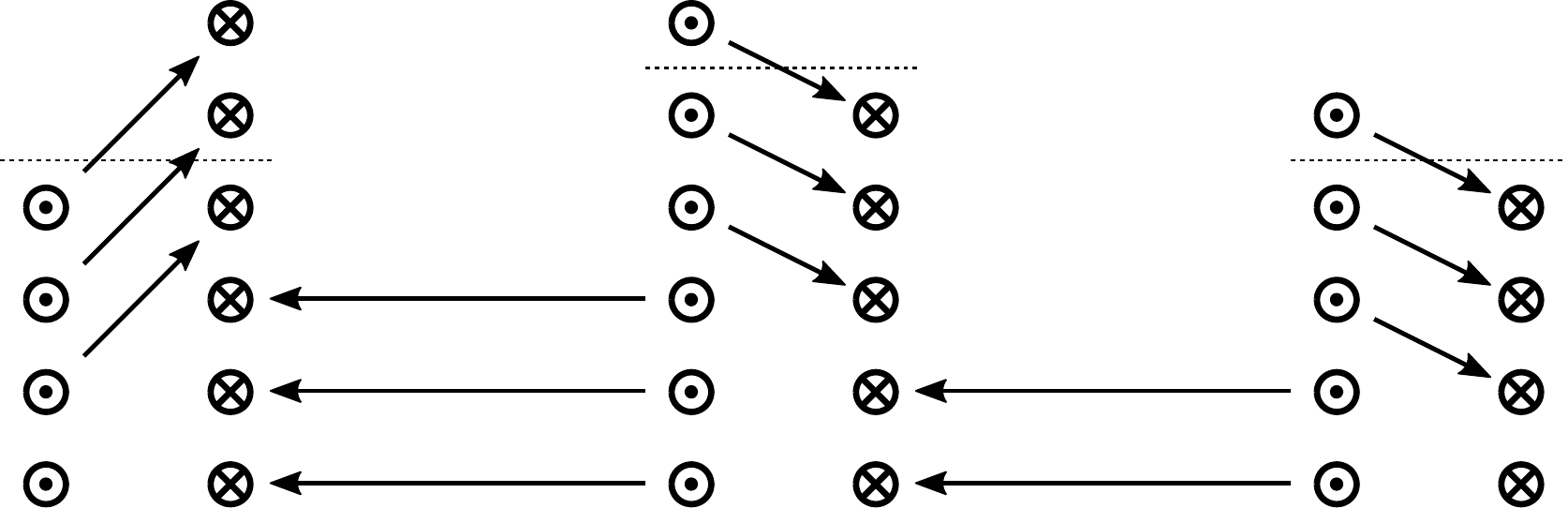}
    \caption{A schematic picture of how we divide $\mathcal{E}$ into sequences. We group the vertices according to which ladderpole curves they lie on. The vertices above the dotted lines are in the given admissible collection while the vertices below the dotted lines are copies of all the red edges. The placement of the arrows is rather arbitrary, and by moving them around we can arrange it so that we divide $\mathcal{E}$ into only one sequence.} 
    \label{fig:admissibletoedgeseq}
\end{figure}

Note that there is a large amount of freedom in the above construction, coming from the fact that there is no constraint on how to label the $a_{j,1},...,a_{j,s_j}$ and $b_{j,1},...,b_{j,t_j}$. Indeed, our next task is to argue that the labelling can be made so that we only end up with a single sequence. 

For each vertex $T$ as above, if $a_{j,i_1}$ and $a_{j,i_2}$ lie in different sequences for some $i_1,i_2$, we exchange their labels. This effectively performs a cut and paste operation on the two sequences and reduces the number of sequences by one. Thus we can assume that $a_{j,i}$ all lie in the same sequence for fixed $j$. Similarly, we can assume that $b_{j,i}$ all lie in the same sequence for fixed $j$. Since $s_j-u_{j-1} \geq (\frac{n}{2}M + \max_T \alpha_T + 1)-(\frac{n}{2} M + \max_T \alpha_T) \geq 1$, the common sequences of $a_{j,i}$ and $b_{j,i}$ agree. Similarly, since $u_j \geq 1$, the common sequences of $a_{j,i}$ and $b_{j+1,i}$ agree. Hence we can assume that all the elements of $\mathcal{E}$ that have a vertex at a fixed $T$ lie in the same sequence. Once this is achieved, we see that there is only one sequence, since $\overline{\Delta}$ is connected thus has connected 1-skeleton. Since each $u_j \geq 1$, case (1) in \Cref{defn:edgeseq} must occur at some index for this single sequence, and thus we have the desired red edge sequence.
\end{proof}

\begin{rmk} \label{rmk:admissibletoedgeseq}
Notice that in \Cref{lemma:admissibletoedgeseq}, for a fixed admissible collection, the minimum number of times the constructed edge sequence crosses a blue ladderpole curve $l$ on a fixed $T$ is attained at the same $l$ for any choice of $\alpha_T$. Indeed, the way $l$ is chosen in the proof only depends on the integers $s_k-t_k$, which in turn only depends on the given admissible collection.
\end{rmk}

\subsection{From edge sequences to edge paths} \label{subsec:edgeseqtoedgepath}

In this section, we show that edge sequences can be lifted to nice edge paths. For the purpose of proving the bounds in \Cref{thm:cuspedBirkhoffsection} and \Cref{thm:closedBirkhoffsection}, we need to pay attention to how these edge paths behave at the vertices. We make the following definition.

\begin{defn}
Suppose we fix a system of ladderpole transversals $\{t_T\}$. Consider the complete set of lifts of $\{t_T\}$ in $\widetilde{\Delta}$, thus at each vertex $\widetilde{T}$ of $\widetilde{\Delta}$ we get a $\mathbb{Z}$-collection of curves $(t_{\widetilde{T},j})_{j \in \mathbb{Z}}$ which we index, in order, from bottom to top.

Now let $(R_i)$ be a edge path. Let the vertices determined by the edges of $\widetilde{\Delta}$ corresponding to $R_i$ and $R_{i+1}$ at $s_i$ be $v_i$ and $v_{i+1}$ respectively. Let $\alpha$ be a curve from $v_i$ to $v_{i+1}$. We say that $R_{i+1}$ \textit{lies $k$ ladderpoles above} $R_i$ at $s_i$ if the signed intersection number between $\alpha$ and the $t_{s_i,j}$ is $k$. Here we take the convention that if the starting or ending point of $\alpha$ approaches some $t_{s_i,j}$ from above, we do not count that point in the intersection number, but if the point approaches some $t_{s_i,j}$ from below, we do count it in the intersection number. This ensures that the intersection number has the expected additivity properties.
\end{defn}

\begin{lemma} \label{lemma:edgeseqtoedgepath}
Every red edge sequence $(e_i)$ of period $P$ lifts to a nice red $g$-edge path $(R_i)$ of period $P$. Moreover, given integers $\beta_i \geq 4$, it can be arranged so that each $R_{i+1}$ lies $\beta_i$ ladderpoles above $R_i$ at $s_i$, and in that case it can be arranged so that $g$ quotients an orbit of $\widehat{\phi}$ to a closed orbit $\gamma$ of $\phi$ of flow graph complexity at most $2((\max_i \beta_i + 4)\lambda + 3)^2 \delta^2 P^2$.
\end{lemma}
\begin{proof}
By cycling permuting the edges, we can assume that $e_1$ and $e_2$ lie in the same ladderpole at $T_1$. Choose some lift $\widetilde{e_1}$ of $e_1$, and let $R_1$ be the edge rectangle corresponding to $\widetilde{e_1}$. We will inductively define $R_2,R_3,...$. Suppose we have defined $R_1,...,R_{j-1}$ where $e_{j-1}$ and $e_j$ lie in the same ladderpole at $T_{j-1}$. 

Let $e_j,...,e_k$ be a maximal subsequence such that $e_{i+1}$ lies $2$ ladderpoles to the left of $e_i$ at $T_i$ for $j \leq i <k$.
Let $\widetilde{e_{j-1}}$ be the edge of $\widetilde{\Delta}$ corresponding to $R_{j-1}$. Let $s_{j-1} \in \mathcal{S}$ be the corner of $R_{j-1}$ that corresponds to a vertex of $\widetilde{\Delta}$ covering $T_{j-1}$. Let $E_{j-1}$ be the tetrahedron of $\widetilde{\Delta}$ which has $\widetilde{e_{j-1}}$ as its bottom edge. Let $\widetilde{f_j}$ be the blue side edge of $E_{j-1}$ that has a vertex at $s_{j-1}$. Let $F_j$ be the tetrahedron of $\widetilde{\Delta}$ which has $\widetilde{f_j}$ as its bottom edge. Let $\widetilde{e'_j}$ be the red side edge of $F_j$ with a vertex at $s_{j-1}$. Let $e'_j$ be the image of $\widetilde{e'_j}$ in $\Delta$. See \Cref{fig:edgeseqtoedgepath1}.

\begin{figure} 
    \centering
    \fontsize{14pt}{14pt}\selectfont
    \resizebox{!}{8cm}{
\begingroup%
  \makeatletter%
  \providecommand\color[2][]{%
    \errmessage{(Inkscape) Color is used for the text in Inkscape, but the package 'color.sty' is not loaded}%
    \renewcommand\color[2][]{}%
  }%
  \providecommand\transparent[1]{%
    \errmessage{(Inkscape) Transparency is used (non-zero) for the text in Inkscape, but the package 'transparent.sty' is not loaded}%
    \renewcommand\transparent[1]{}%
  }%
  \providecommand\rotatebox[2]{#2}%
  \newcommand*\fsize{\dimexpr\f@size pt\relax}%
  \newcommand*\lineheight[1]{\fontsize{\fsize}{#1\fsize}\selectfont}%
  \ifx\svgwidth\undefined%
    \setlength{\unitlength}{436.63009784bp}%
    \ifx\svgscale\undefined%
      \relax%
    \else%
      \setlength{\unitlength}{\unitlength * \real{\svgscale}}%
    \fi%
  \else%
    \setlength{\unitlength}{\svgwidth}%
  \fi%
  \global\let\svgwidth\undefined%
  \global\let\svgscale\undefined%
  \makeatother%
  \begin{picture}(1,0.84777838)%
    \lineheight{1}%
    \setlength\tabcolsep{0pt}%
    \put(0,0){\includegraphics[width=\unitlength,page=1]{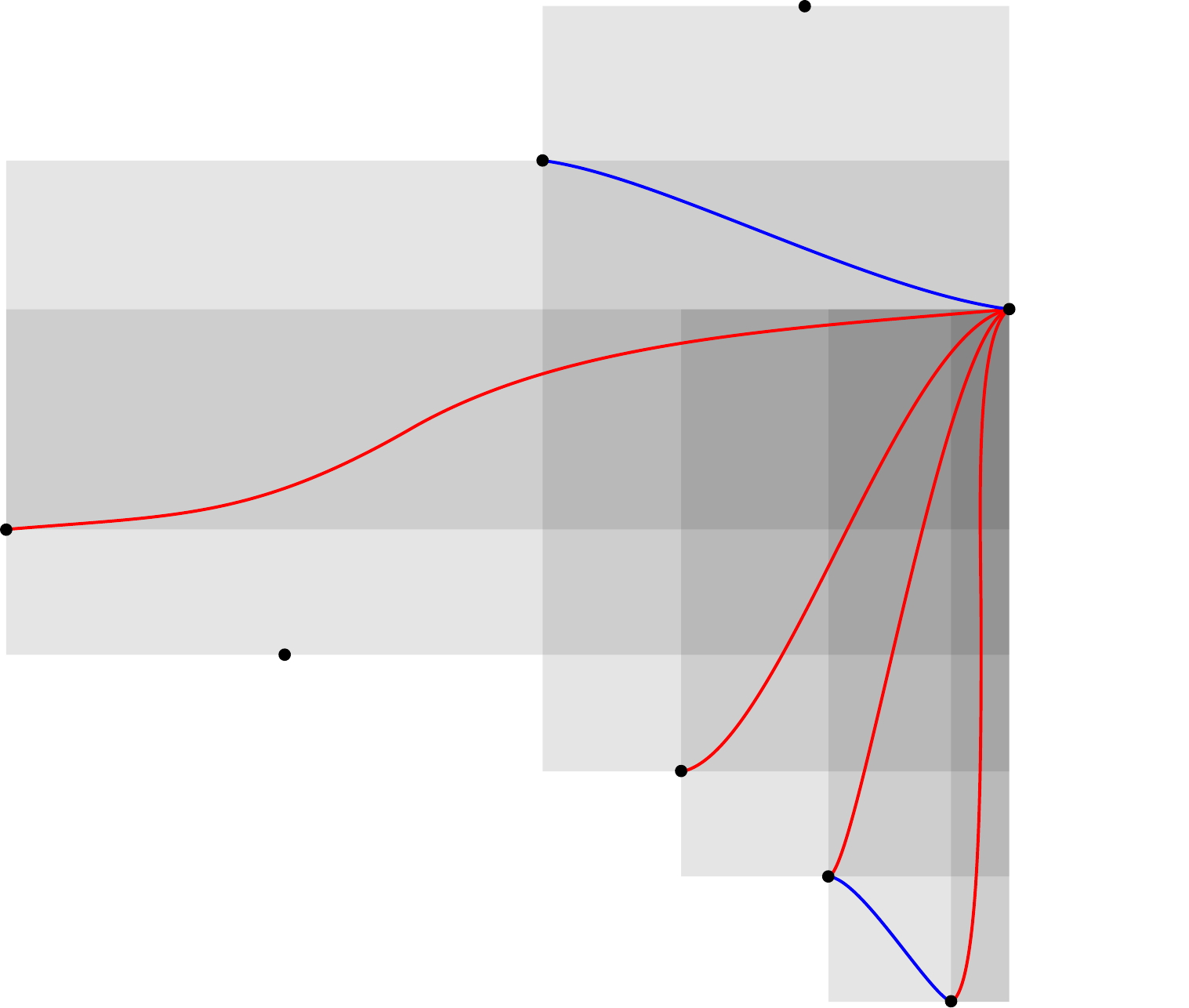}}%
    \put(0.2270362,0.4757699){\color[rgb]{0,0,0}\makebox(0,0)[lt]{\lineheight{1.25}\smash{\begin{tabular}[t]{l}$\widetilde{e_{j-1}}$\end{tabular}}}}%
    \put(0.66286755,0.65818905){\color[rgb]{0,0,0}\makebox(0,0)[lt]{\lineheight{1.25}\smash{\begin{tabular}[t]{l}$\widetilde{f_j}$\end{tabular}}}}%
    \put(0.83545856,0.27101951){\color[rgb]{0,0,0}\makebox(0,0)[lt]{\lineheight{1.25}\smash{\begin{tabular}[t]{l}$\widetilde{e_j}$\end{tabular}}}}%
    \put(0.69457072,0.0372894){\color[rgb]{0,0,0}\makebox(0,0)[lt]{\lineheight{1.25}\smash{\begin{tabular}[t]{l}$\widetilde{g_j}$\end{tabular}}}}%
    \put(0.62954344,0.34647569){\color[rgb]{0,0,0}\makebox(0,0)[lt]{\lineheight{1.25}\smash{\begin{tabular}[t]{l}$\widetilde{e'_j}$\end{tabular}}}}%
    \put(0.69852179,0.24341358){\color[rgb]{0,0,0}\makebox(0,0)[lt]{\lineheight{1.25}\smash{\begin{tabular}[t]{l}$\widetilde{h_j}$\end{tabular}}}}%
  \end{picture}%
\endgroup%
}
    \caption{Constructing $\widetilde{f_j}$, $\widetilde{e'_j}$, and $\widetilde{e_j}$.} 
    \label{fig:edgeseqtoedgepath1}
\end{figure}

Meanwhile, let $l$ be the ladderpole curve on $T_{j-1}$ determined by both $e'_j$ and $e_j$. Let $a$ be the sub-arc of $l$ going from $e_{j-1}$ to $e'_j$, and let $b$ be the sub-arc going from $e'_j$ to $e_j$. Here if $e'_j=e_j$, we take $b=l$. $a * b$ crosses the ladderpole transversal some $x \in [0,2]$ times. Lift $a * b * l^{\beta_{j-1}-x} * e_j$ to a path starting at the endpoint of $\widetilde{e_{j-1}}$ at $s_{j-1}$, and let $R_j$ be the edge rectangle corresponding to the edge $\widetilde{e_j}$ that $e_j$ lifts to.

Let $R'_j$ be the edge rectangle corresponding to $\widetilde{e'_j}$. We note that by construction, any point in $R_{j-1} \cap R'_j$ is closer to the vertical side of $R_{j-1}$ containing $s_{j-1}$ than the core point $c(R_{j-1})$. 

We will then inductively define $R_i$ for $j<i \leq k$. Let $s_{i-1} \in \mathcal{S}$ be the corner of $R_{i-1}$ that corresponds to a vertex of $\widetilde{\Delta}$ covering $T_{i-1}$. Let $t_{i-1}$ be the tetrahedron of $\widetilde{\Delta}$ which has $\widetilde{e_{i-1}}$ as the top edge. Let $\widetilde{g_{i-1}}$ be the blue side edge of $t_i$ with a vertex at $s_{i-1}$, and $\widetilde{h_{i-1}}$ be the red side edge of $t_i$ with a vertex at $s_{i-2}$. Let $g_{i-1}$ be the image of $\widetilde{g_{i-1}}$ in $\Delta$. Let $G_{i-1}$ be the tetrahedron of $\widetilde{\Delta}$ which has $\widetilde{g_{i-1}}$ as the bottom edge. Let $\widetilde{f_i}$ be the blue side edge of $G_{i-1}$ that has a vertex at $s_{i-1}$. Let $f_i$ be the image of $\widetilde{f_i}$ in $\Delta$. Let $F_i$ be the tetrahedron of $\widetilde{\Delta}$ which has $\widetilde{f_i}$ as its bottom edge. Let $\widetilde{e'_i}$ be the red side edge of $F_i$ with a vertex at $s_{i-1}$. Let $e'_i$ be the image of $\widetilde{e'_i}$ in $\Delta$. See \Cref{fig:edgeseqtoedgepath2}.

\begin{figure} 
    \centering
    \fontsize{14pt}{14pt}\selectfont
    \resizebox{!}{8.5cm}{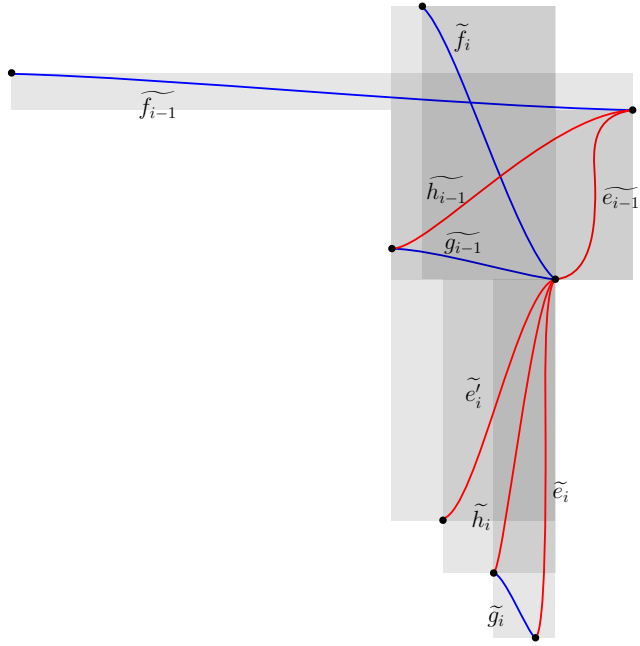}
    \caption{Constructing $\widetilde{f_i}$, $\widetilde{e'_i}$, and $\widetilde{e_i}$ for $j<i \leq k$.} 
    \label{fig:edgeseqtoedgepath2}
\end{figure}

Let $l$ be the ladderpole curve on $T_{i-1}$ determined by both $e'_i$ and $e_i$. Let $a$ be a path going from $e_{i-1}$ to $e'_i$, and let $b$ be the sub-arc of $l$ going from $e'_i$ to $e_i$. Here, again, if $e'_i=e_i$, then we take $b=l$. $a$ can be taken to be an edge path on the boundary triangulation going from $e_{i-1}$ to $g_{i-1}$ to $f_i$ then $e'_i$, hence crosses the ladderpole transversal some $y \in [-3,3]$ times, while $b$ crosses the ladderpole transversal some $z \in [0,1]$ times, hence $a * b$ crosses the ladderpole transversal some $x=y+z \in [-3,4]$ times. Lift $a * b * l^{\beta_{i-1}-x} * e_i$ to a path starting at the endpoint of $\widetilde{e_{i-1}}$ at $s_{i-1}$, and let $R_i$ be the edge rectangle corresponding to the edge $\widetilde{e_i}$ that $e_i$ lifts to. We note that we use the fact that $\beta_{i-1} \geq 4$ here.

By construction, for each $j<i \leq k$, the edge rectangle corresponding to $\widetilde{f_i}$ is taller than that of $\widetilde{f_{i-1}}$ and the edge rectangle corresponding to $\widetilde{f_{i-1}}$ is wider than that of $\widetilde{e'_{i-1}}$, hence wider than that of $\widetilde{f_i}$. From this, we can construct a staircase $S_{j,k}$ for $R_j,...,R_k$.

Following the construction, we arrive at $R_{P+1}$ eventually, which is the edge rectangle corresponding to $\widetilde{e_{P+1}}$. But $\widetilde{e_{P+1}}$ and $\widetilde{e_1}$ are both lifts of $e_1$, so $\widetilde{e_{P+1}} = g \cdot \widetilde{e_1}$ for some $g \in \pi_1(M^\circ)$. By naturality of the construction, $R_{i+rP} = g^r \cdot R_i$ for all $r \geq 0$. We thus extend the definition of $R_i$ by setting $R_{i+rP} = g^r \cdot R_i$ for all $r$. 

Suppose $e_{j_1},...,e_{k_1}$ and $e_{j_2},...,e_{k_2}$ are two consecutive maximal subsequences such that $e_{i+1}$ lies $2$ ladderpoles to the left of $e_i$ at $T_i$ for $j_s \leq i <k_s$ (hence $j_2=k_1+1$), then it follows from the definition of staircases and the fact that $R_{j_2}$ is taller than $R_{k_1}$ at $s_{j_2}$ that we have property (1) below.
\begin{enumerate}
    \item $S_{j_1,k_1}$ is wider than $S_{j_2,k_2}$ and $S_{j_2,k_2}$ is taller than $S_{j_1,k_1}$
    \item Any point $x \in R_{k_1} \cap S_{j_2,k_2}$ lies closer to the vertical side of $R_{k_1}$ containing $s_{k_1}$ than the core point $c(R_{k_1})$
\end{enumerate}
For property (2), notice that $R'_{j_2}$ is wider than $S_{j_2,k_2}$, so this follows from the corresponding property of $R_{k_1} \cap R'_{j_2}$ pointed out before.

Now consider the intersection of all $S_{j,k}$. By property (1) above, the intersection is nonempty. Moreover, since the set is invariant under $g$, it must consist of a single point $p$ corresponding to an orbit of $\widehat{\phi}$ covering a closed orbit $\gamma$ of $\phi$ of homotopy class $g$.

This shows that $\{R_i\}$ is a red $g$-edge path. We now show that it is nice: If $R_{i-1}$ and $R_i$ lie in the same quadrant at $s_{i-1}$ and $R_i$ and $R_{i+1}$ lie in the same quadrant at $s_i$, then $p$, being inside $S_{i,i} \cap S_{i+1,k}$, lies closer to the vertical side of $R_i$ containing $s_i$ than the core point $c(R_i)$ by property (2) above.

It remains to bound the flow graph complexity of $\gamma$. What we will do is bound the length of a cycle $c$ of the dual graph which maps to $\gamma$ under $F_\Gamma$ and apply \Cref{prop:dualgraphflowgraph}. 

We will use the following observation repeatedly in our computation.

Observation: Let $e$ be an edge. For any tetrahedron $t$ with $e$ as a side edge, there exists a path in the dual graph from $t$ to the tetrahedron with $e$ as the bottom edge of length $\leq \delta$.

Let $e_j,...,e_k$ be a maximal subsequence such that $e_{i+1}$ lies $2$ ladderpoles to the left of $e_i$ at $T_i$ for $j \leq i <k$. By the observation, there is a $\widetilde{\Gamma}$-path from $E_{j-1}$ to $F_j$ of length $\leq \delta$. Then for $j \leq i \leq k$, there is a $\widetilde{\Gamma}$-path from $F_i$ to the tetrahedron $E'_i$ with $\widetilde{e'_i}$ as the bottom edge of length $\leq \delta$. There exists a $\widetilde{\Gamma}$-path from $E'_i$ to $E_i$ as the bottom edge of length $\leq (\beta_{i-1}+4) \lambda \delta$. For $i<k$, this includes as a sub-arc a $\widetilde{\Gamma}$-path from $E'_i$ to the tetrahedron $H_i$ with $\widetilde{h_i}$ as the bottom edge of length $\leq (\beta_{i-1}+4) \lambda \delta$. There exists a $\widetilde{\Gamma}$-path from $H_i$ to $G_i$ of length $\leq \delta$. Finally, there exists a $\widetilde{\Gamma}$-path from $G_i$ to $F_{i+1}$ of length $\leq \delta$. Putting everything together, we see that there is a $\widetilde{\Gamma}$-path from $E_{j-1}$ to $E_k$ of length 
\begin{align*}
    &\leq \delta + (k-j) (\delta + (\max \beta_i + 4) \lambda \delta + \delta + \delta) + (\delta + (\max \beta_i + 4) \lambda \delta) \\
    &\leq (k-j+1) ((\max \beta_i + 4) \lambda + 3) \delta
\end{align*}
Hence there exists a $\widetilde{\Gamma}$-path from $E_1$ to $E_{N+1}$ of length $\leq ((\max \beta_i + 4) \lambda + 3) \delta P$, which descends to a $\Gamma$-cycle $c$, with the same bound on length, which is homotopic to $\gamma$. Now apply \Cref{prop:dualgraphflowgraph}. 
\end{proof}

Note that the $\widetilde{\Gamma}$-path constructed in the proof of \Cref{lemma:edgeseqtoedgepath} contains as a subpath $\beta_i-4$ consecutive lifts of a $\Gamma$-cycle $c_i$ that is homotopic to the orbit in $\mathcal{C}$ corresponding to $T_i$, between $E'_{i+1}$ and $H_{i+1}$ for each $i$. Thus, up to increasing some $\beta_i$ by $1$, which concatenates the constructed $\Gamma$-cycle $c$ with one more copy of $c_i$, we can arrange for $c$ to be a primitive cycle. However, this does not guarantee that the orbit $\gamma=F_\Gamma(c)$ is primitive. If one wants $\gamma$ to be primitive (equivalently, $g$ to be primitive) in this construction, which is what we will need in the proof of \Cref{thm:closedBirkhoffsection}, one needs to work harder.

We first inspect the dynamic plane $D_i$ determined by a lift $\widetilde{c_i}$ of $c_i$. Suppose $c_i$ has length $l_i$. We label the vertices of $\widetilde{c_i}$ as $(v^i_j)_{j \in \mathbb{Z}}$. One can check that $c_i$ lies on an annulus face of the complementary region of the stable branched surface that contains $T_i$. The two branch cycles on the boundary of this face lift to two branch lines on $D_i$. Any infinite $\widetilde{\Phi}$-path must enter the region $R_i$ bounded by these two branch lines eventually. See \Cref{fig:singorbitdyanmicplane} for a picture of the form of $D_i$. Also see \cite[Section 5.1.1]{LMT20} for a similar discussion. In particular this implies that if $\alpha$ is a $\widetilde{\Phi}$-path starting from a point $x$ on $\partial \Delta(v^i_j)$ outside of $R_i$ and ending on a point $y$ on $\partial \Delta(v^i_{j+l_i})$, then the distance between $y$ and $R_i$ on $\partial \Delta(v^i_{j+l_i})$ must be strictly less than the distance between $x$ and $R_i$ on $\partial \Delta(v^i_j)$.

\begin{figure} 
    \centering
    \fontsize{14pt}{14pt}\selectfont
    \resizebox{!}{4.5cm}{
\begingroup%
  \makeatletter%
  \providecommand\color[2][]{%
    \errmessage{(Inkscape) Color is used for the text in Inkscape, but the package 'color.sty' is not loaded}%
    \renewcommand\color[2][]{}%
  }%
  \providecommand\transparent[1]{%
    \errmessage{(Inkscape) Transparency is used (non-zero) for the text in Inkscape, but the package 'transparent.sty' is not loaded}%
    \renewcommand\transparent[1]{}%
  }%
  \providecommand\rotatebox[2]{#2}%
  \newcommand*\fsize{\dimexpr\f@size pt\relax}%
  \newcommand*\lineheight[1]{\fontsize{\fsize}{#1\fsize}\selectfont}%
  \ifx\svgwidth\undefined%
    \setlength{\unitlength}{433.82772503bp}%
    \ifx\svgscale\undefined%
      \relax%
    \else%
      \setlength{\unitlength}{\unitlength * \real{\svgscale}}%
    \fi%
  \else%
    \setlength{\unitlength}{\svgwidth}%
  \fi%
  \global\let\svgwidth\undefined%
  \global\let\svgscale\undefined%
  \makeatother%
  \begin{picture}(1,0.50312656)%
    \lineheight{1}%
    \setlength\tabcolsep{0pt}%
    \put(0,0){\includegraphics[width=\unitlength,page=1]{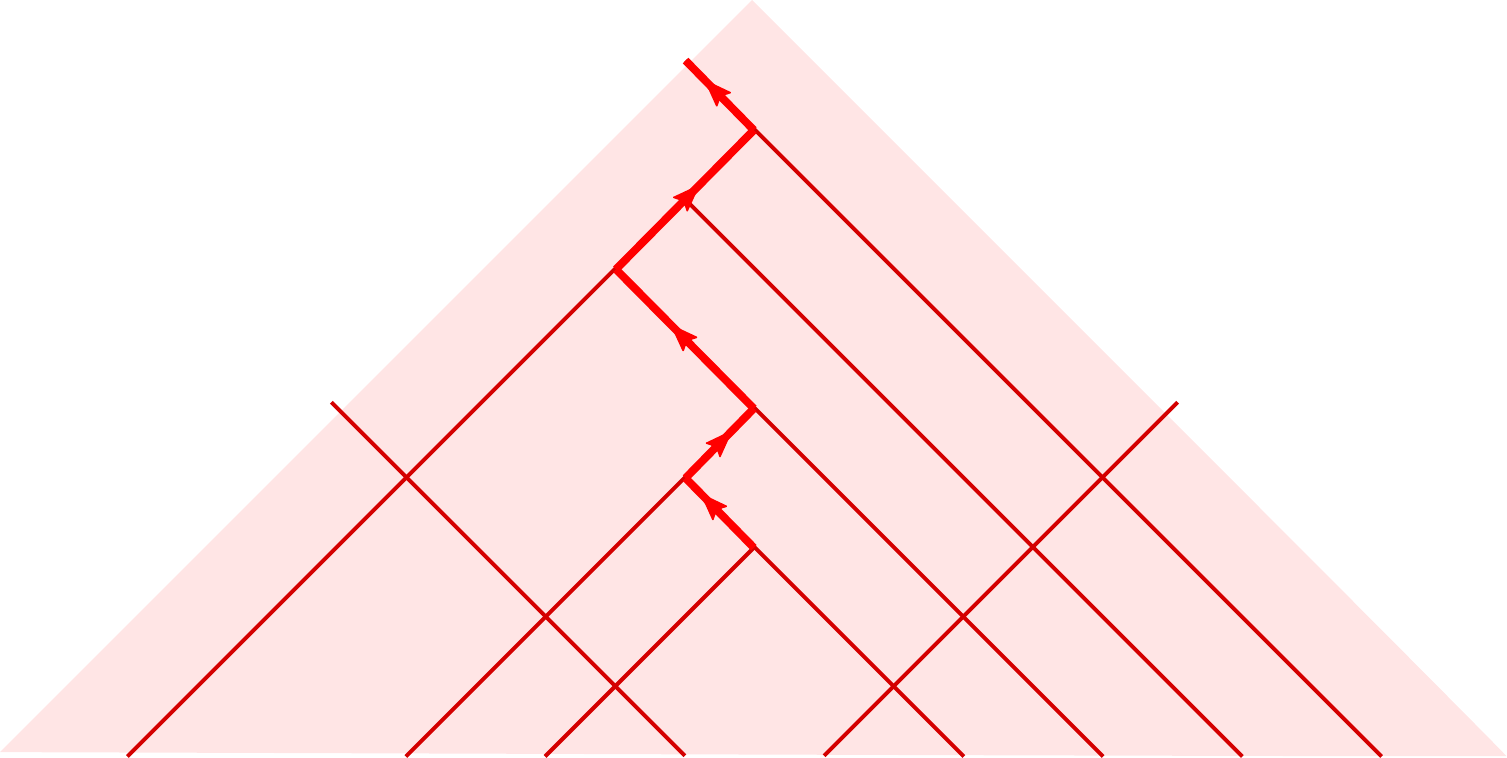}}%
    \put(0.19071435,0.25150266){\color[rgb]{0,0,0}\makebox(0,0)[lt]{\lineheight{1.25}\smash{\begin{tabular}[t]{l}$l_1$\end{tabular}}}}%
    \put(0.77034368,0.25150266){\color[rgb]{0,0,0}\makebox(0,0)[lt]{\lineheight{1.25}\smash{\begin{tabular}[t]{l}$l_2$\end{tabular}}}}%
  \end{picture}%
\endgroup%
}
    \caption{The dynamic plane $D_i$ associated to a lift $\widetilde{c_i}$ of $c_i$. There are branch lines $l_1$ and $l_2$ bounding a region $R_i$ such that any infinite $\widetilde{\Phi}$-path must enter the region $R_i$ eventually.} 
    \label{fig:singorbitdyanmicplane}
\end{figure}

Moreover, since the segment of each $\partial \Delta(v^i_j)$ within $R_i$ lies on the bottom sides of only two sectors, $D_i$ has at most two $[c_i]$-invariant bi-infinite $\widetilde{\Phi}$-paths (hence $D_i$ has at most one AB strip) and any $\widetilde{\Phi}$-path starting at a point on $\partial \Delta(v^i_j)$ in $R_i$ and ending on $\partial \Delta(v^i_{j+l_i})$ will converge into one of these paths. This also implies that any point on $\partial \Delta(v^i_j)$ in $R_i$ is $\leq \delta$ edges away from $v^i_j$.

Returning to the $\widetilde{\Gamma}$-path constructed in the proof of \Cref{lemma:edgeseqtoedgepath}, we elongate it by equivariance so that it is a lift of $c$, then consider the dynamic plane $D$ associated to it. $D$ contains $\beta_i-4$ adjacent fundamental regions of $D_i$, the union of which we denote by $D'_i$. We identify each $D'_i$ with the region between $\partial \Delta(v^i_0)$ and $\partial \Delta(v^i_{(\beta_i-4)l_i})$ on $D_i$. From the proof of \Cref{lemma:edgeseqtoedgepath}, we know that the segment of $\widetilde{c}$ between $D'_i$ and $D'_{i+1}$ has length $\leq (16\lambda+3)\delta$. In particular the $\widetilde{\Phi}$-path starting at $v^i_{(\beta_i-4)l_i}$ and ending on $\partial \Delta(v^{i+1}_0)$ is $\leq (16\lambda+3)\delta$ edges away from $v^{i+1}_0$ by \Cref{lemma:upperbounds}. Hence if $\beta_i-4 \geq (16 \lambda+3) \delta + \delta + 1$, then the infinite $\widetilde{\Phi}$-path starting at any point on $\partial \Delta(v^1_0) \cap R_1$ must converge into one of the $[c_i]$-invariant $\widetilde{\Phi}$-paths within each $D'_i$ that it passes through. We illustrate the situation schematically in \Cref{fig:primitive}. As reasoned in the proof of \Cref{prop:dualgraphflowgraph}, such a path eventually becomes periodic, hence we can find a $g$-invariant bi-infinite $\widetilde{\Phi}$-path on $D$ satisfying the same property, assuming that $D$ has any $g$-invariant bi-infinite $\widetilde{\Phi}$-paths at all, which quotients down to an element of $h(c)$.

\begin{figure} 
    \centering
    \fontsize{24pt}{24pt}\selectfont
    \resizebox{!}{7.5cm}{
\begingroup%
  \makeatletter%
  \providecommand\color[2][]{%
    \errmessage{(Inkscape) Color is used for the text in Inkscape, but the package 'color.sty' is not loaded}%
    \renewcommand\color[2][]{}%
  }%
  \providecommand\transparent[1]{%
    \errmessage{(Inkscape) Transparency is used (non-zero) for the text in Inkscape, but the package 'transparent.sty' is not loaded}%
    \renewcommand\transparent[1]{}%
  }%
  \providecommand\rotatebox[2]{#2}%
  \newcommand*\fsize{\dimexpr\f@size pt\relax}%
  \newcommand*\lineheight[1]{\fontsize{\fsize}{#1\fsize}\selectfont}%
  \ifx\svgwidth\undefined%
    \setlength{\unitlength}{354.73159109bp}%
    \ifx\svgscale\undefined%
      \relax%
    \else%
      \setlength{\unitlength}{\unitlength * \real{\svgscale}}%
    \fi%
  \else%
    \setlength{\unitlength}{\svgwidth}%
  \fi%
  \global\let\svgwidth\undefined%
  \global\let\svgscale\undefined%
  \makeatother%
  \begin{picture}(1,1.35874324)%
    \lineheight{1}%
    \setlength\tabcolsep{0pt}%
    \put(0,0){\includegraphics[width=\unitlength,page=1]{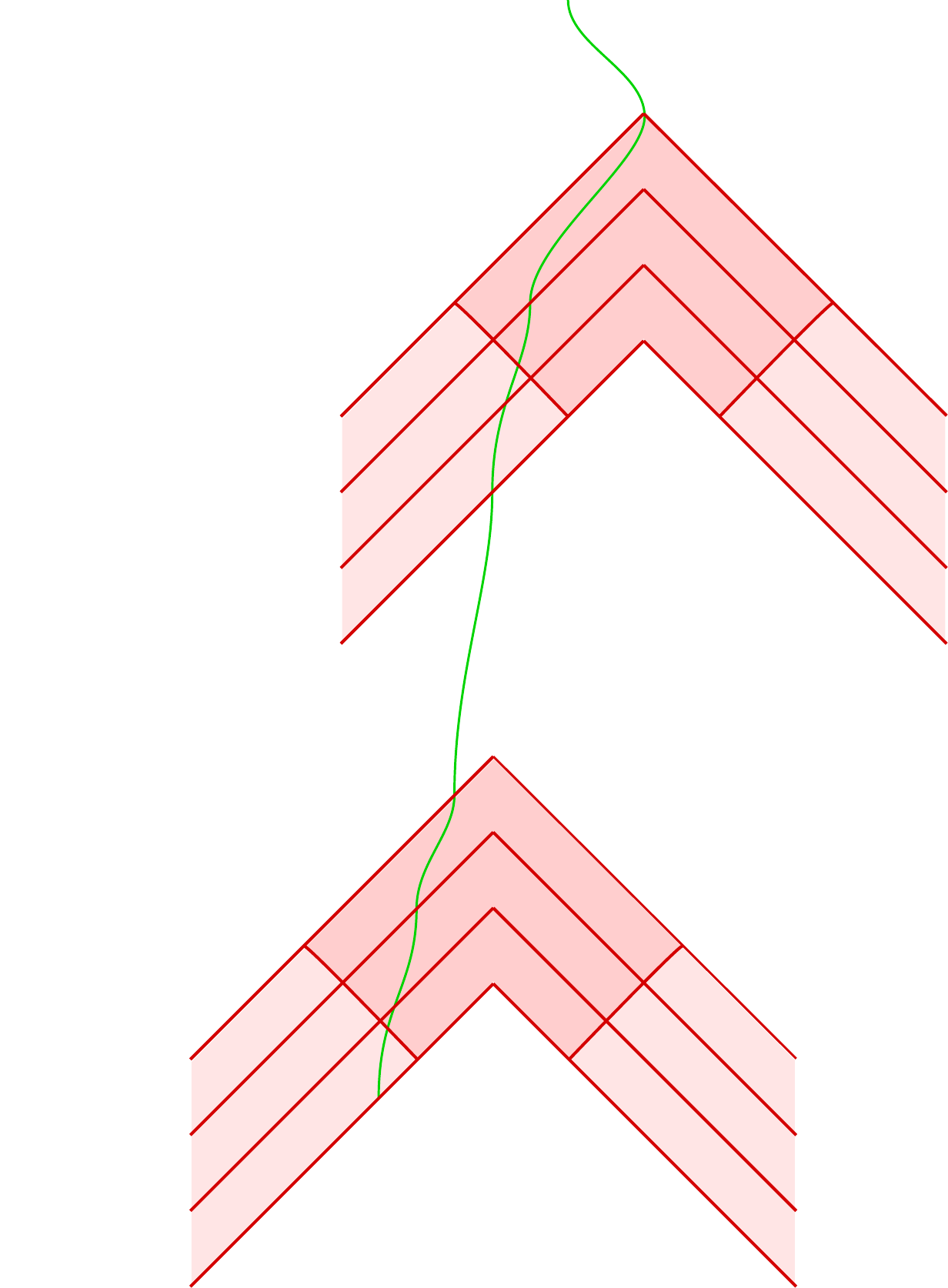}}%
    \put(-0.00384038,0.76913376){\color[rgb]{0,0,0}\makebox(0,0)[lt]{\lineheight{1.25}\smash{\begin{tabular}[t]{l}$D'_{i+1}$\end{tabular}}}}%
    \put(-0.00384038,0.08990217){\color[rgb]{0,0,0}\makebox(0,0)[lt]{\lineheight{1.25}\smash{\begin{tabular}[t]{l}$D'_i$\end{tabular}}}}%
  \end{picture}%
\endgroup%
}
    \caption{A schematic picture of the proof of \Cref{lemma:edgepathprimitive}. By increasing the number of fundamental regions in $D'_i$, we can obtain an $g$-invariant bi-infinite $\widetilde{\Phi}$-path on $D$ that meets the $[c_i]$-invariant $\widetilde{\Phi}$-paths within each $D'_i$.}
    \label{fig:primitive}
\end{figure}

The significance of this is that if we now increase $\beta_i$ by 1, then we concatenate this element of $h(c)$ by a $\Phi$-cycle homotopic to $c_i$. As in the case of $\Gamma$-cycles, this can be used to ensure that $h(c)$ contains a primitive element. Moreover, since each $D_i$ can have at most one AB strip, $D$ cannot have more than one AB strip as well, so this primitive element of $h(c)$ does not lie between AB strips. Hence using \Cref{prop:flowgraphencode}, we can guarantee that $\gamma=F_\Phi(h(c))$ is primitive. If $D$ does not contain $g$-invariant bi-infinite $\widetilde{\Phi}$-paths, then we can apply a similar argument to $c^2$.

We record this as an addendum to \Cref{lemma:edgeseqtoedgepath}. 

\begin{lemma} \label{lemma:edgepathprimitive}
In the setting of \Cref{lemma:edgeseqtoedgepath}, if each $\beta_i \geq (16 \lambda+4) \delta +1$, then up to increasing some $\beta_i$ by 1, we can assume that $g$ is primitive for the constructed nice red $g$-edge path.
\end{lemma}

\subsection{From edge paths to broken transverse helicoids} \label{subsec:helicoidconstr}

\begin{constr} \label{constr:helicoid}
Suppose we are given a winding red $g$-edge path $(R_i)$ of period $P$, where $g$ quotients an orbit $\widehat{\gamma}$ of $\widehat{\phi}$ to a closed orbit $\gamma$ of $\phi$. We construct a broken transverse surface as follows.

For each $i$ such that $R_i$ and $R_{i+1}$ lie in the same quadrant at $s_i$, choose a path $\alpha_i$ on $\mathcal{P}$ from $p$ to $s_i$ that lies in the intersection of the slope for $R_j,...,R_i$ and the slope for $R_{i+1},...,R_k$, that is topologically transverse to $\mathcal{P}^s$ and $\mathcal{P}^u$. This ensures that within each staircase $S_{j,k}$, the paths $\alpha_j$ and $\alpha_k$ are disjoint. We can also arrange it so that $\alpha_{i+rP}=g^r \alpha_i$. Next, for each of the remaining $i$, take a path $\alpha_i$ on $\mathcal{P}$ from $p$ to $s_i$ in the slope $S_{j,k}$ that $s_i$ lies in. Again, we make the choice so that $\alpha_i$, $j \leq i \leq k$ are topologically transverse to $\mathcal{P}^s$ and $\mathcal{P}^u$, are mutually disjoint within the staircase, and so that $\alpha_{i+rP}=g^r \alpha_i$.

Now take a sequence of points $\{x_i\}$ on $\widehat{\gamma}$ such that $x_{i+rP}=g^r \cdot x_i$ and $x_{i+1}$ lies further along $\widehat{\gamma}$ as a flow line than $x_i$. Also let $\widetilde{e_i}$ be the edge of $\widetilde{\Delta}$ that corresponds to the edge rectangle $R_i$. For each $i$, consider the restriction of the fibration $\widetilde{M} \to \mathcal{O}$ to $\alpha_i$, which is a trivial bundle, and choose a lift of $\alpha_i$ to $\widetilde{\alpha_i}$ starting at $x_i \in \widetilde{\gamma}$. Then for each $i$ consider the restriction of the fibration to the region bounded by $\alpha_i$, $\widetilde{e_i}$, and $\alpha_{i+1}$, minus the points $\widehat{\gamma}$, $s_i$, and $s_{i+1}$. For generic choices of lifts $\widetilde{\alpha_i}$ and placements of $\widetilde{e_i}$, we can lift this region to a hexagonal broken transverse surface $\widetilde{H_i}$ with horizontal boundary along $\widetilde{\alpha_i}$, $\widetilde{e_i}$, $\widetilde{\alpha_{i+1}}$, and vertical boundary along $\widehat{\gamma}$, and the orbits of $\widehat{\phi}$ corresponding to $s_i$ and $s_{i+1}$. Again, we make the choice so that $\widetilde{H_{i+rP}}=g^r \cdot \widetilde{H_i}$.

Finally, take the union over all $\widetilde{H_i}$, possibly resolving any turning points on the orbits of $\widehat{\phi}$ corresponding to the $s_i$ as in the last step of Fried resolution, to get a helicoidal broken transverse surface $\widetilde{H(R_i)}$ with one boundary component lying along $\widehat{\gamma}$ and the sides of the other boundary component that lie in $\partial_h \widetilde{H(R_i)}$ being $(\widetilde{e_i})$. See \Cref{fig:helicoid}. This surface is $g$-invariant, so we can take its quotient to get a broken transverse surface $H(R_i)$ with boundary along $\gamma$ and $\{e_i\}$.

Moreover, if $g$ is primitive, then the boundary component of $H(R_i)$ along $\gamma$ is embedded. In general, if $g$ is the $r^{\text{th}}$ power of a primitive element, then the boundary component of $H(R_i)$ along $\gamma$ $r$-fold covers $\gamma$.

\begin{figure} 
    \centering
    \resizebox{!}{5.5cm}{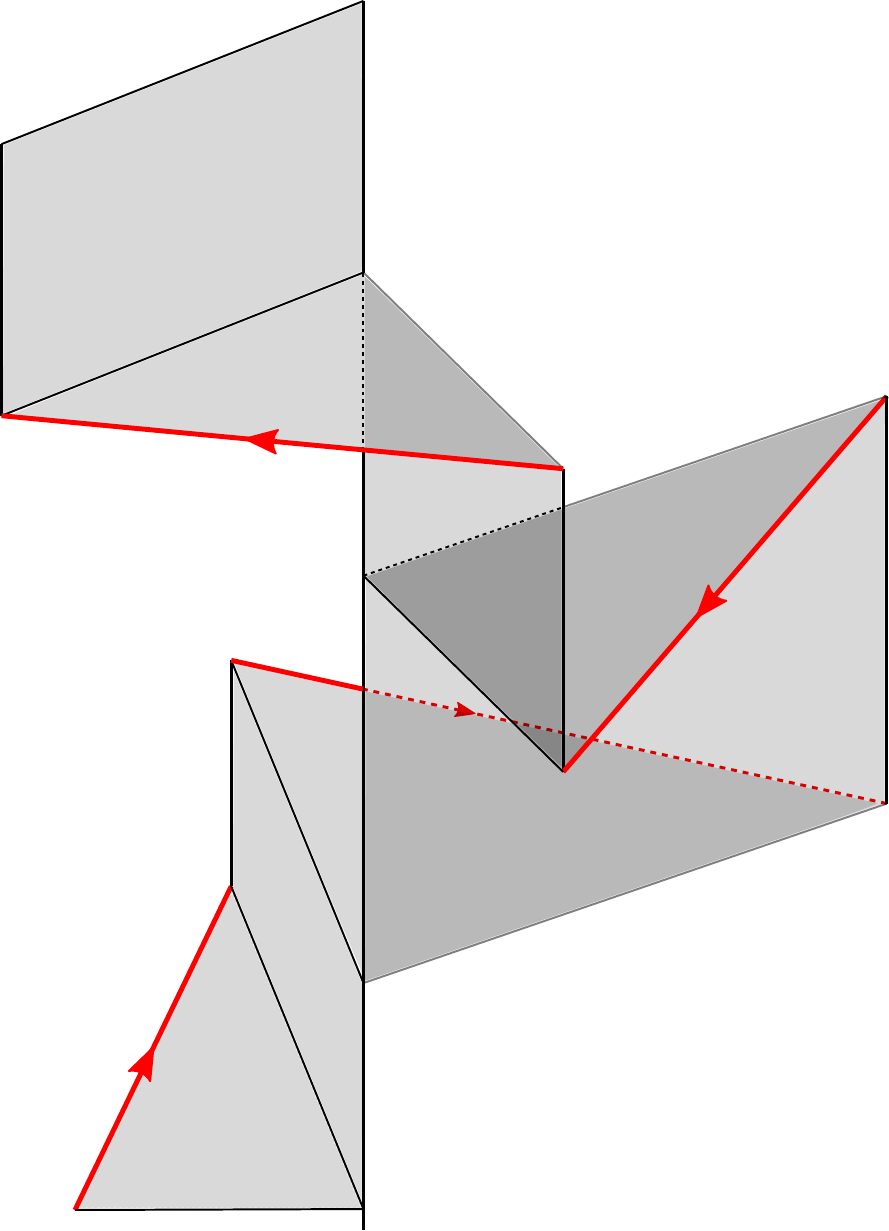}
    \caption{The helicoidal broken transverse surface $\widetilde{H(R_i)}$ with boundary along $\widehat{\gamma}$ and $(\widetilde{e_i})$.} 
    \label{fig:helicoid}
\end{figure}
\end{constr}

\begin{lemma} \label{lemma:helicoidintersection}
If $(R_i)$ is a winding red $g$-edge path of period $P$ and $\gamma'$ is a closed orbit of $\phi$ that is not an element of $\mathcal{C}$ of flow graph complexity $\mathfrak{c}$, then $\gamma'$ intersects $H(R_i)$ at most $\leq 32 \delta^4 \mathfrak{c} P$ times.
\end{lemma}
\begin{proof}
The projection of $\widetilde{H(R_i)}$ on $\mathcal{P}$ is covered by the $\langle g \rangle$-orbits of the tetrahedron rectangles corresponding to the tetrahedra with the edges corresponding to $R_1,...,R_P$ as bottom edges. Conversely, the $\langle g \rangle$-orbit of a tetrahedron rectangle can appear at most $P$ time in such a union. So this follows from the third statement in \Cref{lemma:complexityrectbound}.
\end{proof}

\section{Constructing Birkhoff sections} \label{sec:Birkhoffsection}

In this section, we will introduce the remaining type of broken transverse surfaces which we will use in our construction, and use them to prove \Cref{thm:cuspedBirkhoffsection} and \Cref{thm:closedBirkhoffsection}.

\subsection{Shearing decomposition} \label{subsec:veeringsolidtori}

We recall the shearing decomposition of a veering triangulation, introduced by Schleimer and Segerman in \cite{SS22a}. 

Fix the same setting as before: Let $\phi$ be a pseudo-Anosov flow on an oriented closed 3-manifold $M$ without perfect fits relative to $\mathcal{C}$, and let $\Delta$ be a veering triangulation associated to $\phi$ on $M \backslash \bigcup \mathcal{C}$.

Pick and fix an equatorial square $q_t$ for each tetrahedron $t$ in $\Delta$ that is transverse to the flow $\phi$. Each square $q_t$ divides $t$ into an upper half-tetrahedron and a lower half-tetrahedron. Let $Q$ be the union over all squares. The complementary regions of $Q$ can be obtained by gluing all upper and lower half-tetrahedron along their triangular faces. These complementary regions are known as the \textit{shearing regions}. The decomposition of $\Delta$ into these shearing regions is known as \textit{the shearing decomposition}.

We describe each shearing region in more detail. Each of these is topologically a solid torus with some $l \geq 1$ \textit{upper square faces} and $l$ \textit{lower square faces} on their boundary. We call $l$ the length of the shearing region. The upper square faces meet along edges of a fixed color. We call these edges the \textit{upper helical edges}. The lower square faces also meet along edges of that same fixed color. We call these edges the \textit{lower helical edges}. If the color of these helical edges is blue/red, we say that the shearing region is \textit{blue/red} respectively. We call the union of upper/lower square faces the \textit{upper/lower boundary} of the shearing region respectively. The upper and lower boundary meet along edges of the opposite color as the shearing region. We call these edges the \textit{longitudinal edges}. Finally, the interior of the shearing region contains $2l$ triangular faces, each of which have two blue/red edges along the helical edges, and one red/blue edge along the longitudinal edges, if the shearing region is blue/red respectively.

We illustrate a blue shearing region of length 6 in \Cref{fig:veeringsolidtorus}.

\begin{figure} 
    \centering
    \resizebox{!}{3.5cm}{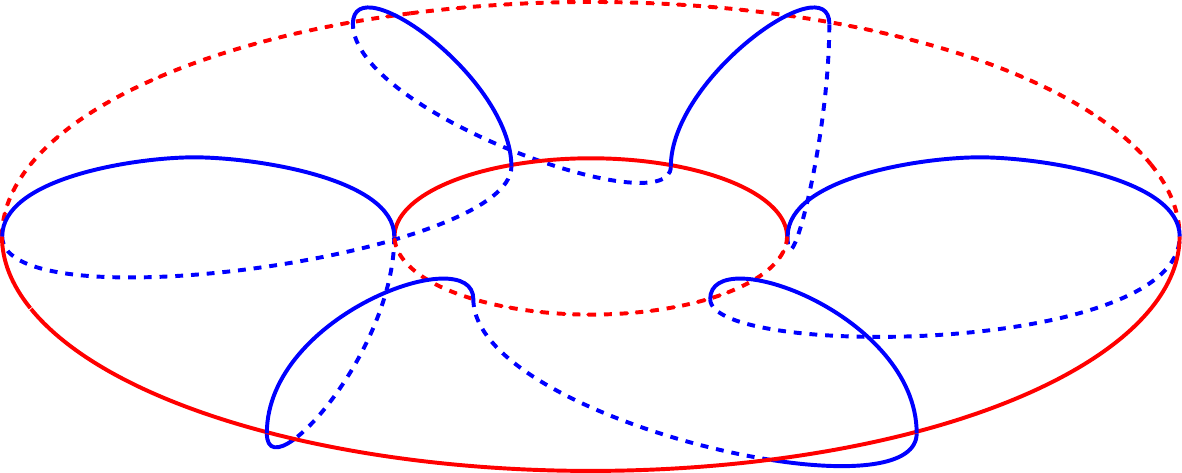}
    \caption{A blue shearing region of length 6.} 
    \label{fig:veeringsolidtorus}
\end{figure}

We denote the upper boundary of a shearing region $U$ by $\partial^U$. The closure of an upper boundary in the closed 3-manifold $M$ is a broken transverse surface with horizontal boundary along the edges of $\Delta$ and vertical boundary along $\mathcal{C}$, at least for generic placements of the edges. We will abuse notation and refer to the closure of an upper boundary by the same name as the upper boundary itself.

Since each shearing region has a product structure induced by the orbits of $\phi$, it is clear that each orbit of $\phi$ intersects $Q$ in finite forward and backward time. In other words, we have the following lemma.

\begin{lemma} \label{lemma:upperannuluscatchesall}
Let $E=\bigcup_U \partial^U$, where the union is taken over all shearing regions $U$. Each orbit of $\phi$ intersects $U$ in finite forward and backward time.
\end{lemma}

Notice that if we orient the boundary edges in the upper boundary of a blue shearing region using the (co)orientation on the squares, then such a collection of red edges is admissible in the sense of \Cref{defn:admissible}. We also have the symmetric fact for red shearing regions. We will make use of both of these facts to show \Cref{thm:closedBirkhoffsection}. 

However, for \Cref{thm:cuspedBirkhoffsection} we want to just deal with edges of one color, so we have to work harder. Let $U$ be a red shearing region. Let $f$ be a triangular face in the interior of $U$, let $e$ be the blue edge of $f$. $e$ is the upper helical edge of some blue shearing region $U'$. Let $q'$ be the upper square face of $U'$ that has an edge on $e$ and lies in the same side of $e$ as $f$. $q'$ must be the lower square face of a red shearing region $U''$. Let $e'$ be the other blue edge of $q'$, and let $f''$ be the triangular face in $U''$ that has its blue edge along $e'$. Then the upper boundary of $U'$ with $q'$ removed but with $f$ and $f''$ added is a surface transverse to orbits of $\phi$ which contains $f$ and whose boundary consists solely of red edges. We call this surface the \textit{croissant} with tip $f$ and denote it by $G_f$. See \Cref{fig:croissant}. 

\begin{figure} 
    \centering
    \resizebox{!}{3.3cm}{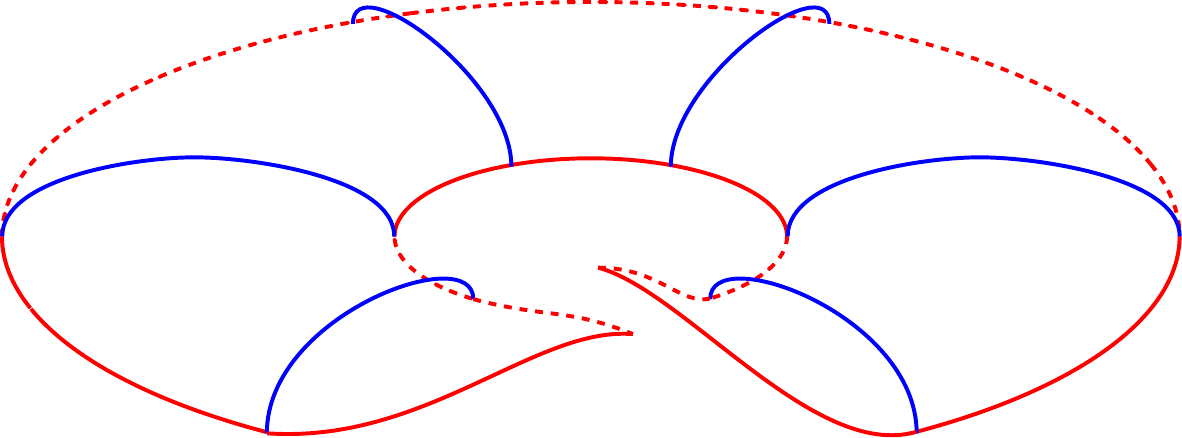}
    \caption{A croissant.} 
    \label{fig:croissant}
\end{figure}

Similarly to the case of upper boundaries, if we orient the edges in the boundary of a croissant using the (co)orientation on the squares and triangles, then such a collection of red edges is admissible in the sense of \Cref{defn:admissible}. Also, the closure of a croissant in the closed 3-manifold $M$ is a broken transverse surface with horizontal boundary along the edges of $\Delta$ and vertical boundary along $\mathcal{C}$, at least for generic placements of the edges. We will abuse notation and refer to the closure of a croissant by the same name as the croissant itself.

\begin{lemma} \label{lemma:croissantscatchall}
Let $E_R=\bigcup_U \partial^U \cup \bigcup_f G_f$, where the first union is taken over all blue shearing regions $U$ and the second union is taken over all triangular faces $f$ in red shearing regions. Each orbit of $\phi$ intersects $E_R$ in finite forward and backward time.
\end{lemma}
\begin{proof}
A square in the upper boundary of a red shearing region can be flowed backwards into the union of two triangular faces in the shearing region. This implies that each orbit of $\phi$ intersects the union of all upper boundaries of blue shearing regions and triangular faces of red shearing regions in finite forward and backward time. Thus the lemma follows from the fact that $G_f$ contains $f$.
\end{proof}

\subsection{Birkhoff sections on the cusped manifold} \label{subsec:cusped}

In this subsection we prove \Cref{thm:cuspedBirkhoffsection}. This is a good warm-up to the proof of \Cref{thm:closedBirkhoffsection}. As we will discuss in \Cref{subsec:vtquestions}, \Cref{thm:cuspedBirkhoffsection} also has some relevance in the theory of veering triangulations. 

Recall the notation for the parameters of a veering triangulation as set up in \Cref{subsec:vt}. We will employ the rough bounds $\delta, \lambda \leq 2N$ and $\nu \leq N$ to obtain bounds that only depend on the number of tetrahedra $N$, as far as the veering triangulation is concerned.

\begin{thm} \label{thm:cuspedBirkhoffsection}
Let $\phi$ be a pseudo-Anosov flow on an oriented closed 3-manifold $M$ without perfect fits relative to a collection of closed orbits $\mathcal{C}$. Let $\Delta$ be the veering triangulation associated to $\phi$ on $M \backslash \bigcup \mathcal{C}$. Suppose $\Delta$ has $N$ tetrahedra. 

Then there exists a closed orbit $\gamma$ of $\phi$ of complexity $\leq 10^6 N^{12}$ and a Birkhoff section $S$ with boundary along $\bigcup \mathcal{C} \cup \gamma$ and with Euler characteristic $\geq - 10^{10} N^{20}$.
\end{thm}
\begin{proof}
Consider the surfaces in $E_R$ in \Cref{lemma:croissantscatchall}. The collection of sides in the horizontal boundary of the surfaces, which we denote by $\mathcal{D}$, is a collection of red oriented edges which is admissible. We bound the size of $\mathcal{D}$. The union $\bigcup_U \partial^U$ involves at most $N$ squares, hence contributes at most $2N$ edges. Each $G_f$ consists of 2 triangles and at most $N-1$ squares, hence contributes at most $2N+2$ edges, and there are at most $2N$ many $f$. Putting all this together, we deduce that $|\mathcal{D}| \leq 2N+2N(2N+2) \leq 10N^2$. 

Now we can apply \Cref{lemma:admissibletoedgeseq} to the negative of $\mathcal{D}$, with $\alpha_T$ chosen to be 1 for all $T$, to get a red edge sequence $\mathcal{E}$ of period $\leq (\frac{\nu}{2} \cdot 10 N^2 + 1 + 1)N + 10 N^2 \leq 17N^4$. Then we apply \Cref{lemma:edgeseqtoedgepath} to $\mathcal{E}$, with $\beta_i$ chosen to be 4 for all $i$, to get a nice red $g$-edge path $(R_i)$. Here $g$ quotients an orbit of $\widehat{\phi}$ to a closed orbit $\gamma$ of complexity $\leq 2(8\lambda+3)^2 \delta^2 (17N^4)^2 \leq 834632 N^{12} =: \mathfrak{c}$. 

Once we have this $\gamma$, we apply \Cref{prop:windingedgepath} to rechoose the triangulation so that $(R_i)$ is winding. Then we apply \Cref{constr:helicoid} to get a broken transverse surface $H(R_i)$ with boundary along $\gamma$ and $\mathcal{E}$. 

Form an immersed Birkhoff section $S'$ by taking the union of $H(R_i)$ and $E_R$. That every orbit intersects $S'$ in finite forward and backward time follows from \Cref{lemma:croissantscatchall}. We bound the complexity of $S'$. The index of $H(R_i)$ is $\geq -\frac{17}{2} N^4$. The surfaces in $E_R$ consist of $N$ squares in the upper boundaries, and $2N(N-1)$ squares and $4N$ triangles in the croissants. Each square has index $-1$ and each triangle has index $-\frac{1}{2}$, so the sum of indices of surfaces in $E_R$ is $\geq -N-2N(N-1)-2N \geq -3N^2$. 

Meanwhile $\gamma$ intersects each square and each triangle at most $2\delta^2 \mathfrak{c}$ times, by \Cref{lemma:complexityrectbound} and \Cref{lemma:complexityfacebound}. So $\gamma$ intersects the surfaces in $E_R$ for $\leq 2N\delta^2 \mathfrak{c} +4N(N-1) \delta^2 \mathfrak{c} + 8N \delta^2 \mathfrak{c} \leq 40N^4 \mathfrak{c}$ times. By \Cref{lemma:helicoidintersection}, $\gamma$ intersects $H(R_i)$ (away from its boundary) for $\leq 32\delta^4 \mathfrak{c} \cdot 17N^4 \leq 8704N^8 \mathfrak{c}$ times. Putting everything together, the complexity of $S'$ is bounded above by $\frac{17}{2}N^3+3N^2 + 40 N^4 \mathfrak{c} + 8704 N^8 \mathfrak{c} \leq 10^{10} N^{20}$. 

Finally, we apply Fried resolution to get a Birkhoff section $S$ as in the statement of the theorem.
\end{proof}

\subsection{Birkhoff sections on the closed manifold} \label{subsec:closed}

In this subsection we finally come to the proof of \Cref{thm:closedBirkhoffsection}.

\begin{thm} \label{thm:closedBirkhoffsection}
Let $\phi$ be a pseudo-Anosov flow on an oriented closed 3-manifold $M$ without perfect fits relative to a collection of closed orbits $\mathcal{C}$. Let $\Delta$ be the veering triangulation associated to $\phi$ on $M \backslash \bigcup \mathcal{C}$. Suppose $\Delta$ has $N$ tetrahedra. For each vertex $T$ of $\Delta$, let $l_T$ be a ladderpole curve and let $t_T$ be a ladderpole transversal at $T$. Let the meridian of $M$ at $T$ be $a_T t_T + b_T l_T$, for $a_T>0$.

Then there exists two closed orbit $\gamma_1$ and $\gamma_2$ of $\phi$, each of complexity $\leq 10^{10} N^{20} (\max a_T + \max |b_T|)^2 (\max a_T)^2$, and a Birkhoff section $S$ with two boundary components, one embedded along $\gamma_1$ and one embedded along $\gamma_2$, with Euler characteristic $\geq - 10^{13} N^{27} (\max a_T + \max |b_T|)^2 (\max a_T)^3$.
\end{thm}
\begin{proof}
Consider the collection $E_L$ of upper boundaries of blue shearing regions. The collection of sides in the horizontal boundary of the annuli in the collection, which we denote by $\mathcal{D}_R$, is a collection of red oriented edges which is admissible. There are at most $N$ squares in the upper boundaries, so there are at most $2N$ elements in $\mathcal{D}_R$. 

Notice that this implies that the number of sides in the vertical boundary of the surfaces is also at most $2N$. Each such side is formed by two squares meeting along a blue edge, hence they each cross a blue ladderpole curve once. Hence we conclude that the number of times the sides in the vertical boundaries cross a fixed blue ladderpole curve at vertex $T$ is some $y_T \in [0,2N]$. Here we choose the fixed blue ladderpole curve at vertex $T$ to be the one at which the minimum in \Cref{lemma:admissibletoedgeseq} is attained. By \Cref{rmk:admissibletoedgeseq}, this is well-defined since we have a fixed admissible collection $\mathcal{D}_R$.

Next we bound the number of times the sides in the vertical boundaries cross the ladderpole transversal at a vertex $T$. Again, each such side is formed by two squares meeting along a blue edge, hence each side crosses a ladderpole transversal for some $x \in [-2,2]$ times. Together, the total number of times the sides cross the ladderpole transversal at $T$ is some $x_T \in [-4N,4N]$ times.

Now we apply \Cref{lemma:admissibletoedgeseq} to the negative of $\mathcal{D}_R$, choosing $\alpha_T = 3N a_T - y_T$, to get a red edge sequence $\mathcal{E}_R$, of period $\leq (\frac{\nu}{2} \cdot 2N + 3N \max a_T + 1)N + 2N \leq 4N^3 + 3N^2 \max a_T \leq 7N^3 \max a_T$. 

Let $B_T = 511N^5 \max a_T + 6N|b_T| + 4N$ for each vertex $T$. Apply \Cref{lemma:edgeseqtoedgepath} to $\mathcal{E}$ to get a nice red $g_R$-edge path $(R_{R,i})$, where we choose the integers $\beta_i$ in the lemma such that for each vertex $T$, all but one of the $\beta_i$ for which $T_i=T$ equals to $(16 \lambda +4)\delta+1$ and the sum of such $\beta_i$ equals $B_T+6N b_T-x_T$. This is possible since $7N^3 \max a_T \cdot ((16 \lambda +4)\delta+1) \leq 511 N^5 \max a_T \leq B_T+6Nb_T-x_T$ and there are at most $7N^3 \max a_T$ indices. Here $g_R$ is primitive, up to increasing one of the $\beta_i$ by 1 (hence increasing $B_T$ by 1) by \Cref{lemma:edgepathprimitive}, and $g_R$ quotients an orbit of $\widehat{\phi}$ to a closed orbit $\gamma_R$ of complexity 
\begin{align*}
    &\leq 2((\max (B_T+6Nb_T-x_T)+1)\lambda +3)^2 \delta^2 (7N^3 \max a_T)^2 \\
    &\leq 2(511 N^5 \max a_T +12N \max |b_T| + 8N+1)\lambda +3)^2 \delta^2 (7N^3 \max a_T)^2 \\
    &\leq 2(1043N^6 \max a_T + 24N^2 \max |b_T|)^2 (2N)^2 (7N^3 \max a_T)^2 \\
    &\leq 426436808 N^{20} (\max a_T + \max |b_T|)^2 (\max a_T)^2 =: \mathfrak{c}
\end{align*}

Symmetrically, we look at the collection $E_R$ of upper boundaries of red shearing regions and let the collection of sides in the horizontal boundary of the annuli in the collection be $\mathcal{D}_L$. Then we apply \Cref{lemma:admissibletoedgeseq} with the analogous choice of $\alpha_T$ to get a blue edge sequence $\mathcal{E}_L$, then \Cref{lemma:edgeseqtoedgepath} to get a nice blue $g_L$-edge path $(R_{L,i})$, but this time we choose the integers $\beta_i$ in the lemma such that for each vertex $T$, all but one of the $\beta_i$ for which $T_i=T$ equals to $(16 \lambda +4)\delta+1$ and the sum of such $\beta_i$ equals $B_T$; such a choice is possible since $7N^3 \max a_T \cdot ((16 \lambda +4)\delta+1) \leq 511 N^5 \max a_T \leq B_T$. Here $g_L$ is primitive, up to increasing one of the $\beta_i$ by 1 (hence increasing $B_T$ by 1) by \Cref{lemma:edgepathprimitive} and $g_L$ quotients an orbit of $\widehat{\phi}$ to a closed orbit $\gamma_R$ of complexity $\leq \mathfrak{c}$ by a similar computation as above.

Once we have $\gamma_R$ and $\gamma_L$, we apply \Cref{prop:windingedgepath} to rechoose the triangulation so that $(R_{R,i})$ and $(R_{L,i})$ are winding. Then we apply \Cref{constr:helicoid} to get broken transverse surface $H(R_{R,i})$ with boundary along $\gamma_R$ and $\mathcal{E}_R$ and broken transverse surface $H(R_{L,i})$ with boundary along $\gamma_L$ and $\mathcal{E}_L$. 

Form a surface $S_R$ transverse to the flow by taking the union of $H(R_{R,i})$ and surfaces in $E_L$, and form a surface $S_L$ transverse to the flow by taking the union of $H(R_{L,i})$ and surfaces in $E_R$. The homology class of $\partial S_R$ on vertex $T$ is $3Na_T t_T+ (B_T+6Nb_T) l_T$, while the homology class of $\partial S_L$ on vertex $T$ is $3Na_T t_T-B_T l_T$, so together they add up to $6N(a_T t_T + b_T l_T)$, which is a multiple of the meridian.

Note that $S_R \cup S_L$ is an immersed Birkhoff section, since every orbit intersects it in finite forward and backward time by \Cref{lemma:upperannuluscatchesall}. We bound the complexity of $S_R \cup S_L$. The index of $H(R_{R,i})$ and $H(R_{L,i})$ are each $\geq -\frac{7}{2}N^3 \max a_T$. The surfaces in $E_L$ and $E_R$ consist of $N$ squares, so the sum of indices of surfaces in $E_L$ and $E_R$ is $-N$.

$\gamma_R$ intersects all of the squares at most $2\delta^2 \mathfrak{c}$ times by \Cref{lemma:complexityrectbound}. So $\gamma_R$ intersects the surfaces in $E_L$ and $E_R$ for $\leq 2\delta^2 \mathfrak{c} \leq 8N^2 \mathfrak{c}$ times. Meanwhile by \Cref{lemma:helicoidintersection}, $\gamma_R$ intersects $H(R_{R,i})$ (away from its boundary) at most $32\delta^4 \mathfrak{c} \cdot 7N^3 \max a_T \leq 3584 N^7 \mathfrak{c} \max a_T$ times and intersects $H(R_{L,i})$ at most $3584 N^7 \mathfrak{c} \max a_T$ times similarly. Symmetric statements hold for $\gamma_L$. Putting everything together, the complexity of $S_R \cup S_L$ is 
\begin{align*}
    &\leq \frac{7}{2}N^3 \max a_T + \frac{7}{2}N^3 \max a_T + N + 2(8N^2 \mathfrak{c} + 3584 N^7 \mathfrak{c} \max a_T + 3584 N^7 \mathfrak{c} \max a_T) \\
    &\leq 8N^3 \max a_T + 14352 N^7 \mathfrak{c} \max a_T \\
    &\leq 10^{13} N^{27} (\max a_T + \max |b_T|)^2 (\max a_T)^3 \\
\end{align*}

Finally, we apply Fried resolution to get a Birkhoff section $S$. Since the homology classes of the boundary components of $S_R \cup S_L$ along each element of $\mathcal{C}$ add up to a multiple of the meridian, after we resolve the turning points in the last step of Fried resolution, the resulting surface $S$ only has boundary components along $\gamma_R$ and $\gamma_L$. 

By the discussion in \Cref{constr:helicoid}, the boundary component of $H(R_{R,i})$ along $\gamma_R$ is embedded and similarly for $H(R_{L,i})$. Thus after applying Fried resolution, $S$ has one boundary component embedded along $\gamma_R$ and one boundary component embedded along $\gamma_L$.

Technically, one has to worry about the possibility that $\gamma_R=\gamma_L=:\gamma$. In that case $S_R \cap S_L$ has two boundary components along $\gamma$ whose homology class adds up to a multiple of the meridian in $M$, hence after Fried resolution $S$ will in fact be a global section for $\phi$, which is even better. But if one still wants to show the statement of the theorem in this case, one can just modify the choice of $\beta_i$ slightly for, say, $\gamma_R$, which will make $[\gamma_R] \neq [\gamma_L]$ in $\pi_1(M)$ hence ensure that they are distinct orbits by \Cref{lemma:nohomotopicorbits}. 
\end{proof}

\section{Discussion and further questions} \label{sec:questions}

\subsection{Birkhoff sections with one boundary component} \label{subsec:oneboundarycomponent}

As remarked in the introduction, one cannot in general find a Birkhoff section with only one boundary component. However, one can ask:

\begin{quest} \label{quest:oneboundarycomponent}
How can one characterize the orbit equivalence class of pseudo-Anosov flows which have a Birkhoff section with only one boundary component?
\end{quest}

Here we make some speculations as to what an answer to this might look like.



Recall that in the case of Anosov flows, one answer to this question has been provided by Marty.

\begin{thm}[{\cite[Theorem G]{Mar21}, \cite[Theorem E]{Mar23}}] \label{thm:Mar}
An Anosov flow admits a Birkhoff section with only one boundary component if and only if it is skew $\mathbb{R}$-covered.
\end{thm}

Motivated by this, one can try to look for a generalization of the `skew $\mathbb{R}$-covered' condition to pseudo-Anosov flows, such that the statement of \Cref{thm:Mar} holds with such a generalization. A natural guess of such a generalization might be the condition that there are only lozenges of a fixed sign, where one defines the sign of a lozenge using the orientation on the orbit space, similar to \Cref{defn:edgerectcolor}.

Another good starting point might be the following weaker version of \Cref{quest:oneboundarycomponent}:

\begin{quest} \label{quest:positiveboundarycomponents}
How can one characterize the orbit equivalence class of pseudo-Anosov flows which have a Birkhoff section with all the boundary components of the same sign?
\end{quest}

We also point out that in the case when $M$ is a rational homology sphere, there is a notion of right-handed flows, introduced by Ghys in \cite{Ghy09}. The basic idea is that orbits of a right-handed flow have positive (asymptotic) linking number. In \cite{Ghy09}, Ghys proves that a right-handed flow admits a Birkhoff section with boundary along any chosen finite collection of closed orbits. In other words, for a fixed 3-manifold $M$ that is a rational homology sphere, if we let:

$\mathcal{B}^+_1$ be the orbit equivalence class of pseudo-Anosov flows on $M$ which have a Birkhoff section with only one positive boundary component,

$\mathcal{B}^+$ be the orbit equivalence class of pseudo-Anosov flows on $M$ which have a Birkhoff section with only positive boundary components,

$\mathcal{B}_{RH}$ be the orbit equivalence class of pseudo-Anosov flows on $M$ which are right-handed.

Then $\mathcal{B}_{RH} \subset \mathcal{B}^+_1 \subset \mathcal{B}^+$. 

\begin{quest} \label{quest:righthandedflows}
Is $\mathcal{B}_{RH} = \mathcal{B}^+$?
\end{quest}






\subsection{More on complexity} \label{subsec:complexityquestions}

The reader will notice that we handled the complexity bounds in this paper very loosely, in the sense that we employ very loose bounds on, for example, the parameters $\delta, \nu, \lambda$. One could conceivably work harder and improve the bounds in the statement of \Cref{thm:cuspedBirkhoffsection} and \Cref{thm:closedBirkhoffsection} by a bit. However, we do not believe that a bound obtained via our methods would be the sharpest possible bound in any case, so we have not bothered to be more tight with our bounding.

In this connection, one can ask:

\begin{quest}
What are the sharpest possible bounds in \Cref{thm:cuspedBirkhoffsection} and \Cref{thm:closedBirkhoffsection}, for both the Euler characteristic of the Birkhoff section and the flow graph complexity of the boundary orbits?
\end{quest}

This is likely a very difficult question. Perhaps a more tractable question would be:

\begin{quest}
What are the smallest possible exponents of $N$ in the bounds in \Cref{thm:cuspedBirkhoffsection} and \Cref{thm:closedBirkhoffsection}?
\end{quest}

The bounds would certainly have to be at least linear, since one can always take covers. The question concerns how complicated the flows that are not just covers of other flows can be as one increases $N$. 

A variation of this question would be to impose a constraint on the Euler characteristic of the Birkhoff section, and ask for the sharpest bound for the flow graph complexity of the boundary orbits among the Birkhoff sections that satisfy the constraint; and vice versa.

We would be remiss not to mention the following well-known conjecture of Fried.

\begin{conj}[Fried] \label{conj:Fried}
Any Anosov flow with orientable stable and unstable foliations admits a genus one Birkhoff section.
\end{conj}

The Birkhoff sections we construct in this paper are in a sense the opposite of those asked for in Fried's conjecture; they have large genus but small number of boundary components. One might be able to use the techniques of this paper to construct Birkhoff sections with relatively small genus but large number of boundary components. In particular, in practice one might be able to find substantially smaller collections of oriented edges when applying \Cref{lemma:admissibletoedgeseq} to get edge sequences, for a given veering triangulation. However, new ideas would probably be required to push the genus of the Birkhoff section in the general case down to $1$, if it is even possible.

On the topic of complexity, one can also attempt to define the complexity of a pseudo-Anosov flow. One way to do this is to ask for the maximum Euler characteristic among all Birkhoff sections to the flow. A potentially more interesting definition would be to ask for the minimum number of tetrahedra in the veering triangulation associated to the flow on $M \backslash \bigcup \mathcal{C}$, as $\mathcal{C}$ varies over all collections of orbits for which the flow has no perfect fits relative to.

For either definition, one can then ask:

\begin{quest} \label{quest:flowcomplexity}
How does the complexity of a flow behaves under operations such as taking covers or performing surgery (e.g. Goodman-Fried surgery, Foulon-Hasselblatt surgery \cite{FH13} or Salmoiraghi's surgeries \cite{Sal21,Sal22})? 
\end{quest}

\subsection{Questions regarding veering triangulations} \label{subsec:vtquestions}

We already mentioned that it might not be very meaningful to try to improve the bounding on complexities done in this paper, but one exception to this might be \Cref{prop:dualgraphflowgraph}. In particular, we are interested in the following:

\begin{quest} \label{quest:dualgraphflowgraphupperbound}
Can the upper bound in \Cref{prop:dualgraphflowgraph} be made linear?
\end{quest}

A positive answer to this question would mean that the dual graph and the flow graph are `equally good' at encoding closed orbits. In any case, attempts to answer this question might also lead to better intuition of the combinatorics of dynamic planes, thus veering triangulations.

Another question that is interesting to the study of veering triangulations is to understand the operation of drilling out orbits on the level of triangulations. More precisely, let $\phi$ be a pseudo-Anosov flow on an oriented closed 3-manifold $M$ without perfect fits relative to a collection of orbits $\mathcal{C}$. Let $\Delta$ be the veering triangulation associated to $\phi$ on $M \backslash \mathcal{C}$. For any collection $\mathcal{C}' \supset \mathcal{C}$, $\phi$ has no perfect fits relative to $\mathcal{C}'$ as well. Let $\Delta'$ be the veering triangulation associated to $\phi$ on $M \backslash \mathcal{C}'$. We say that $\Delta'$ is obtained from $\Delta$ by \textit{drilling out the orbits in $\mathcal{C}' \backslash \mathcal{C}$}.

It is known that circular pseudo-Anosov flows give rise to \textit{layered veering triangulations}. See for example \cite[Theorem 5.15]{LMT20} for an explanation of this. In this language, \Cref{thm:cuspedBirkhoffsection} may be restated as saying that any veering triangulation can be drilled along a single orbit to give a layered veering triangulation.

Layered veering triangulations are generally better understood than non-layered ones. This motivates one to understand how the combinatorics of a veering triangulation change under drilling, in the hopes that one can transfer facts from layered triangulations to general triangulations via \Cref{thm:cuspedBirkhoffsection}.

\begin{quest} \label{quest:drillingdescription}
Can one describe the triangulation $\Delta'$ in terms of $\Delta$ and, say, flow graph cycles that encode each orbit in $\mathcal{C}' \backslash \mathcal{C}$?
\end{quest}

\begin{quest} \label{quest:drillingflowgraphcomplexity}
Can one bound the change in flow graph complexity of an orbit $\gamma$ of $\phi$ which is not an element of $\mathcal{C}'$ when measured relative to $\mathcal{C}$ and to $\mathcal{C}'$?
\end{quest}

Finally, we relay a question that is asked to us by Schleimer and Segerman:

\begin{quest}[Schleimer, Segerman] \label{quest:drillfortoggle}
For any veering triangulation, is it always possible to drill out some collection of orbits to get a veering triangulation with only toggle tetrahedra?
\end{quest}

A positive answer to this question would have significance towards the enumeration of veering triangulations, at least on a conceptual level.

\bibliographystyle{alphaurl}

\bibliography{bib.bib}

\end{document}